\newif\iffinal
\finalfalse	

\documentclass[letterpaper,11pt,reqno]{amsart} 
\RequirePackage[utf8]{inputenc}
\usepackage[portrait,margin=3cm]{geometry}
\usepackage{color}              
\usepackage[usenames,dvipsnames]{xcolor}

\usepackage{mathrsfs,soul}
\usepackage{hyperref}
\usepackage{amssymb,amsthm,amsmath,amsfonts,amsbsy,latexsym,dsfont}
\usepackage[textsize=small]{todonotes}       
\usepackage{graphicx}
\usepackage[numeric,initials,nobysame]{amsrefs}
\usepackage{upref,setspace}

\definecolor{BrickRed}{rgb}{0.65,0.08,0}

\usepackage{enumerate}

\numberwithin{equation}{section}
\numberwithin{figure}{section}
\numberwithin{table}{section}

\newtheorem{Lemma}{Lemma}[section]

\newtheorem{Theorem}{Theorem}[section]

\newtheorem{Condition}{Condition}[section]

\newtheorem{Remark}{Remark}[section]

\newcommand{\Erdos}{Erd\H{o}s-R\'enyi }

\newcommand{\half}{{\frac{1}{2}}}

\newcommand{\lan}{\langle}
\newcommand{\ran}{\rangle}
\newcommand{\lfl}{\lfloor}
\newcommand{\rfl}{\rfloor}

\newcommand{\Bmb}{{\mathbb{B}}}
\newcommand{\Cmb}{{\mathbb{C}}}
\newcommand{\Dmb}{{\mathbb{D}}}
\newcommand{\Emb}{{\mathbb{E}}}

\newcommand{\Hmb}{{\mathbb{H}}}

\newcommand{\Nmb}{{\mathbb{N}}}

\newcommand{\Pmb}{{\mathbb{P}}}
\newcommand{\Qmb}{{\mathbb{Q}}}
\newcommand{\Rmb}{{\mathbb{R}}}
\newcommand{\Smb}{{\mathbb{S}}}

\newcommand{\Xmb}{{\mathbb{X}}}

\newcommand{\Zmb}{{\mathbb{Z}}}

\newcommand{\Amc}{{\mathcal{A}}}
\newcommand{\Bmc}{{\mathcal{B}}}

\newcommand{\Fmc}{{\mathcal{F}}}
\newcommand{\Gmc}{{\mathcal{G}}}
\newcommand{\Hmc}{{\mathcal{H}}}

\newcommand{\Lmc}{{\mathcal{L}}}
\newcommand{\Mmc}{{\mathcal{M}}}

\newcommand{\Pmc}{{\mathcal{P}}}

\newcommand{\Smc}{{\mathcal{S}}}
\newcommand{\Tmc}{{\mathcal{T}}}
\newcommand{\Umc}{{\mathcal{U}}}

\newcommand{\Ebf}{{\mathbf{E}}}

\newcommand{\one}{{\boldsymbol{1}}}

\newcommand{\etabar}{{\bar{\eta}}}

\newcommand{\Fmcbar}{{\bar{\Fmc}}}

\newcommand{\gammabar}{{\bar{\gamma}}}

\newcommand{\mubar}{{\bar{\mu}}}

\newcommand{\nubar}{{\bar{\nu}}}

\newcommand{\omegabar}{{\bar{\omega}}}

\newcommand{\thetabar}{{\bar{\theta}}}

\newcommand{\Fmctil}{{\widetilde{\Fmc}}}
\newcommand{\gtil}{{\widetilde{g}}}

\newcommand{\Gammatil}{{\widetilde{\Gamma}}}

\newcommand{\itil}{{\widetilde{i}}}
\newcommand{\jtil}{{\widetilde{j}}}\newcommand{\Jtil}{{\widetilde{J}}}

\newcommand{\ktil}{{\widetilde{k}}}

\newcommand{\Ntil}{{\widetilde{N}}}
\newcommand{\nutil}{{\widetilde{\nu}}}

\newcommand{\Omegatil}{{\widetilde{\Omega}}}

\newcommand{\Pmbtil}{{\widetilde{\Pmb}}}

\newcommand{\xtil}{{\widetilde{x}}}\newcommand{\Xtil}{{\widetilde{X}}}
\newcommand{\xitil}{{\widetilde{\xi}}}

\newcommand{\Ytil}{{\widetilde{Y}}}

\newcommand{\Ncompensated}{\Ntil}

\newcommand{\measurea}{\rho}
\newcommand{\dBL}{\textnormal{d}_{\textnormal{BL}}}

\numberwithin{equation}{section}

\begin{document}
	
	\title[Mean field interaction on dynamic random graphs]{Mean field interaction on random graphs with dynamically changing multi-color edges}

	\date{\today}
	\subjclass[2010]{
	60F05 
	60K35 
	60J75 
	05C80 
	60G09 
	60J27 
	60K37}
	\keywords{Dynamical random graphs; Mean field interaction; Propagation of chaos; Central limit theorems; Endogenous common noise; Exchangeability; Interacting particle systems}
	\author[Bayraktar]{Erhan Bayraktar}
		\thanks{E.\ Bayraktar is supported in part by the National Science Foundation under DMS-2106556, and in part by the Susan M. Smith Professorship.} 
		\address{Department of Mathematics, University of Michigan, 530 Church Street, Ann Arbor, MI 48109} 
	\author[Wu]{Ruoyu Wu}
		\address{Department of Mathematics, Iowa State University, 411 Morrill Road, Ames, IA 50011} 
		\email{erhan@umich.edu, ruoyu@iastate.edu}
	 
	 \begin{abstract}
	 We consider weakly interacting jump processes on time-varying random graphs with dynamically changing multi-color edges. 
	 The system consists of a large number of nodes in which the node dynamics depends on the joint empirical distribution of all the other nodes and the edges connected to it, while the edge dynamics depends only on the corresponding nodes it connects. 
	 Asymptotic results, including law of large numbers, propagation of chaos, and central limit theorems, are established. 
	 In contrast to the classic McKean-Vlasov limit, the limiting system exhibits a path-dependent feature in that the evolution of a given particle depends on its own conditional distribution given its past trajectory.
	 We also analyze the asymptotic behavior of the system when the edge dynamics is accelerated.
	 A law of large number and a propagation of chaos result is established, and the limiting system is given as independent McKean-Vlasov processes.
	 Error between the two limiting systems, with and without acceleration in edge dynamics, is also analyzed.
	 \end{abstract}
	 
	 \maketitle

\tableofcontents

\section{Introduction}

In this work we study some asymptotic results for large particle systems with mean field interactions on time varying random graphs.
The model is described in terms of two collections of countable-state pure jump processes, one that gives the evolution of (the states of) nodes in the system, and the other that drives the evolution of (the colors of) edges which govern the interaction between nodes in the system.
We consider mean field interaction between nodes, in that the node dynamics depends on the joint empirical distribution of all the other nodes and the edges connected to it.
The edge dynamics on the other hand depends only on the nodes it connects.
More precisely,
\begin{align}
	X_i^n(t) & = X_i(0) + \int_{[0,t]\times\Zmb\times\Rmb_+} y\one_{[0,\Gamma(y,X_i^n(s-),\nu^n_i(s-))]}(z) \, N_i(ds\,dy\,dz), \notag \\
	\xi_{ij}^n(t) & = \xi_{ij}(0) + \int_{[0,t]\times\Zmb\times\Rmb_+} y\one_{[0,\beta(n)\Gammatil(y,\xi_{ij}^n(s-),X_i^n(s-),X_j^n(s-))]}(z) \, N_{ij}(ds\,dy\,dz), \label{eq:system} \\
	\nu^n_i(t) & = \frac{1}{n} \sum_{j=1}^n \delta_{(X_j^n(t),\xi_{ij}^n(t))}, \quad i,j=1,\dotsc,n, \notag 
\end{align}
where $\{X_i(0)\}$ are independent and identically distributed (i.i.d.) $\Zmb$-valued random variables with some probability distribution $\mu(0)$, $\{\xi_{ij}(0)\}$ are i.i.d.\ $\Zmb$-valued random variables with some probability distribution $\theta(0)$, $\{N_i\}$ and $\{N_{ij}\}$ are i.i.d.\ Poisson random measures (whose precise definition will be introduced in Section \ref{sec:notation}) with intensity $ds \times \measurea(dy) \times dz$ for some finite measure $\measurea$ on $\Zmb$, $\Gamma$ and $\Gammatil$ are functions governing the jump rates with $\Gamma(y,x,\nu) = \int_{\Zmb^2} \gamma(y,x,\xtil,\xitil) \,\nu(d\xtil\,d\xitil)$ for some measurable function $\gamma$.
Here $X_i^n$ denotes the state of node $i$, $\xi_{ij}^n$ describes the color of the edge between nodes $i$ and $j$, and $\beta(n) \ge 0$ is a sequence of real numbers representing the scale of jump rates of edges. 

A typical example where such a system arises is in the study of gossip algorithms {(see e.g. \cite[Section 1]{shah2009gossip} for a basic setup and \cite{LavaeiMurray2009quantized} for a setup with weighted graph)}, where $\xi_{ij}^n$ denotes whether there is an edge between nodes $i,j$ in a graph with $n$ nodes, and $X_i^n$ denotes whether certain information has spread to node $i$.
{Systems with state-dependent edge evolution may also arise in the modeling of brain networks \cite{BetzelBassett2017multi,KhambhatiSizemoreBetzelBassett2018modeling,BaladronFaugeras2012,Touboul2014propagation} and other biological phenomena \cite{BarreDegondPeurichardZatorska2020modelling,BarreDegondZatorska2017kinetic}. For example, in neuroscience, $X_i^n$ represents the state of each neuron and $\xi_{ij}^n$ describes the dynamical connection between neurons (see e.g.\ \cite{BaladronFaugeras2012,Touboul2014propagation} for a diffusion setup with static graphs).}
{Such a system may also be used for analysis in epidemiology \cite{GrossDLimaBlasius2006epidemic,MarceauNoelHebertAllardDube2010adaptive}.
For example, in the analysis of individual-based SIR model (see e.g.\ \cite{RochaMasuda2016individual} for a basic setup), $X_i^n$ represents the status of an individual (location, wearing a mask or not, vaccinated or not, susceptible, infectious, or recovered, etc.), while $\xi_{ij}^n$ denotes the type of interactions between individuals $i$ and $j$ (such as shaking hands, social distanced, and no interactions). In such a setup, the dynamics of $\xi_{ij}^n$ depends on the status of individuals $i$ and $j$ and its evolution may be much faster than that of $X_i^n$.}
In simpler terms, one may also view the system as $n$ children playing at $M$ places with $K$ types friendship between each pair of children, where $M$ and $K$ could be infinity.
The node $X_i^n(t) \in \{1,\dotsc,M\}$ denotes the place at which the $i$-th child is, while the edge $\xi_{ij}^n(t) \in \{1,\dotsc,K\}$ denotes the type of friendship in which the $i$-th child views the $j$-th child at time $t$.
The jump rate of $X_i^n$ depends on the empirical distribution of all children's positions and their friendship from the viewpoint of the $i$-th child, that is, $\nu^n_i$.

When there is only one possible color, i.e.\ the graph is simply a complete graph with $\xi_{ij}^n \equiv 1$, the model reduces to the classic mean-field system, the study of which dates back to works of Boltzmann, Vlasov, McKean and others (see \cite{Sznitman1991,Kolokoltsov2010} and references therein).
The original motivation for the study of mean-field systems came from statistical physics but in recent years similar models have arisen in many different application areas, ranging from economics and chemical and biological systems to communication networks and social sciences (see e.g.\ \cite{BudhirajaDupuisFischerRamanan2015limits} for an extensive list of references).
The asymptotic picture is well resolved and many different results have been established, including laws of large numbers (LLN), propagation of chaos (POC) properties, and central limit theorems (CLT), see e.g.\ \cite{BhamidiBudhirajaWu2019weakly} and the references therein.

When there are two possible colors (denoted by $0$ and $1$ for example) and edges are independently drawn and fixed at time $0$, i.e.\ the edges could be present or absent and form an \Erdos random graph, the model has recently drawn much attention.
It has been shown that the same LLN, CLT, and large deviation principles (LDP), as in the mean-field case, hold under certain conditions.
In particular for interacting diffusions, quenched and annealed LLN are studied in \cite{Delattre2016}, CLT is established in \cite{BhamidiBudhirajaWu2019weakly}, and LDP is obtained in \cite{CoppiniDietertGiacomin2019law,OliveiraReis2019interacting}.
For certain pure jump processes arising from the study of large-scale queuing networks, LLN is studied in \cite{BudhirajaMukherjeeWu2019supermarket}.
Moreover, mean field games on \Erdos random graphs are analyzed in \cite{Delarue2017mean}, and graphon mean field games on static graphs with possibly uncountable players have recently been studied (see e.g.\ \cite{CainesHuang2018graphon,PariseOzdaglar2019graphon,CarmonaCooneyGravesLauriere2019stochastic}).
{We note that most of these works are focusing on the dense graph regime, and the graph in \eqref{eq:system} is dense when focusing on the network with a particular edge color that has a non-zero probability.
After this work was posted, a diffusive system with accelerated state-dependent edge dynamics (two colors) in the sparse graph regime is studied in \cite{BarreDobsonOttobreZatorska2021fast} and uniform in time averaging results are obtained (see also references therein for other areas where the system \eqref{eq:system} may arise).}

The goal of the current work is to study asymptotic behaviors of the system \eqref{eq:system} as $n \to \infty$ and $\beta(n) \to \beta \in [0,\infty]$.
Our first main result is LLN, POC, and CLT (Theorems \ref{thm:no_acceleration_2}, \ref{thm:CLT_nu} and \ref{thm:CLT_mu}) for node states and their empirical measures when $\beta<\infty$.
The proof of LLN and POC relies on certain coupling arguments, exchangeability properties of the nodes and edges, and a key conditional independence structure for the limiting system (Theorem \ref{thm:no_acceleration_1}).
Intuitively speaking, due to the state-dependent evolution of edges, one would not expect in the limiting system that edges are independent, although nodes are i.i.d.
Indeed, Theorem \ref{thm:no_acceleration_1} states that conditioning on \textit{the path} of the state of a node, edges connected with this node are i.i.d.
The proof of such a statement follows from a careful time-discretization argument along with applying de Finetti's theorem.
Moreover, we generalize the special case $\beta(n) \to 0$, in which the limit represents a random but a static graph, to a case in which the edge processes $\xi_{ij}^n$ are i.i.d.\ adapted and could be non-Markovian.
LLN and POC for such systems are obtained in Theorem \ref{thm:iid_LLN}.
The CLT result characterizes the fluctuation of functionals of the empirical measures of nodes and edges.
The proof of CLT relies on a change of measure technique using Girsanov's theorem, and this approach goes back to \cite{Sznitman1984,ShigaTanaka1985}.
This technique reduces the problem to working with the limiting system where nodes are i.i.d.\ and edges are conditionally independent, while the price to pay is that one must carefully analyze the asymptotic behavior of the Radon–Nikodym derivative.
The presence of conditionally independent edges {with state-dependent dynamics} requires more challenging work than what has been done in the single-color case (e.g.\ in \cite{Sznitman1984,ShigaTanaka1985}) and the two-color case (in \cite{BhamidiBudhirajaWu2019weakly}).
In particular, {unlike in \cite{BhamidiBudhirajaWu2019weakly} where edges are evolving independently (as in Section \ref{sec:iid}) and the limit of fluctuations of node states is Gaussian, here} the node plays the role of common noises in the analysis (see Lemma \ref{lem:Un_joint}) and as a result the CLT limit is not a Gaussian random variable but rather a Gaussian mixture.

Our second main result is the study of the averaging principle of the system \eqref{eq:system} when $\beta(n) \to \infty$.
Systems of stochastic processes with fast components or regime-switching features have a long history of applications and the averaging principle has been well studied, when there is \textit{one} fast component (see e.g.\ \cite{YinZhu2010hybrid,BudhirajaDupuisGanguly2018large,NguyenYinHoang2019laws}).
However, a collection of fast state-dependent switching edges are present in the system considered here, and more careful analysis is needed.
In particular in the limiting system, the jump rate of the slow component $X_i$ corresponding to node $i$ depends on its own probability distribution and the conditional invariant measure of the fast component given slow components (see \eqref{eq:acceleration_system} for the precise form).
LLN and POC for \eqref{eq:system} when $\beta(n)\to\infty$ are obtained in Theorem \ref{thm:acc_LLN}.
Compared to the limiting system in the regime $\beta(n) \to \beta < \infty$, this one does not suffer from the path-dependent conditional independence subtlety and serves as a nice approximation of the former, with the approximation error analyzed in Theorem \ref{thm:comparison}.

\subsection{Organization}
The paper is organized as follows.
In Section \ref{sec:no-acceleration} we analyze the system \eqref{eq:system} when $\beta(n) \to \beta \in [0,\infty)$.
A basic condition (Condition \ref{cond:no_acceleration}) is stated, under which the limiting system \eqref{eq:no_acceleration_system} has a unique solution and a certain (conditional) independence property (Theorem \ref{thm:no_acceleration_1}).
A law of large numbers and a propagation of chaos property are obtained in Theorem \ref{thm:no_acceleration_2}. 
In Section \ref{sec:iid}, we also present a LLN and POC (Theorem \ref{thm:iid_LLN}) for a system with i.i.d.\ and possibly non-Markovian edge processes $\xi_{ij}^n$.
In Section \ref{sec:acceleration}, the system \eqref{eq:system} with accelerated edge dynamics, namely when $\beta(n) \to \infty$, is studied.
LLN and POC are obtained in Theorem \ref{thm:acc_LLN}, and the approximation error, as $\beta \to \infty$, between the corresponding limiting system and \eqref{eq:no_acceleration_system} are characterized in Theorem \ref{thm:comparison}.
The convenience of this limiting system is illustrated in Section \ref{sec:eg-acceleration}, by characterizing the evolution of marginal distributions as Riccati equations.
In Section \ref{sec:CLT} we present a CLT (Theorem \ref{thm:CLT_nu}) for the fluctuation of functionals of the empirical measures of nodes and edges connecting to a given node.
As noted above, the limit is not a Gaussian random variable but rather a Gaussian mixture.
We also provide a CLT (Theorem \ref{thm:CLT_mu}) for the fluctuation of functionals of the empirical measures of nodes.
The limit is given by a simpler form and this point is illustrated through an example in Section \ref{sec:eg-CLT} where the variance of the limit Gaussian random variable has an explicit form.
Proofs of all LLN and POC are given in Section \ref{sec:pf-LLN}.
Finally Section \ref{sec:pf-CLT} contains proofs of Theorems \ref{thm:CLT_nu} and \ref{thm:CLT_mu}.

\subsection{Notation}
\label{sec:notation}
Given a Polish space $\Smb$, denote by $\Bmc(\Smb)$ the Borel $\sigma$-field. 
Let $\Pmc(\Smb)$ be the space of probability measures on $\Smb$ endowed with the topology of weak convergence. 
A convenient metric for this topology is the bounded-Lipschitz metric $\dBL$, defined by
$$\dBL(\nu_1,\nu_2) = \sup_{f \in \Bmb_1} \left| \lan f, \nu_1-\nu_2 \ran \right|, \quad \nu_1, \nu_2 \in \Pmc(\Smb),$$
where $\Bmb_1$ is the collection of all Lipschitz functions $f$ that are bounded by $1$ and such that the corresponding Lipschitz constant is bounded by $1$ as well; and $\lan f, \nu \ran := \int f \, d\nu$ for a signed measure $\nu$ on $\Smb$ and $\nu$-integrable $f \colon \Smb \to \Rmb$.
Given a collection of random probability measures $\nu^n$, $\nu$ on $\Smb$, we write $\nu^n \to \nu$ in $\Pmc(\Smb)$ in probability, if $d(\nu^n,\nu) \to 0$ in probability as $n \to \infty$, for any metric $d$ on $\Pmc(\Smb)$ that metrizes the weak convergence topology.
We say a collection $\{X_n\}$ of $\Smb$-valued random variables is tight if {the collection of their distributions} are tight in $\Pmc(\Smb)$.
We use the symbol `$\Rightarrow$' to denote convergence in distribution and `$\stackrel{d}{=}$' to denote the equality in distribution.
The probability law of a random variable $X$ will be denoted by $\Lmc(X)$.
For a measurable function $f \colon \Smb \to \Rmb$, let $\|f\|_\infty:= \sup_{x \in \Smb} |f(x)|$. 
Fix $T \in (0,\infty)$.
Denote by $\Cmb([0,T]:\Smb)$ (resp.\ $\Dmb([0,T]:\Smb)$) the space of continuous functions (resp.\ right continuous functions with left limits) from $[0, T]$ to $\Smb$, endowed with the uniform topology (resp.\ Skorokhod topology).
We will use the notations $X(t)$ and $X_t$ interchangeably for stochastic processes.
For $x \in \Dmb([0,T]:\Smb)$, let $\|x\|_{*,t} := \sup_{0 \le s \le t} \|x(s)\|$, $x[t] := (x(s):0 \le s \le t)$, and $x[t-] := (x(s):0 \le s < t)$.
{For a Banach space $\Hmc$, denote the norm (and inner product if $\Hmc$ is a Hilbert space) in $\Hmc$ by $\|\cdot\|_\Hmc$ (and $\lan \cdot, \cdot \ran_\Hmc$).}
For a $\sigma$-finite measure $\nu$ on a Polish space $\Smb$, denote by {$L^2(\Smb,\nu)$ (resp.\ $L^1(\Smb,\nu)$) the space of $\nu$-square integrable (resp.\ $\nu$-integrable)} functions from $\Smb$ to $\Rmb$.
We denote by $\Smb^\infty$ the countable product space of copies of $\Smb$, equipped with the usual product topology.

Let $[k] := \{1,\dotsc,k\}$ for each $k \in \Nmb$.
We will use $\kappa$, $\kappa_0$, $\kappa_1$, $\dotsc$ for constants in the proofs, whose value may change over lines.
Let $\Xmb_t = [0,t]\times\Zmb\times\Rmb_+$ for each $t \in [0,T]$, and let $\Mmc_t$ be the space of $\sigma$-finite measures on $(\Xmb_t,\Bmc(\Xmb_t))$ with the topology of vague convergence.
A Poisson random measure (PRM) $N$ on $\Xmb_T$ with intensity measure $\nu \in \Mmc_T$ is an $\Mmc_T$-valued random variable such that for each $A \in \Bmc(\Xmb_T)$ with $\nu(A) < \infty$, $N(A)$ is Poisson distributed with mean $\nu(A)$ and for disjoint $A_1,\dotsc,A_k \in \Bmc(\Xmb_T)$, $N(A_1),\dotsc,N(A_k)$ are mutually independent random variables (cf.\ \cite{IkedaWatanabe1990SDE}).

Let $(\Omega,\Fmc,\Pmb,\{\Fmc_t\})$ be a filtered probability space on which we are given i.i.d.\ PRM $\{N_i,N_{ij} : i,j \in \Nmb\}$ on $\Xmb_T$ with intensity measure $ds \times \measurea(dy) \times dz$ for some finite measure $\measurea$ on $\Zmb$.
Expectations under $\Pmb$ (resp.\ some other probability measure $\Qmb$) will be denoted by $\Emb$ (resp.\ $\Emb_\Qmb$).

\section{Systems with node-dependent edge dynamics}
\label{sec:no-acceleration}

In this section we study the system \eqref{eq:system} when $\beta(n) \to \beta \in [0,\infty)$.
Recall that $\Gamma(y,x,\nu) = \int_{\Zmb^2} \gamma(y,x,\xtil,\xitil) \,\nu(d\xtil\,d\xitil)$ for $y,x \in \Zmb$ and $\nu \in \Pmc(\Zmb^2)$, where $\gamma$ is some measurable function from $\Zmb^4$ to $\Rmb$.
We make the following assumptions on $\gamma$ and $\Gammatil$.

\begin{Condition}
	\phantomsection
	\label{cond:no_acceleration}
	\begin{enumerate}[(i)]
	\item There exists $\{\gamma_y \in [0,\infty) : y \in \Zmb\}$ such that
	$0 \le \gamma(y,x,\xtil,\xitil) \le \gamma_y$ and $0 \le \Gammatil(y,\xitil,x,\xtil) \le \gamma_y$ for all $y,x,\xtil,\xitil \in \Zmb$ and $C_\gamma := \int_\Zmb |y|\gamma_y\,\measurea(dy) < \infty$.
	\item $\{X_i(0)\}$ are i.i.d.\ with some common probability distribution $\mu(0) \in \Pmc(\Zmb)$ and $\Emb|X_i(0)|<\infty$. $\{\xi_{ij}(0)\}$ are i.i.d.\ with some common probability distribution $\theta(0) \in \Pmc(\Zmb)$ and $\Emb|\xi_{ij}(0)|<\infty$.
	\end{enumerate}
\end{Condition}

\begin{Remark}
	\phantomsection
	\label{rmk:Lipschitz}
	\begin{enumerate}[(a)]
	\item 
		Condition \ref{cond:no_acceleration}(i) holds clearly if the system is finite state, such as $\gamma(y,x,\xtil,\xitil)=0$ whenever $|y+x| > r$, for some $r \in \Nmb$.
	\item 
		Noting that every bounded function on $\Zmb^d$ is automatically Lipschitz, we see that {the function $(x,\xtil,\xitil) \mapsto \gamma(y,x,\xtil,\xitil)$ is $\gamma_y$-Lipschitz by Condition \ref{cond:no_acceleration}(i). In fact,
		\begin{equation}
			\label{eq:gamma-Lipschitz}
			|\gamma(y,x_1,x_2,\xi)-\gamma(y,x_1',x_2',\xi')| \le \gamma_y \one_{\{(x_1,x_2,\xi) \ne (x_1',x_2',\xi')\}}, \quad \forall\, y,x_1,x_2,\xi,x_1',x_2',\xi' \in \Zmb.
		\end{equation}}
	\end{enumerate}
\end{Remark}

The next two theorems show that, under Condition \ref{cond:no_acceleration}, the limiting system is given by the pathwise unique solution to the following equations:
\begin{align}
	X_i(t) & = X_i(0) + \int_{\Xmb_t} y\one_{[0,\Gamma(y,X_i(s-),\nu_i(s-))]}(z) \, N_i(ds\,dy\,dz), \notag \\
	\xi_{ij}(t) & = \xi_{ij}(0) + \int_{\Xmb_t} y\one_{[0,\beta\Gammatil(y,\xi_{ij}(s-),X_i(s-),X_j(s-))]}(z) \, N_{ij}(ds\,dy\,dz), \label{eq:no_acceleration_system} \\
	\nu_i(t) & = \lim_{n \to \infty} \frac{1}{n} \sum_{j=1}^n \delta_{(X_j(t),\xi_{ij}(t))}, \notag
\end{align}
where the limit in $\nu_i(t)$ is understood as almost surely in $\Pmc(\Zmb^2)$ for each $i \in \Nmb$ and $t \in [0,T]$.
Note that the PRMs are the same as those in \eqref{eq:system}, for ease of deriving Theorem \ref{thm:no_acceleration_2} below.

The proofs of Theorems \ref{thm:no_acceleration_1} and \ref{thm:no_acceleration_2} are given in Section \ref{sec:pf-no-acc}.

\begin{Theorem}
	\phantomsection
	\label{thm:no_acceleration_1}
	
	\begin{enumerate}[(a)]
	\item 
		Suppose Condition \ref{cond:no_acceleration}(ii) holds. 
		Also suppose that there exists $\{\gamma_y \in [0,\infty) : y \in \Zmb\}$ such that
		\begin{equation}
			\label{eq:weaker-assumption}
			0 \le \gamma(y,x,\xtil,\xitil) \le \gamma_y(1+|x|), \: 0 \le \Gammatil(y,\xitil,x,\xtil) \le \gamma_y(1+|\xitil|+|x|+|\xtil|)
		\end{equation}
		for all $y,x,\xtil,\xitil \in \Zmb$, and $C_\gamma := \int_\Zmb |y|\gamma_y\,\measurea(dy) < \infty$.
		Then the system \eqref{eq:no_acceleration_system} has a weak solution.
	\item
		Suppose Condition \ref{cond:no_acceleration} holds.
		Then the system \eqref{eq:no_acceleration_system} has a unique pathwise solution.
		Moreover, $X_i$ are i.i.d., $\{(X_j[t],\xi_{ij}[t]) : j \in \Nmb, j \ne i\}$ are i.i.d.\ conditioning on $X_i[t]$, and $\nu_i(t) = \Lmc((X_j(t),\xi_{ij}(t)) \,|\, X_i[t]) = \Phi_t(X_i[t])$ for each $j \ne i$, where $\Phi_t \colon \Dmb([0,t]:\Zmb) \to \Pmc(\Zmb^2)$ is some measurable map independent of the choice of $i$.
	\end{enumerate}	
\end{Theorem}

\begin{Remark}
	Using Theorem \ref{thm:no_acceleration_1}, the system \eqref{eq:no_acceleration_system} could be rewritten in the following equivalent and perhaps more familiar form:
	\begin{align*}
		X_i(t) & = X_i(0) + \int_{\Xmb_t} y\one_{[0,\Gamma(y,X_i(s-),\nu_i(s-))]}(z) \, N_i(ds\,dy\,dz), \\
		\xi_{ij}(t) & = \xi_{ij}(0) + \int_{\Xmb_t} y\one_{[0,\beta\Gammatil(y,\xi_{ij}(s-),X_i(s-),X_j(s-))]}(z) \, N_{ij}(ds\,dy\,dz), \\
		\nu_i(t) & = \Lmc((X_k(t),\xi_{ik}(t)) \,|\, X_i[t]) = \Phi_t(X_i[t]), \quad \forall k \ne i.
	\end{align*}
\end{Remark}

For $i \in [n]$ and $t \in [0,T]$, {define the following trajectory and marginal empirical measures:}
\begin{equation}
	\label{eq:nu_i_n}
	\nu^n_i := \frac{1}{n} \sum_{j=1}^n \delta_{(X_j^n(\cdot),\xi_{ij}^n(\cdot))}, \quad \mu^n := \frac{1}{n} \sum_{i=1}^n \delta_{X_i^n(\cdot)}, \quad \mu^n(t) := \frac{1}{n} \sum_{i=1}^n \delta_{X_i^n(t)}.
\end{equation}

\begin{Theorem}
	\label{thm:no_acceleration_2}
	Suppose Condition \ref{cond:no_acceleration} holds.
	Then 
	\begin{enumerate}[(a)]
	\item 
		There exists some $\kappa =\kappa(T,\beta) < \infty$ such that
		\begin{equation}
			\max_{i \in [n]} \Emb \|X_i^n-X_i\|_{*,T} + \max_{i,j \in [n]} \Emb \|\xi_{ij}^n-\xi_{ij}\|_{*,T} \le \frac{\kappa}{\sqrt{n}} + \kappa |\beta(n)-\beta|. \label{eq:no_acc_LLN1}
		\end{equation}
	\item (POC) For any $k \in \Nmb$, as $n \to \infty$,
	\begin{equation}
		\label{eq:no_acceleration_POC}
		\Lmc(X_1^n,\dotsc,X_k^n) \to \mu^{\otimes k}, \quad \Lmc(X_1^n(t),\dotsc,X_k^n(t)) \to [\mu(t)]^{\otimes k} \text{ for each } t \in [0,T],
	\end{equation}	
	where $\mu := \Lmc(X_i) \in \Pmc(\Dmb([0,T]:\Zmb))$ and $\mu(t) := \Lmc(X_i(t)) \in \Pmc(\Zmb)$ for $i \in \Nmb$.
	\item
		(LLN) As $n \to \infty$,
		\begin{align}
			\nu^n_i & \to \nu_i := \lim_{n \to \infty} \frac{1}{n} \sum_{j=1}^n \delta_{(X_j(\cdot),\xi_{ij}(\cdot))} \text{ in $\Pmc(\Dmb([0,T]:\Zmb^2))$ in probability}, \label{eq:no_acc_LLN_nu_i} \\
			\nu^n_i(t) & \to \nu_i(t) \text{ in $\Pmc(\Zmb^2)$ in probability, for each $t \in [0,T]$}, \label{eq:no_acc_LLN_nu_i_t}
		\end{align}
		for each $i \in [n]$, and
		\begin{align}
			\mu^n & \to \mu \text{ in $\Pmc(\Dmb([0,T]:\Zmb))$ in probability},
			\label{eq:no_acc_LLN_mu_i} \\
			\mu^n(t) & \to \mu(t) \text{ in $\Pmc(\Zmb)$ in probability, for each $t \in [0,T]$}.
			\label{eq:no_acc_LLN_mu_i_t}
		\end{align}
	\end{enumerate}
\end{Theorem}

\begin{Remark}
	(a) We abuse the notation to use $\nu_i^n,\nu_i,\mu^n,\mu$ to denote the empirical measures on the path space, and use $\nu_i^n(t),\nu_i(t),\mu^n(t),\mu(t)$ to denote the processes of the marginal empirical measures.
	We always precisely state the space to avoid the ambiguity in statements such as \eqref{eq:no_acc_LLN_nu_i} and \eqref{eq:no_acc_LLN_mu_i}.
	
	(b) Although \eqref{eq:no_acc_LLN_mu_i} only states the convergence on the path space, by applying standard arguments (cf.\ \cite[Theorem 4.7 and Lemma 4.8]{Meleard1996asymptotic}), one can make use of the fact that $\mu$ is deterministic and obtain suitable controls of jump sizes of $X_i$, to argue that the process $\{\mu^n(t) : t \in [0,T]\}$ converges in probability to $\{\mu(t):t \in [0,T]\}$ in the space $\Dmb([0,T]:\Pmc(\Zmb))$ endowed with the uniform topology.
\end{Remark}

\begin{Remark}
	(a) We note that $\Lmc(\xi_{ii}^n) \ne \Lmc(\xi_{ij}^n)$ in \eqref{eq:system} (resp.\ $\Lmc(\xi_{ii}) \ne \Lmc(\xi_{ij})$ in \eqref{eq:no_acceleration_system}) for $j \ne i$.
	However, the contribution of $\xi_{ii}^n$ to $\nu_i^n$ (resp.\ $\xi_{ii}$ to $\nu_i$) is negligible.
	Therefore one does not have to worry about the special evolution of $\xi_{ii}^n$.
	One may also simply define $\nu_i^n(t) = \frac{1}{n} \sum_{j \ne i}^n \delta_{(X_j^n(t),\xi_{ij}^n(t))}$ and the LLN, POC, and CLT results in this work will still be valid.

	(b) Although in this paper we consider the case of directed graphs, that is, we do not assume $\xi_{ij}^n = \xi_{ji}^n$ for $j \ne i$, we note that the results naturally extend to the undirected graph scenario, with an additional symmetry assumption $\Gammatil(y,\xitil,x,\xtil) = \Gammatil(y,\xitil,\xtil,x)$ and minor changes to the proofs.
\end{Remark}

\subsection{Systems with independent edge dynamics}
\label{sec:iid}

In this section we consider a system that generalizes the $\beta(n) \to 0$ limit of \eqref{eq:system}, in that we allow for non-Markovian edge processes (such as processes with delays and renewal processes).
Recall the node process $X_i^n$ and the local empirical measure process $\nu_i^n$
\begin{align}
	\label{eq:iid_system} 
	\begin{aligned}
		X_i^n(t) & = X_i(0) + \int_{\Xmb_t} y\one_{[0,\Gamma(y,X_i^n(s-),\nu^n_i(s-))]}(z) \, N_i(ds\,dy\,dz), \\
		\nu^n_i(t) & = \frac{1}{n} \sum_{j=1}^n \delta_{(X_j^n(t),\xi_{ij}^n(t))}, \quad i=1,\dotsc,n.
	\end{aligned}
\end{align}
Suppose $\xi_{ij}^n(t)=\xi_{ij}(t)$ are adapted, i.i.d.\ with $\Lmc(\xi_{ij}) = \theta \in \Pmc(\Dmb([0,T]:\Zmb))$, and independent of $\{X_i(0),N_i\}$.

We make the following assumption on $\gamma$.
\begin{Condition}
	\label{cond:iid}
	There exists $\{\gamma_y \in [0,\infty) : y \in \Zmb\}$ such that
	$0 \le \gamma(y,x,\xtil,\xitil) \le \gamma_y$ for all $y,x,\xtil,\xitil \in \Zmb$ and $C_\gamma := \int_\Zmb |y|\gamma_y\,\measurea(dy) < \infty$.
\end{Condition}

The next theorem shows that, under Condition \ref{cond:iid}, the limiting system is given by the unique solution to the following equations:
\begin{align}
	\label{eq:iid_system_limit}
	\begin{aligned}
		X_i(t) & = X_i(0) + \int_{\Xmb_t} y\one_{[0,\Gamma(y,X_i(s-),\nu(s-))]}(z) \, N_i(ds\,dy\,dz), \\
		\nu(t) & := \mu(t) \otimes \theta(t) := \Lmc(X_i(t)) \otimes \Lmc(\xi_{ij}(t)), \quad i,j \in \Nmb.
	\end{aligned}
\end{align}
The proof of Theorem \ref{thm:iid_LLN} is a standard application of a coupling argument.
For completeness it is given in Appendix \ref{sec:pf-iid}.

\begin{Theorem}
	\label{thm:iid_LLN}
	Suppose Condition \ref{cond:iid} hold. 
	Then
	\begin{enumerate}[(a)]
	\item The system \eqref{eq:iid_system_limit} has a unique pathwise solution.
	\item There exists some $\kappa =\kappa(T) < \infty$ such that
	\begin{equation}
		\label{eq:iid_LLN1}
		\max_{i \in [n]} \Emb \|X_i^n-X_i\|_{*,T} \le \frac{\kappa}{\sqrt{n}}.
	\end{equation} 
	\item (POC) For any $k \in \Nmb$, as $n \to \infty$,
	\begin{equation}
		\label{eq:iid_POC}
		\Lmc(X_1^n,\dotsc,X_k^n) \to \mu^{\otimes k}, \quad \Lmc(X_1^n(t),\dotsc,X_k^n(t)) \to [\mu(t)]^{\otimes k} \text{ for each } t \in [0,T],
	\end{equation}	
	where $\mu := \Lmc(X_i) \in \Pmc(\Dmb([0,T]:\Zmb))$ for $i \in \Nmb$.
	\item (LLN) As $n \to \infty$, for each $i \in [n]$,
	\begin{align}
		\nu^n_i := \frac{1}{n} \sum_{j=1}^n \delta_{(X_j^n(\cdot),\xi_{ij}^n(\cdot))} & \to \nu := \mu \otimes \theta \text{ in $\Pmc(\Dmb([0,T]:\Zmb^2))$ in probability}, \label{eq:iid_LLN_nu_i} \\
		\nu^n_i(t) & \to \nu(t) \text{ in $\Pmc(\Zmb^2)$ in probability, for each $t \in [0,T]$}. \label{eq:iid_LLN_nu_i_t}
	\end{align}
	Moreover, as $n \to \infty$,
	\begin{align}		
		\mu^n := \frac{1}{n} \sum_{i=1}^n \delta_{X_i^n(\cdot)} & \to \mu \text{ in $\Pmc(\Dmb([0,T]:\Zmb))$ in probability}, \label{eq:iid_LLN_mu_i} \\
		\mu^n(t) := \frac{1}{n} \sum_{i=1}^n \delta_{X_i^n(t)} & \to \mu(t) \text{ in $\Pmc(\Zmb)$ in probability, for each $t \in [0,T]$}. \label{eq:iid_LLN_mu_i_t}
	\end{align}	
	\end{enumerate}
\end{Theorem}

\begin{Remark}
	Theorem \ref{thm:iid_LLN} is not a simple consequence of Theorems \ref{thm:no_acceleration_1} and \ref{thm:no_acceleration_2} for the case $\beta(n) \to 0$, although it seems to be.
	In particular, note that the system \eqref{eq:iid_system} only assumes $\xi_{ij}^n$ to be adapted, which may be a non-Markovian process {(whose evolution is give by, e.g.\, path-dependent stochastic differential equations or renewal processes)}.
\end{Remark}

\section{Systems with accelerated edge dynamics}
\label{sec:acceleration}

In this section we study the system \eqref{eq:system} with accelerated edge dynamics compared to the node dynamics, that is when $\beta(n) \to \infty$.
In addition to Condition \ref{cond:no_acceleration}, we make the following assumption on $\Gamma$ and $\Gammatil$.

\begin{Condition}
	\phantomsection
	\label{cond:acceleration} 
	\begin{enumerate}[(i)]
	\item 	
		For $\xitil_0,x,\xtil \in \Zmb$, denote by $\Ytil_{(x,\xtil,\xitil_0)}$ the continuous-time Markov chain with transition rate matrix $\Gammatil(y,\xitil,x,\xtil)$ (representing the rate of jumping from $\xitil$ to $\xitil+y$) starting at $\xitil_0$.
		Suppose that there exists a unique invariant distribution $Q(x,\xtil,d\xitil)$ of $\Ytil_{(x,\xtil,\xitil_0)}$ and
		$$\left| \Emb \gamma(y,x,\xtil,\Ytil_{(x,\xtil,\xitil_0)}(t)) - \int_\Zmb \gamma(y,x,\xtil,\xitil) Q(x,\xtil,d\xitil) \right| \le \widetilde\kappa(t) \gamma_y (1+|x|+|\xtil|)$$
		for some $\widetilde\kappa(t)$ such that $\int_0^\infty \widetilde\kappa(t)\,dt < \infty$.
	\item
		$\Emb[|X_i(0)|^2] < \infty$ and $C_{\gamma,2} := \int_\Zmb |y|^2 \gamma_y\,\measurea(dy) < \infty$.
	\end{enumerate}
\end{Condition}

\begin{Remark}
	\phantomsection
	\label{rmk:acc}
	\begin{enumerate}[(a)]
	\item Condition \ref{cond:acceleration}(i) holds if $\Ytil_{(x,\xtil,\xitil_0)}$ is uniformly exponentially ergodic in $(x,\xtil,\xitil_0)$ in the following sense: There exists $\alpha > 0$ and $C > 0$ such that 
	\begin{equation*}
		\sum_{\xitil \in \Zmb} |\Pmb(\Ytil_{(x,\xtil,\xitil_0)}(t)=\xitil) - Q(x,\xtil,\{\xitil\})| \le C e^{-\alpha t} (1+|x|+|\xtil|),
	\end{equation*}
	for all $\xitil_0,x,\xtil \in \Zmb$.
	{Such an ergodicity could be obtained under suitable assumptions on $\Gammatil$ (see e.g.\ the sufficient criteria in \cite{Tweedie1981Criteria}).}
	In particular, it is satisfied if the system is finite state (cf.\ \cite{Tweedie1981Criteria} and \cite[Theorem 6.5]{Anderson1991Strong}). 
	\item The assumption $\int_0^\infty \widetilde\kappa(t)\,dt < \infty$ in Condition \ref{cond:acceleration}(i) is needed to obtain the rate of convergence in Theorem \ref{thm:acc_LLN}.
	If one is only interested in the convergence of $X_i^n, \mu^n, \nu_i^n$, then it would be sufficient (see Remark \ref{rmk:kappa}) to assume that $\lim_{t \to \infty} \frac{1}{t}\int_0^t\widetilde\kappa(s)\,ds = 0$.
	\end{enumerate}
\end{Remark}

The next theorem shows that, under Conditions \ref{cond:no_acceleration} and \ref{cond:acceleration}, the limiting system is given by the unique solution to the following equations:
\begin{align}
	X_i(t) & = X_i(0) + \int_{\Xmb_t} y\one_{[0,\Gamma(y,X_i(s-),\nu_i(s-))]}(z) \, N_i(ds\,dy\,dz), \notag \\
	\nu_i(t)(d\xtil\,d\xitil) & = \mu_t(d\xtil)\,Q(X_i(t),\xtil,d\xitil), \label{eq:acceleration_system} \\
	\mu & = \Lmc(X_i), \quad \mu_t = \Lmc(X_i(t)), \quad i \in \Nmb. \notag
\end{align}
We note that here the system does not suffer from the path-dependent subtlety in \eqref{eq:no_acceleration_system}, and hence serves as a simpler approximation of \eqref{eq:system} than \eqref{eq:no_acceleration_system}, when $\beta(n)$ is large.

The proof of Theorem \ref{thm:acc_LLN} is given in Section \ref{sec:pf-acc}.

\begin{Theorem}
	\label{thm:acc_LLN}
	Suppose Conditions \ref{cond:no_acceleration} and \ref{cond:acceleration} hold. Then
	\begin{enumerate}[(a)]
	\item There is a unique pathwise solution $\{X_i\}$ to the system \eqref{eq:acceleration_system}.
	\item There exists some $\kappa =\kappa(T) < \infty$ such that
		\begin{align}
			\Emb \|X_i^n-X_i\|_{*,T} & \le \frac{\kappa}{\sqrt{n}}+\frac{\kappa}{\sqrt{\beta(n)}}, \label{eq:acc_LLN1}
		\end{align}
	\item (POC) For any $k \in \Nmb$, as $n \to \infty$,
	\begin{equation}
		\label{eq:acceleration_POC}
		\Lmc(X_1^n,\dotsc,X_k^n) \to \mu^{\otimes k}, \quad \Lmc(X_1^n(t),\dotsc,X_k^n(t)) \to [\mu(t)]^{\otimes k} \text{ for each } t \in [0,T].
	\end{equation}
	\item (LLN) As $n \to \infty$,
		\begin{equation}
			\mbox{$\mu^n \to \mu$ in $\Pmc(\Dmb[0,T]:\Zmb)$ and $\mu^n(t) \to \mu(t)$ in $\Pmc(\Zmb)$, in probability}, \label{eq:acc_LLN2}
		\end{equation}
		for each $t \in [0,T]$, where $\mu^n$ and $\mu^n(t)$ are introduced in \eqref{eq:nu_i_n}.
	\item Suppose in addition that 
		\begin{equation}
			\label{eq:acc_asump_addition}
			\sum_{\xitil \in \Zmb} |\Pmb(\Ytil_{(x,\xtil,\xitil_0)}(t)=\xitil) - Q(x,\xtil,\{\xitil\})| \le C(t) (1+|x|+|\xtil|)
		\end{equation}
		for some positive $C(t)$ such that $\lim_{t \to \infty} C(t)=0$, and $\Ytil_{(x,\xtil,\xitil_0)}$ and $Q(x,\xtil,\{\xitil\})$ are as in Condition \ref{cond:acceleration}.
		Then
		\begin{equation}
			\nu_i^n(t) \to \nu_i(t) \mbox{ in $\Pmc(\Zmb^2)$ in probability, for each $i \in [n]$ and } t \in (0,T]
			\label{eq:acc_LLN3}
		\end{equation}
		and hence
		\begin{equation}
			\eta_i^n(dt\,d\xtil\,d\xitil) := \nu_i^n(t)(d\xtil\,d\xitil) \,dt \to \eta_i(dt\,d\xtil\,d\xitil) := \nu_i(t)(d\xtil\,d\xitil) \, dt \label{eq:acc_LLN4}
		\end{equation}
		in $\Pmc([0,T]\times\Zmb^2)$, in probability, for each $i \in [n]$.
	\end{enumerate}	
\end{Theorem}

\subsection{Approximation error}

In this section we study the approximation error in terms of $\beta \to \infty$ between two limiting systems \eqref{eq:no_acceleration_system} and \eqref{eq:acceleration_system} obtained as $\beta(n) \to \beta \in [0,\infty)$ and $\beta(n) \to \infty$.
To distinguish the two systems, we rewrite \eqref{eq:no_acceleration_system} as
\begin{align*}
	X_i^\beta(t) & = X_i(0) + \int_{\Xmb_t} y\one_{[0,\Gamma(y,X_i^\beta(s-),\nu_i^\beta(s-))]}(z) \, N_i(ds\,dy\,dz), \\
	\xi_{ij}^\beta(t) & = \xi_{ij}(0) + \int_{\Xmb_t} y\one_{[0,\beta\Gammatil(y,\xi_{ij}^\beta(s-),X_i^\beta(s-),X_j^\beta(s-))]}(z) \, N_{ij}(ds\,dy\,dz), \\
	\nu_i^\beta(t) & = \lim_{n \to \infty} \frac{1}{n} \sum_{j=1}^n \delta_{(X_j^\beta(t),\xi_{ij}^\beta(t))},
\end{align*}
and recall \eqref{eq:acceleration_system}.
The following theorem characterizes the approximation error of these two limiting systems as $\beta \to \infty$.
The proof is given in Section \ref{sec:pf-acc}.

\begin{Theorem}
	\label{thm:comparison}
	Suppose Conditions \ref{cond:no_acceleration} and \ref{cond:acceleration} hold. 
	Then
	\begin{enumerate}[(a)]
	\item 
		There exists some $\kappa =\kappa(T) < \infty$ such that
		\begin{equation}
			\Emb \|X_i^\beta-X_i\|_{*,T} \le \frac{\kappa}{\sqrt{\beta}} \label{eq:comparison_1}
		\end{equation}
		and
		\begin{equation}
			\dBL(\mu^\beta,\mu) \le \frac{\kappa}{\sqrt{\beta}}, \label{eq:comparison_2} 
		\end{equation}
		where $\mu^\beta := \Lmc(X_i^\beta)$.
	\item 
		Suppose in addition that \eqref{eq:acc_asump_addition} holds, then
		\begin{equation}
			\nu_i^\beta(t) \to \nu_i(t) \mbox{ in } \Pmc(\Zmb^2), \mbox{ in probability, for each } i \in \Nmb \mbox{ and } t \in (0,T]
			\label{eq:comparison_3}
		\end{equation}
		and hence
		\begin{equation}
			\eta_i^\beta(dt\,d\xtil\,d\xitil) := \nu_i^\beta(t)(d\xtil\,d\xitil) \,dt \to \eta_i(dt\,d\xtil\,d\xitil) := \nu_i(t)(d\xtil\,d\xitil) \, dt \label{eq:comparison_4}
		\end{equation}
		in $\Pmc([0,T]\times\Zmb^2)$, in probability, for each $i=1,\dotsc,n$.
	\end{enumerate}
\end{Theorem}

\begin{Remark}
	\label{rmk:comparison}
	Similar to Remark \ref{rmk:acc}(b), the assumption $\int_0^\infty \widetilde\kappa(t)\,dt < \infty$ in Condition \ref{cond:acceleration}(i) is needed to obtain the rate of convergence in Theorem \ref{thm:comparison}.
	If one is only interested in the convergence of $X_i^\beta, \mu^\beta, \nu_i^\beta$, then it would be sufficient (see Remark \ref{rmk:kappa-acc}) to assume that $\lim_{t \to \infty} \frac{1}{t} \int_0^t \widetilde\kappa(s)\,ds = 0$.
\end{Remark}

\subsection{Riccati equation for limiting marginal probabilities}
\label{sec:eg-acceleration}

In this section we will get a Riccati equation for the evolution of $\mu = \Lmc(X_i)$.

{For $t \in [0,T]$ and $k \in \Zmb$, let $$p_t(k) = \mu_t(\{k\}) = \Pmb(X_i(t)=k), \quad i \in \Nmb.$$
From \eqref{eq:acceleration_system} we have
\begin{equation*}
	\one_{\{X_i(t)=k\}} = \one_{\{X_i(0)=k\}} + \int_{\Xmb_t} \left( \one_{\{X_i(s-)+y=k\}} - \one_{\{X_i(s-)=k\}} \right) \one_{[0,\Gamma(y,X_i(s-),\nu_i(s-))]}(z) \, N_i(ds\,dy\,dz).
\end{equation*}
Taking expectations and differentiating with respect to $t$ gives
\begin{align*}
	\frac{d}{dt} p_t(k) & = -p_t(k) \sum_{y \in \Zmb} \int_{\Zmb^2} \rho(y) \gamma(y,k,\xtil,\xitil) \,\mu_t(d\xtil)\,Q(k,\xtil,d\xitil) \\
	& \quad + \sum_{y \in \Zmb} p_t(k-y) \int_{\Zmb^2} \rho(y) \gamma(y,k-y,\xtil,\xitil) \,\mu_t(d\xtil)\,Q(k-y,\xtil,d\xitil) \\
	& = -p_t(k) \sum_{y \in \Zmb} \sum_{\xtil \in \Zmb} \int_\Zmb \rho(y) \gamma(y,k,\xtil,\xitil) p_t(\xtil)\,Q(k,\xtil,d\xitil) \\
	& \quad + \sum_{y \in \Zmb} p_t(k-y) \sum_{\xtil \in \Zmb} \int_\Zmb \rho(y) \gamma(y,k-y,\xtil,\xitil) p_t(\xtil)\,Q(k-y,\xtil,d\xitil),
\end{align*}
where $\rho$ is the finite measure introduced in Section \ref{sec:notation}.}
This is an infinite dimensional Riccati equation.

\begin{Remark}
	\label{rmk:riccati}
	If, for simplicity, the system is finite dimensional with $\rho(y)=1$, {$\gamma(y,x,\xtil,\xitil) = \gamma_0(y,x,\xitil)$} does not depend on $\xtil$, and $Q(x,\xtil,d\xitil)=Q_0(d\xitil)$ does not depend on $x$ or $\xtil$, then the above Riccati equation reduces to the following finite dimensional linear ordinary differential equations:
	\begin{align*}
		\frac{d}{dt} p_t(k) &  = -p_t(k) \sum_{y \in \Zmb} \int_\Zmb \gamma_0(y,k,\xitil) \,Q_0(d\xitil) + \sum_{y \in \Zmb} p_t(k-y) \int_\Zmb \gamma_0(y,k-y,\xitil) \,Q_0(d\xitil). 
	\end{align*}
	Note that the special choice of {$\gamma$} and $Q$ does not mean that the particle system is trivial.
	Indeed, ${\gamma_0}(y,x,\xitil)$ still allows interaction as there is the dependence on relationship $\xitil$.
	Also note that although $Q_0(d\xitil)$ does not depend on $x$ or $\xtil$, one can still have general $\Gammatil(y,\xitil,x,\xtil)$ associated to $\xi_{ij}^n$ such that the stationary distribution is given by $Q_0(d\xitil)$.
\end{Remark}

\section{Fluctuations and central limit theorems}
\label{sec:CLT}

Finally we will study the fluctuations of empirical measures about the law of large numbers limit when $\beta(n) \to \beta \in [0,\infty)$.
For simplicity, we assume $\beta(n)=\beta \in [0, \infty)$.
In addition to Condition \ref{cond:no_acceleration}, we make the following assumption on $\Gamma$.

\begin{Condition}
	\phantomsection
	\label{cond:CLT}
	There exists $\varepsilon \in (0,1]$ such that $\varepsilon \le \gamma(y,x,\xtil,\xitil) \le \frac{1}{\varepsilon}$ for all $y,x,\xtil,\xitil \in \Zmb$.
\end{Condition}

The fluctuations will be characterized by CLT in Theorems \ref{thm:CLT_nu} and \ref{thm:CLT_mu} in Section \ref{sec:CLT-sub}.
Variances in these CLT will be expressed in terms of norms of certain integral operators defined in Section \ref{sec:integral_operators} using canonical spaces and processes introduced in Section \ref{sec:canonical}.

\subsection{Canonical processes}
\label{sec:canonical}

We first introduce the following canonical spaces and stochastic processes.
Let $\Omega_v = \Mmc_T \times \Dmb([0,T]:\Zmb)^2$, $\Omega_e = \Dmb([0,T]:\Zmb)$, and $\Omega_0 = \Dmb([0,T]:\Zmb)$.
Define for $n \in \Nmb$ the probability measure $P^n$ on $\Omega_n := \Omega_0 \times \Omega_v^n \times \Omega_e^{n(n-1)}$ by
\begin{equation*}
	P^n := \Lmc \left( X_1, (N_1,X_1,\xi_{11}), (N_2,X_2,\xi_{12}), \dotsc, (N_n,X_n,\xi_{1n}), \{\xi_{ij}:i=2,\dotsc,n, j \in [n] \} \right).
\end{equation*}
For $\omega = (\omega_0, \omega_1, \omega_2, \dotsc, \omega_n, \omegabar) \in \Omega_n$ with $\omega_i=(\omega_{ik} : k=1,2,3)$ for $i=1,\dotsc,n$ and $\omegabar=(\omegabar_{ij} : i=2,\dotsc,n, j \in [n])$, let
$$V_0(\omega):=\omega_0, \quad V_i(\omega):=\omega_i=(\omega_{i1},\omega_{i2},\omega_{i3}), \: i \in [n],$$
and, abusing notation, write
\begin{align*}
	V_0 := X_1, \quad V_i := (N_i,X_i,\xi_{1i}), \: i \in [n], \quad \xi_{ij}(\omega) := \omegabar_{ij}, \: i=2,\dotsc,n, \: j \in [n].
\end{align*}
Note that such an abuse of notation just says that the distribution of the canonical processes $(N_i,X_i,\xi_{ij} : i,j=1,\dotsc,n)$ under $P^n$ is the same as that of processes $(N_i,X_i,\xi_{ij} : i,j=1,\dotsc,n)$ in \eqref{eq:no_acceleration_system}.
Also define the canonical processes $V_* := (N_*,X_*,\xi_*)$ on $\Omega_v$ by
\begin{equation*}
	V_*(\omega) := (N_*(\omega),X_*(\omega),\xi_*(\omega)) := (\omega_1,\omega_2,\omega_3), \quad \omega=(\omega_1,\omega_2,\omega_3) \in \Omega_v.
\end{equation*}
Define the compensated PRM $\Ncompensated_*$ on $\Omega_v$ by 
$$\Ncompensated_*(\omega)(ds\,dy\,dz) := N_*(\omega)(ds\,dy\,dz) - ds \times \measurea(dy) \times dz.$$

\subsection{Some integral operators}
\label{sec:integral_operators}

We will need some functions for stating our central limit theorem.
Let $$\alpha(x,\cdot) := \Lmc((N_2,X_2,\xi_{12}) \in \cdot | V_0=x)$$ and $$\theta(t,x[t],\xtil[t]) := \Lmc(\xi_{12}(t) \,|\, X_1[t]=x[t],X_2[t]=\xtil[t])$$ be the corresponding regular conditional probabilities, for $x,\xtil \in \Dmb([0,T]:\Zmb)$.
Let $$P_0 := \Lmc(X_1), \quad \Xi := \Lmc(N_2,X_2,\xi_{12}).$$
Note that $P_0$ is just $\mu$ but we write it in this way to emphasize its role as a common factor.
Recall $\nu_i(t) = \Lmc((X_j(t),\xi_{ij}(t)) \,|\, X_i[t])$ and we rewrite it as $\nu(t,X_i[t])$ to emphasize its dependence on $X_i[t]$.
Define the function $h \colon \Omega_v \times \Omega_v \to \Rmb$ ($\Xi \times \Xi$ a.s.) by
\begin{align*}
	h(\omega_1,\omega_2) & := \int_{\Xmb_T} \one_{[0,\Gamma(y,X_*(\omega_1)(s-),\nu(s-,X_*(\omega_1)[s-]))]}(z) \\
	& \cdot \frac{\lan \gammabar_{s-,y}(X_*(\omega_1)[s-],X_*(\omega_2)(s-),\cdot), \theta(s-,X_*(\omega_1)[s-],X_*(\omega_2)[s-]) \ran }{\Gamma(y,X_*(\omega_1)(s-),\nu(s-,X_*(\omega_1)[s-]))} \,\Ncompensated_*(\omega_1)(ds\,dy\,dz),
\end{align*}
where
\begin{equation*}
	\gammabar_{t,y}(x_1,x_2,\xi) := \gamma(y,x_1(t),x_2,\xi) - \lan \gamma(y,x_1(t),\cdot,\cdot), \nu(t,x_1[t]) \ran,
\end{equation*}
for $t \in [0,T]$, $x_1 \in \Dmb([0,t]:\Zmb)$ and $y,x_2,\xi \in \Zmb$.
Fix $\omega_0 \in \Omega_0$ and consider the Hilbert space $\Hmc_{\omega_0} := L^2(\Omega_v,\alpha(\omega_0,\cdot))$.
Define the integral operator $A_{\omega_0}$ on $\Hmc_{\omega_0}$ by
\begin{equation*}
	A_{\omega_0} g(\omega_2) = \int_{\Omega_v} g(\omega_1) h(\omega_1,\omega_2) \, \alpha(\omega_0,d\omega_1), \quad g \in \Hmc_{\omega_0}, \omega_2 \in \Omega_v.
\end{equation*}
Denote by $I$ the identity operator.

\subsection{Central limit theorems}
\label{sec:CLT-sub}

We denote by $\Amc$ the collection of all measurable maps $\varphi \colon \Dmb([0,T]:\Zmb)^2 \to \Rmb$ such that $\varphi(X_*,\xi_*) \in L^2(\Omega_v,\alpha(\omega_0,\cdot))$ for $P_0$ a.e.\ $\omega_0 \in \Omega_0$.
For $\varphi \in \Amc$ and $\omega_0 \in \Omega_0$, let
\begin{equation*}
	m_\varphi(\omega_0) := \int_{\Omega_v} \varphi(X_*(\omega_1),\xi_*(\omega_1)) \, \alpha(\omega_0,d\omega_1), \quad \Phi_{\omega_0}(\omega) := \varphi(X_*(\omega),\xi_*(\omega)) - m_\varphi(\omega_0), \quad \omega \in \Omega_v.
\end{equation*}
For $\omega_0 \in \Omega_0$ and $\varphi \in \Amc$, define
\begin{equation*}
	\sigma^\varphi_{\omega_0} := \|(I-A_{\omega_0})^{-1} \Phi_{\omega_0}\|_{\Hmc_{\omega_0}},
\end{equation*}
and denote by $\pi_{\omega_0}^\varphi$ the normal distribution with mean $0$ and standard deviation $\sigma_{\omega_0}^\varphi$.
Here the operator $I-A_{\omega_0}$ is invertible by Lemma \ref{lem:A}.
Define $\pi^\varphi \in \Pmc(\Rmb)$ by
\begin{equation*}
	\pi^\varphi := \int_{\Omega_0} \pi_{\omega_0}^\varphi \, P_0(d\omega_0).
\end{equation*}
Finally let
\begin{equation*}
	\eta^n(\varphi) := \sqrt{n} \lan \varphi, \nu_1^n-\nu_1 \ran = \sqrt{n} \left( \frac{1}{n} \sum_{j=1}^n \varphi(X_j^n,\xi_{ij}^n) - m_\varphi(X_1^n) \right).
\end{equation*}

The following is the CLT for the empirical measure $\nu_1^n$ of neighboring nodes and edges of node $1$.
The proof is given in Section \ref{sec:pf-CLT}.

\begin{Theorem}
	\label{thm:CLT_nu}
	Suppose that Conditions \ref{cond:no_acceleration} and \ref{cond:CLT} hold.
	Then, for all $\varphi \in \Amc$, $\Lmc(\eta^n(\varphi)) \to \pi^\varphi$ weakly as $n \to \infty$.
\end{Theorem}

We note that the limit of the fluctuation of $\nu_1^n$ is a Gaussian mixture.
This is due to the conditional independence of $\{\xi_{ij} : j \ne i\}$ given $X_i$, so that $X_i$ serves as a common noise which would not be averaged out.
{Unlike in \cite{BhamidiBudhirajaWu2019weakly} where the fluctuations of node states are analyzed, when one restricts to two colors, such a common noise feature still exists and Theorem \ref{thm:CLT_nu} could be understood as a CLT for the node states and connectivity, jointly, from the viewpoint of node $1$.}

{The CLT limit} could be written as a single Gaussian random variable if one is interested in the fluctuation of $\mu^n$.
To be precise, let $\Amc_x := L^2(\Dmb([0,T]:\Zmb),\mu)$ and $\Hmc = L^2(\Omega_v,\Xi)$.
Define the integral operator $A$ on $\Hmc$ by
\begin{equation*}
	A g(\omega_2) = \int_{\Omega_v} g(\omega_1) h(\omega_1,\omega_2) \, \Xi(d\omega_1), \quad g \in \Hmc, \omega_2 \in \Omega_v.
\end{equation*}
For $\varphi \in \Amc_x$, let
\begin{align*}
	\Phi(\omega) & := \varphi(X_*(\omega)) - \int_{\Omega_v} \varphi(X_*(\omega_1))\,\Xi(d\omega_1), \quad \omega \in \Omega_v,
\end{align*}
and
\begin{equation*}
	\eta_x^n(\varphi) := \sqrt{n} \lan \varphi, \mu^n-\mu \ran = \sqrt{n} \left( \frac{1}{n} \sum_{j=1}^n \varphi(X_j^n) - \Emb \varphi(X_1) \right).
\end{equation*}

The following is the CLT for empirical measure $\mu^n$ of all nodes in the system.
The proof is given in Section \ref{sec:pf-CLT}.

\begin{Theorem}
	\label{thm:CLT_mu}
	Suppose that Conditions \ref{cond:no_acceleration} and \ref{cond:CLT} hold.
	Then $\{\eta^n_x(\varphi) : \varphi \in \Amc_x\}$ converges as $n \to \infty$ to a mean $0$ Gaussian field $\{\eta_x(\varphi) : \varphi \in \Amc_x\}$ in the sense of convergence of finite dimensional distributions, where for $\varphi,\psi \in \Amc_x$,
	\begin{equation*}
		\Emb[\eta_x(\varphi)\eta_x(\psi)] = \lan (I-A)^{-1} \Phi, (I-A)^{-1} \Psi \ran_{\Hmc},
	\end{equation*}
	where $\Psi := \psi(X_*) - \Emb \psi(X_1) \in \Hmc$.
\end{Theorem}

\subsection{An example with explicit variance}
\label{sec:eg-CLT}

In this section we provide an example where one can get explicit variance from Theorem \ref{thm:CLT_mu} for some functionals $\varphi \in \Amc_x$.

Suppose Conditions \ref{cond:no_acceleration} and \ref{cond:CLT} hold, and $$\gamma(y,x,\xtil,\xitil)=c_0(y)b_0(x)+c_1(y)b_1(\xtil)+c_2(y)b_2(\xitil)+c_3(y),$$
for some bounded functions $b_0,b_1,b_2,c_0,c_1,c_2,c_3 \colon \Zmb \to \Rmb$.
Assume that $\measurea(\{\cdot\}), \Gammatil(\cdot,\xitil,x,\xtil)$ are even functions from $\Zmb$ to $\Rmb_+$, $b_1,b_2$ are odd functions, and $c_0,b_0,c_3$ are even functions.
Also suppose $\beta=1$, $X_i(0)=0$ and $\xi_{ij}(0)=0$.

Consider the system
\begin{align}
	X_i(t) & = \int_{\Xmb_t} y\one_{[0,c_0(y)b_0(X_i(s-))+c_3(y)]}(z) \, N_i(ds\,dy\,dz), \notag \\
	\xi_{ij}(t) & = \int_{\Xmb_t} y\one_{[0,\beta\Gammatil(y,\xi_{ij}(s-),X_i(s-),X_j(s-))]}(z) \, N_{ij}(ds\,dy\,dz), \label{eq:CLT-eg0} \\
	\nu_i(t) & = \lim_{n \to \infty} \frac{1}{n} \sum_{j=1}^n \delta_{(X_j(t),\xi_{ij}(t))} =\Lmc((X_k(t),\xi_{ik}(t)) \,|\, X_i[t]), \quad k \ne i. \notag
\end{align}
From the even function property of $\measurea(\{\cdot\})$, $c_0$, $b_0$, $c_3$ and $\Gammatil(\cdot,\xitil,x,\xtil)$ we see that $\Lmc(X_i(t))=\Lmc(-X_i(t))$ and $\Lmc(\xi_{ij}(t)|X_i[t],X_j[t])=\Lmc(-\xi_{ij}(t)|X_i[t],X_j[t])$.
It then follows from the odd function property of $b_1,b_2$ that
\begin{equation}
	\label{eq:CLT-eg1}
	\Gamma(y,X_i(s),\nu_i(s))=c_0(y)b_0(X_i(s))+c_3(y)
\end{equation} 
and 
\begin{equation}
	\label{eq:CLT-eg2}
	\lan \gammabar_{s,y}(X_*(\omega_1)[s],X_*(\omega_2)(s),\cdot), \theta(s,X_*(\omega_1)[s],X_*(\omega_2)[s]) \ran=c_1(y)b_1(X_*(\omega_2)(s))
\end{equation}
for $\omega_1,\omega_2 \in \Omega_v$.
The equality \eqref{eq:CLT-eg1} and Theorem \ref{thm:no_acceleration_1}(b) imply that \eqref{eq:CLT-eg0} is indeed the limiting system \eqref{eq:no_acceleration_system}.

Now consider $\varphi \in \Amc_x$ defined by $$\varphi(x) = x_T - \int_0^T b_1(x_s)\,ds \int_\Zmb yc_1(y)\,\measurea(dy), \quad x \in \Dmb([0,T]:\Zmb).$$
For this example we can explicitly describe the
asymptotic distribution of
$$\eta_x^n(\varphi) = \sqrt{n} \lan \varphi, \mu^n-\mu \ran = \frac{1}{\sqrt{n}} \sum_{j=1}^n \left( X_j^n(T) - \int_0^T b_1(X_j^n(s))\,ds \int_\Zmb yc_1(y)\,\measurea(dy) \right).$$
Following the notation above Theorem \ref{thm:CLT_mu}, we have
\begin{align*}
	\Phi(\omega) & = X_{*,T}(\omega) - \int_0^T b_1(X_{*,s}(\omega))\,ds \int_\Zmb yc_1(y)\,\measurea(dy), \quad \omega \in \Omega_v,
\end{align*}
and from \eqref{eq:CLT-eg2} we have
\begin{align*}
	h(\omega_1,\omega_2) & := \int_{\Xmb_T} \one_{[0,c_0(y)b_0(X_{*,s}(\omega_1))+c_3(y)]}(z) \frac{c_1(y)b_1(X_{*,s}(\omega_2))}{c_0(y)b_0(X_{*,s}(\omega_1))+c_3(y)} \,\Ncompensated_*(\omega_1)(ds\,dy\,dz)
\end{align*}
for $\omega_1,\omega_2 \in \Omega_v$.
The special form of $\varphi$ allows us to determine $(I-A)^{-1}\Phi$.
Indeed, let
\begin{align*}
	\Psi(\omega) & = \int_{\Xmb_T} y\one_{[0,c_0(y)b_0(X_{*,s}(\omega))+c_3(y)]}(z) \, \Ncompensated_*(\omega)(ds\,dy\,dz) \\
	& = \int_{\Xmb_T} y\one_{[0,c_0(y)b_0(X_{*,s}(\omega))+c_3(y)]}(z) \, N_*(\omega)(ds\,dy\,dz) = X_{*,T}(\omega), \quad \Xi\text{-a.s.\ } \omega \in \Omega_v,
\end{align*}
where the second line uses the even function property of $c_0$, $c_3$ and $\measurea(\{\cdot\})$.
Note that
\begin{align*}
	A \Psi(\omega) 
	& = \int_{[0,T]\times\Zmb} yc_1(y)b_1(X_{*,s}(\omega)) \,ds\,\measurea(dy), \\
	(I-A)\Psi(\omega) 
	& = X_{*,T}(\omega) - \int_0^T b_1(X_{*,s}(\omega))\,ds \int_\Zmb yc_1(y)\,\measurea(dy) = \Phi(\omega), \quad \Xi\text{-a.s.\ } \omega \in \Omega_v.
\end{align*}
Therefore $(I-A)^{-1}\Phi = \Psi$ and from Theorem \ref{thm:CLT_mu} we have that $\eta_x^n(\varphi)$ converges in distribution to a mean zero Gaussian random variable with variance
\begin{equation*}
	\sigma_\varphi^2 = \|\Psi\|_\Hmc^2 = \Emb \left[ \left( X_i(T) \right)^2 \right].
\end{equation*}

\begin{Remark}
	\begin{enumerate}[(a)]
	\item 
		We note that the variance is not simply the variance of $X_i(T) - \int_0^T b_1(X_i(s))\,ds \int_\Zmb yc_1(y)\,\measurea(dy)$, since the propagation of chaos property in Theorem \ref{thm:no_acceleration_2}(b) only guarantees asymptotic independence for finite collection of $X_i^n$.
	\item 
		It is straightforward to extend above calculation to functionals that depend on states at finitely many time instants.
		Indeed, taking $$\varphi(x) = \sum_{k=1}^m a_k \left( x_{t_k} - \int_0^{t_k} b_1(x_s)\,ds \int_\Zmb yc_1(y)\,\measurea(dy) \right), \quad x \in \Dmb([0,T]:\Zmb),$$ for some $0 \le t_1<\dotsb<t_m \le T$, $a_1,\dotsc,a_m \in \Rmb$, $m \in \Nmb$, one has $$\Phi(\omega) = \sum_{k=1}^m a_k \left( X_{*,t_k}(\omega) - \int_0^{t_k} b_1(X_{*,s}(\omega))\,ds \int_\Zmb yc_1(y)\,\measurea(dy)\right)$$ and $(I-A)^{-1}\Phi = \Psi$, where
		$\Psi(\omega) = \sum_{k=1}^m a_k X_{*,t_k}(\omega)$ for $\Xi-$a.s.\ $\omega \in \Omega_v$.
		It then follows from Theorem \ref{thm:CLT_mu} that $\eta_x^n(\varphi) = \sqrt{n} \lan \varphi, \mu^n-\mu \ran$
		converges in distribution to a mean zero Gaussian random variable with variance
		\begin{equation*}
			\sigma_\varphi^2 = \|\Psi\|_\Hmc^2 = \Emb \left[ \left( \sum_{k=1}^m a_kX_i(t_k) \right)^2 \right].
		\end{equation*}
	\end{enumerate}
\end{Remark}

\section{Proofs of laws of large numbers and propagation of chaos results}
\label{sec:pf-LLN}

We first state an elementary result on (conditionally) i.i.d.\ random variables. The proof is omitted.
\begin{Lemma}
	\label{lem:iid-moment}
	Let $\{Y_i : i=1,\dotsc,n\}$ be a collection of $\Smb$-valued random variables defined on some probability space $(\Omega,\Fmc,\Pmb)$, where $\Smb$ is some Polish space.
	Suppose $\{Y_i : i=1,\dotsc,n\}$ are conditionally i.i.d.\ given some $\sigma$-field $\Gmc \subset \Fmc$.
	Then for each $k \in \Nmb$, there exists $a_k \in (0,\infty)$ such that
	\begin{equation*}
		\sup_{\|f\|_\infty \le 1} \Emb \left| \frac{1}{n} \sum_{i=1}^n \left( f(Y_i) - \Emb[f(Y_i)|\Gmc] \right) \right|^k \le \frac{a_k}{n^{k/2}}.
	\end{equation*} 
\end{Lemma}

\subsection{Proofs of Theorems \ref{thm:no_acceleration_1} and \ref{thm:no_acceleration_2}}
\label{sec:pf-no-acc}

{The proofs of Theorems \ref{thm:no_acceleration_1} and \ref{thm:no_acceleration_2} are mainly based on tightness, weak convergence and coupling arguments, which may be skipped for specialists.}

\begin{proof}[Proof of Theorem \ref{thm:no_acceleration_1}]
	(a) We will adapt the argument in \cite{KurtzXiong1999} to prove {the weak existence} of the system \eqref{eq:no_acceleration_system}.	
	For $m \in \Nmb$, let 
	\begin{equation}
		\Ntil_i^m([0,t] \times A) := {\int_{[0,\frac{\lfl mt \rfl}{m}]\times A} N_i(ds\,dy\,dz)}, \:\: \Ntil_{ij}^m([0,t] \times A) := {\int_{[0,\frac{\lfl mt \rfl}{m}]\times A} N_{ij}(ds\,dy\,dz)}, \label{eq:Ntil}
	\end{equation}	
	for $t \in [0,T]$, $A \subset \Bmc(\Zmb \times \Rmb_+)$ and $i,j \in \Nmb$.
	Consider an approximation system that is driven by $\{\Ntil_i^m, \Ntil_{ij}^m : i,j\in\Nmb\}$ as follows:
	\begin{align}
		\Xtil_i^m(t) & = X_i(0) + \int_{\Xmb_t} y\one_{[0,\Gamma(y,\Xtil_i^m(s-),\nutil_i^m(s-))]}(z) \, \Ntil_i^m(ds\,dy\,dz), \label{eq:Xtilm} \\
		\xitil_{ij}^m(t) & = \xi_{ij}(0) + \int_{\Xmb_t} y\one_{[0,\beta\Gammatil(y,\xitil_{ij}^m(s-),\Xtil_i^m(s-),\Xtil_j^m(s-))]}(z) \, \Ntil_{ij}^m(ds\,dy\,dz), \label{eq:xitilm} \\
		\nutil_i^m(t) & = \lim_{n \to \infty} \frac{1}{n} \sum_{j=1}^n \delta_{(\Xtil_j^m(t),\xitil_{ij}^m(t))} = \lim_{n \to \infty} \frac{1}{n} \sum_{j=1}^n \delta_{(\Xtil_j^m(\frac{\lfl mt \rfl}{m}),\xitil_{ij}^m(\frac{\lfl mt \rfl}{m}))} \notag \\
		& = \lim_{n \to \infty} \frac{1}{n} \sum_{j=1}^n \delta_{(\Xtil_j^m(\frac{k}{m}),\xitil_{ij}^m(\frac{k}{m}))} \quad \mbox{ for } \frac{k}{m} \le t < \frac{k+1}{m}, \quad k \in \Nmb_0. \notag
	\end{align}
	Note that the system is piece-wise constant and determined recursively over intervals of length $\frac{1}{m}$.
	We claim that the following holds for each $k \in \Nmb_0$:
		There exists a unique solution $\{(\Xtil_i^m[\frac{k}{m}],\xitil_{ij}^m[\frac{k}{m}],\nutil_i^m[\frac{k}{m}]) : i,j \in \Nmb \}$, which is exchangeable, namely for each $K \in \Nmb$ and permutation $\pi$ on $\{1,\dotsc,K\}$,
		\begin{equation*}
			\{(\Xtil_i^m[\frac{k}{m}],\xitil_{ij}^m[\frac{k}{m}],\nutil_i^m[\frac{k}{m}]) : i,j \in \Nmb \} \stackrel{d}{=} \{(\Xtil_{\pi(i)}^m[\frac{k}{m}],\xitil_{\pi(i)\pi(j)}^m[\frac{k}{m}],\nutil_{\pi(i)}^m[\frac{k}{m}]) : i,j \in \Nmb \}.
		\end{equation*}
	Note that we are not claiming the exchangeability of $(\xitil_{ij}^m : j \in \Nmb)$ for each fixed $i$, as the evolution of $\xitil_{ii}^m$ is different from that of other $\xitil_{ij}^m$.	
	Now we prove this claim by induction.
	Since $(X_i(0))$ and $(\xi_{ij}(0))$ are all i.i.d., the claim holds for $k=0$.
	Now assume the claim holds for some $k \in \Nmb_0$.
	Since the solution is linear on $[\frac{k}{m},\frac{k+1}{m})$, we have the existence and uniqueness of $\{(\Xtil_i^m[\frac{k+1}{m}],\xitil_{ij}^m[\frac{k+1}{m}]) : i,j \in \Nmb \}$ {by \eqref{eq:weaker-assumption}}, and the exchangeability of $\{(\Xtil_i^m[\frac{k+1}{m}],\xitil_{ij}^m[\frac{k+1}{m}],\nutil_i^m[\frac{k}{m}]) : i,j \in \Nmb \}$.
	This further implies the existence and uniqueness of $\nutil_i^m(\frac{k+1}{m})$ by de Finetti's theorem (cf.\ \cite[Theorem 4.1]{KotelenezKurtz2008}, see also \cite[Theorem 3.1]{Aldous1985exchangeability}), and the exchangeability of $\{(\Xtil_i^m[\frac{k+1}{m}],\xitil_{ij}^m[\frac{k+1}{m}],\nutil_i^m[\frac{k+1}{m}]) : i,j \in \Nmb \}$.
	Therefore the claim holds for $k+1$ and hence holds for each $k \in \Nmb_0$ by induction.
	
	Using {\eqref{eq:weaker-assumption}, Condition \ref{cond:no_acceleration}(ii)}, the evolution in \eqref{eq:Xtilm} and \eqref{eq:xitilm}, and Gronwall's inequality, one can show that
	\begin{equation*}
		\Emb \|\Xtil_i^m\|_{*,T} + \Emb \|\xitil_{ij}^m\|_{*,T} \le \kappa_1, \quad \forall \, i,j \in \Nmb,	
	\end{equation*}
	which implies the tightness of $\{(\Xtil_i^m(t),\xitil_{ij}^m(t)):i,j \in \Nmb\}$ in $\Zmb^\infty$ for each $t \in [0,T]$.
	Moreover, for the fluctuations, one can easily verify that
	\begin{align}
		\Emb |\Xtil_i^m(\tau+\delta)-\Xtil_i^m(\tau)| \le C_\gamma \left(\delta+\frac{1}{m}\right), & \quad \Emb |\xitil_{ij}^m(\tau+\delta)-\xitil_{ij}^m(\tau)| \le \beta C_\gamma \left(\delta+\frac{1}{m}\right), \label{eq:fluctuation_est_1}
	\end{align}
	for each $\delta \in (0,1)$ and $\Fmc_t$-stopping times $\tau$ with $\tau \in [0,T-\delta]$ a.s.
	Therefore the sequence of $\{(\Xtil_i^m,\xitil_{ij}^m):i,j \in \Nmb\}$ is tight in $\Dmb([0,T]:\Zmb^\infty)$ by applying Aldous' tightness criterion \cite[Theorem 2.7]{Kurtz1981approximation} with $m \to \infty$ and then $\delta \to 0$.
	
	Taking a subsequence if necessary, we assume that $\{(\Xtil_i^m,\xitil_{ij}^m):i,j \in \Nmb\} \Rightarrow \{(\Xtil_i,\xitil_{ij}):i,j \in \Nmb\}$, defined on some probability space $(\Omegatil,\Fmctil,\Pmbtil)$, as $m \to \infty$.
	Since $\{(\Xtil_j^m,\xitil_{ij}^m):j \in \Nmb, j \ne i\}$ is exchangeable, so is $\{(\Xtil_j,\xitil_{ij}):j \in \Nmb, j \ne i\}$, and we can define
	\begin{equation}
		\label{eq:nutil}
		\nutil_i(t) := \lim_{n \to \infty} \frac{1}{n} \sum_{j=1}^n \delta_{(\Xtil_j(t),\xitil_{ij}(t))}.
	\end{equation}
	It then follows from \cite[Lemma 2.1]{KurtzXiong1999} (see also \cite[Lemma 4.2]{KotelenezKurtz2008}), and appealing to the Skorokhod representation theorem, that
	\begin{align}
		& \{(\Xtil_i^m(t_k),\xitil_{ij}^m(t_k),\nutil_i^m(t_k)):i,j \in \Nmb, k \in [K]\} \notag \\
		& \quad \Rightarrow \{(\Xtil_i(t_k),\xitil_{ij}(t_k),\nutil_i(t_k)):i,j \in \Nmb, k \in [K]\} \label{eq:no_acc_nutil}
	\end{align}
	in $\Zmb^\infty \times \Pmc(\Zmb^4)^\infty$ as $m \to \infty$, for each $K \in \Nmb$ and $t_1,\dotsc,t_K \in T_C \subset [0,T]$, where $[0,T] \setminus T_C$ is at most countable.
	Let
	\begin{equation*}
		\widetilde\Xi_i^m := \lim_{n \to \infty} \frac{1}{n} \sum_{j=1}^n \delta_{(\Xtil_j^m(\cdot),\xitil_{ij}^m(\cdot))}.
	\end{equation*}
	Denoting by $\bar\rho$ the Prohorov metric on $\Pmc(\Zmb^2)$, we have
	\begin{align}
		\Emb \bar\rho(\nutil_i^m(\tau+\delta),\nutil_i^m(\tau)) 
		& \le \Emb \left[ \widetilde\Xi_i^m\{x \in \Dmb([0,T]:\Zmb^2):|x(\tau+\delta)-x(\tau)| \ge \varepsilon\} \right] + \varepsilon \notag \\
		& = \Pmb(|(\Xtil_j^m(\tau+\delta),\xitil_{ij}^m(\tau+\delta)) - (\Xtil_j^m(\tau),\xitil_{ij}^m(\tau))| \ge \varepsilon) + \varepsilon \notag \\
		& \le \frac{(1+\beta)C_\gamma (\delta+\frac{1}{m})}{\varepsilon} + \varepsilon \label{eq:fluctuation_Theta}
	\end{align}
	for each $\varepsilon, \delta \in (0,1)$ and $\Fmc_t$-stopping times $\tau$ with $\tau \in [0,T-\delta]$ a.s., where the last line uses \eqref{eq:fluctuation_est_1}.
	Therefore $\{(\Xtil_i^m,\xitil_{ij}^m,\nutil_i^m):i,j \in \Nmb\}$ is tight in $\Dmb([0,T]:\Zmb^\infty \times (\Pmc(\Zmb^2))^\infty)$ by applying Aldous' tightness criterion \cite[Theorem 2.7]{Kurtz1981approximation} with $m \to \infty$, $\delta \to 0$ and then $\varepsilon \to 0$.
	By \eqref{eq:no_acc_nutil}, the finite dimensional distributions converge for time instants in a dense subset set of $[0,T]$.
	So 
	\begin{equation}
		\label{eq:no_acc_cvg0}
		\{(\Xtil_i^m,\xitil_{ij}^m,\nutil_i^m):i,j \in \Nmb\} \Rightarrow \{(\Xtil_i,\xitil_{ij},\nutil_i):i,j \in \Nmb\}
	\end{equation}
	in $\Dmb([0,T]:\Zmb^\infty \times (\Pmc(\Zmb^2))^\infty)$.

	Next, we will verify that $\{(\Xtil_i,\xitil_{ij},\nutil_i):i,j \in \Nmb\}$ satisfies \eqref{eq:no_acceleration_system}. 
	We will need the weak convergence of stochastic integrals with respect to Poisson random measures later, and the following notations from \cite{KurtzProtter1996weak2}.
	Let $\Hmb := {L^1}(\Zmb\times\Rmb_+,\measurea(dy)\,dz)$.
	Given a Polish space $\Smb$, a collection of $\Smb$-valued stochastic processes $S^m$, $S$, and Poisson random measures $\{Y^m_i(ds\,dy\,dz):i \in \Nmb\}$ and $\{Y_i(ds\,dy\,dz):i \in \Nmb\}$ viewed as $\Hmb^\#$-semimartingales (see \cite[Section 3.3]{KurtzProtter1996weak2} for the precise definition), we say that $(S^m,\{Y^m_i:i\in\Nmb\}) \Rightarrow (S,\{Y_i:i\in\Nmb\})$ in $\Dmb([0,T]:\Smb) \otimes \Hmb^\#$ if 
	\begin{align*}
		& \left(S^m, \left\{ \int_{\Xmb_\cdot} \varphi_k(y,z) \, Y^m_i(ds\,dy\,dz) : i \in \Nmb, k \in [K]\right\} \right) \\
		& \quad \Rightarrow \left(S, \left\{ \int_{\Xmb_\cdot} \varphi_k(y,z) \, Y_i(ds\,dy\,dz) : i \in \Nmb, k \in [K]\right\} \right)
	\end{align*}
	in $\Dmb([0,T]:\Smb \times \Rmb^\infty)$ for each $K \in \Nmb$ and $\varphi_1,\dotsc,\varphi_K \in \Hmb$.
	Now take a dense sequence $\{\varphi_k\} \subset \Hmb$ and recall $\Ntil_i^m$ and $\Ntil_{ij}^m$ defined in \eqref{eq:Ntil}.
	{For each fixed $t \in [0,T]$, since
	\begin{align*}
		& \Emb \left|\int_{\Xmb_t} \varphi_k(y,z) \Ntil_i^m(ds\,dy\,dz)\right| + \Emb \left|\int_{\Xmb_t} \varphi_k(y,z) \Ntil_{ij}^m(ds\,dy\,dz)\right| \\
		& \le 2 \int_{\Xmb_t} |\varphi_k(y,z)| \,ds\,\measurea(dy)\,dz \le 2t \|\varphi_k\|_\Hmb, \quad \forall \, i,j \in \Nmb, k \in [K],
	\end{align*}
	we have that the sequence} 
	$$\left\{ \left(\Xtil_i^m(t),\xitil_{ij}^m(t),\nutil_i^m(t), \int_{\Xmb_t} \varphi_k(y,z) \Ntil_i^m(ds\,dy\,dz), \int_{\Xmb_t} \varphi_k(y,z) \Ntil_{ij}^m(ds\,dy\,dz)\right) : i,j \in \Nmb, {k \in [K]} \right\}$$
	is tight in $\Zmb^\infty \times (\Pmc(\Zmb^2))^\infty \times \Rmb^\infty$.
	Also note that
	\begin{align*}
		& \Emb \left| \int_{[0,\tau+\delta] \times \Zmb \times \Rmb_+} \varphi_k(y,z) \, \Ntil_i^m(ds\,dy\,dz) - \int_{[0,\tau] \times \Zmb \times \Rmb_+} \varphi_k(y,z) \, \Ntil_i^m(ds\,dy\,dz) \right| \\
		& \quad \le \left(\delta+\frac{1}{m}\right) \int_{\Zmb \times \Rmb_+} |\varphi_k(y,z)| \, \measurea(dy)\,dz = \left(\delta+\frac{1}{m}\right) \|\varphi_k\|_\Hmb, \\
		& \Emb \left| \int_{[0,\tau+\delta] \times \Zmb \times \Rmb_+} \varphi_k(y,z) \, \Ntil_{ij}^m(ds\,dy\,dz) - \int_{[0,\tau] \times \Zmb \times \Rmb_+} \varphi_k(y,z) \, \Ntil_{ij}^m(ds\,dy\,dz) \right| \\
		& \quad \le \left(\delta+\frac{1}{m}\right) \|\varphi_k\|_\Hmb,
	\end{align*}
	for each $\delta \in (0,1)$ and $\Fmc_t$-stopping times $\tau$ with $\tau \in [0,T-\delta]$ a.s.
	Combining this with \eqref{eq:fluctuation_est_1} and \eqref{eq:fluctuation_Theta} implies that
	$$\left\{ \left(\Xtil_i^m,\xitil_{ij}^m,\nutil_i^m, \int_{\Xmb_\cdot} \varphi_k(y,z) \Ntil_i^m(ds\,dy\,dz), \int_{\Xmb_\cdot} \varphi_k(y,z) \Ntil_{ij}^m(ds\,dy\,dz)\right) : i,j \in \Nmb, {k \in [K]} \right\}$$
	is tight in $\Dmb([0,T]:\Zmb^\infty \times (\Pmc(\Zmb^2))^\infty \times \Rmb^\infty)$, once again by Aldous' tightness criterion \cite[Theorem 2.7]{Kurtz1981approximation} (by taking $m \to \infty$, $\delta \to 0$ and then $\varepsilon \to 0$).
	Therefore $\{(\Xtil_i^m,\xitil_{ij}^m,\nutil_i^m,\Ntil_i^m,\Ntil_{ij}^m):i,j \in \Nmb\}$ is tight in $\Dmb([0,T]:\Zmb^\infty \times (\Pmc(\Zmb^2))^\infty) \otimes \Hmb^\#$.
	From this, \eqref{eq:no_acc_cvg0} and \eqref{eq:Ntil} we have that, taking a subsequence if necessary,
	\begin{equation}
		\label{eq:no_acc_joint_cvg}
		\{(\Xtil_i^m,\xitil_{ij}^m,\nutil_i^m,\Ntil_i^m,\Ntil_{ij}^m):i,j \in \Nmb\} \Rightarrow \{(\Xtil_i,\xitil_{ij},\nutil_i,\Ntil_i,\Ntil_{ij}):i,j \in \Nmb\}
	\end{equation}
	defined again on the probability space $(\Omegatil,\Fmctil,\Pmbtil)$, 
	in $\Dmb([0,T]:\Zmb^\infty \times (\Pmc(\Zmb^2))^\infty) \otimes \Hmb^\#$,
	where {$\{\Ntil_i,\Ntil_{ij} : i,j \in \Nmb\}$ are i.i.d.\ PRM on $\Xmb_T$ with intensity measure $ds \times \measurea(dy) \times dz$.}

	Noting that the maps
	\begin{align*}
		\Zmb \times \Pmc(\Zmb^2) \ni (x,\nu) \mapsto y\one_{[0,\Gamma(y,x,\nu)]}(z) \in \Hmb, \quad \Zmb^3 \ni (\xi,x,x') \mapsto y\one_{[0,\beta\Gammatil(y,\xi,x,x')]}(z) \in \Hmb
	\end{align*}
	are continuous, from \eqref{eq:no_acc_joint_cvg} and the continuous mapping theorem we have 
	\begin{align*}
		y\one_{[0,\Gamma(y,\Xtil_i^m(\cdot),\nutil_i^m(\cdot))]}(z) & \Rightarrow y\one_{[0,\Gamma(y,\Xtil_i(\cdot),\nutil_i(\cdot))]}(z) \mbox{ in } \Dmb([0,T]:\Hmb), \\
		y\one_{[0,\beta\Gammatil(y,\xitil_{ij}^m(\cdot),\Xtil_i^m(\cdot),\Xtil_j^m(\cdot))]}(z) & \Rightarrow y\one_{[0,\beta\Gammatil(y,\xitil_{ij}(\cdot),\Xtil_i(\cdot),\Xtil_j(\cdot))]}(z) \mbox{ in } \Dmb([0,T]:\Hmb),
	\end{align*}
	jointly with the convergence in \eqref{eq:no_acc_joint_cvg}.
	Also note that for each $m \in \Nmb$ and stochastic process $Z \in \Dmb([0,T]:\Hmb)$ with $\sup_{0 \le t \le T} \|Z(t)\|_\Hmb \le 1$, we have
	\begin{align*}
		& \Emb \left\| \int_{\Xmb_\cdot} Z(s,y,z)\, \Ntil_i^m(ds\,dy\,dz) \right\|_{*,T} + \Emb \left\| \int_{\Xmb_\cdot} Z(s,y,z)\, \Ntil_{ij}^m(ds\,dy\,dz) \right\|_{*,T} \\
		& \le 2 \Emb \int_{\Xmb_T} |Z(s,y,z)| \,ds\,\measurea(dy)\,dz \le 2T.
	\end{align*}
	It then follows from the convergence of stochastic integrals (cf.\ \cite[Theorem 4.2]{KurtzProtter1996weak2}), \eqref{eq:Xtilm}, \eqref{eq:xitilm}, and \eqref{eq:nutil} that $\{(\Xtil_i,\xitil_{ij},\nutil_i):i,j \in \Nmb\}$ is a solution to the limiting system \eqref{eq:no_acceleration_system}.
	{This gives the weak existence in part (a).}
	
	(b) We first prove pathwise uniqueness.
	Suppose $\{(X_i,\xi_{ij}, \nu_i):i,j\in\Nmb\}$ and $\{(\Xtil_i,\xitil_{ij}, \nutil_i):i,j\in\Nmb\}$ are two solutions of \eqref{eq:no_acceleration_system} with $X_i(0)=\Xtil(0)$ and $\xi_{ij}(0)=\xitil_{ij}(0)$ for $i,j \in \Nmb$.
	Using Condition \ref{cond:no_acceleration} and Gronwall's inequality, one can show that
	\begin{equation*}
		\Emb \left[ \|X_i\|_{*,T} + \|\xi_{ij}\|_{*,T} + \|\Xtil_i\|_{*,T} + \|\xitil_{ij}\|_{*,T} \right] \le \kappa_2, \quad \forall \, i,j \in \Nmb.	
	\end{equation*}
	By adding and subtracting terms, for $t \in [0,T]$,
	\begin{align*}
		\Emb \|X_i-\Xtil_i\|_{*,t} & \le \Emb \int_{\Xmb_t} |y| \left| \one_{[0,\Gamma(y,X_i(s-),\nu_i(s-))]}(z) - \one_{[0,\Gamma(y,\Xtil_i(s-),\nutil_i(s-))]}(z) \right| N_i(ds\,dy\,dz) \\
		& \le \Emb \int_{[0,t]\times\Zmb} |y| \left| \Gamma(y,X_i(s),\nu_i(s)) - \Gamma(y,\Xtil_i(s),\nutil_i(s)) \right| ds\,\measurea(dy) \\
		& \le \Emb \int_{[0,t]\times\Zmb} |y| \left| \lan \gamma(y,X_i(s),\cdot), \nu_i(s) - \nutil_i(s) \ran \right| ds\,\measurea(dy) \\
		& \quad + \Emb \int_{[0,t]\times\Zmb} |y| \left| \lan \gamma(y,X_i(s),\cdot) - \gamma(y,\Xtil_i(s),\cdot), \nutil_i(s) \ran \right| ds\,\measurea(dy) \\
		& \le \int_{[0,t]\times\Zmb} |y| \gamma_y \left( \Emb \dBL(\nu_i(s), \nutil_i(s)) + \Emb \left| X_i(s)-\Xtil_i(s) \right| \right) ds\,\measurea(dy) \\
		& \le C_\gamma \int_0^t \left( \Emb \left[ \sup_{u \in [0,s]} \dBL(\nu_i(u), \nutil_i(u)) \right] + \Emb \|X_i-\Xtil_i\|_{*,s} \right) ds,
	\end{align*}
	where the fifth line uses the Lipschitz property of $\gamma$ guaranteed by Condition \ref{cond:no_acceleration}(i) and Remark \ref{rmk:Lipschitz}(b).
	Similarly,
	\begin{align*}
		& \Emb \|\xi_{ij}-\xitil_{ij}\|_{*,t} \\
		& \le \Emb \int_{\Xmb_t} |y| \left| \one_{[0,\beta\Gammatil(y,\xi_{ij}(s-),X_i(s-),X_j(s-))]}(z)  - \one_{[0,\beta\Gammatil(y,\xitil_{ij}(s-),\Xtil_i(s-),\Xtil_j(s-))]}(z) \right| N_{ij}(ds\,dy\,dz) \\
		& \le \Emb \int_{[0,t]\times\Zmb} |y| \beta\left| \Gammatil(y,\xi_{ij}(s),X_i(s),X_j(s)) - \Gammatil(y,\xitil_{ij}(s),\Xtil_i(s),\Xtil_j(s)) \right| ds\,\measurea(dy) \\
		& \le \int_{[0,t]\times\Zmb} |y| \beta\gamma_y \left( \Emb \left| \xi_{ij}(s)-\xitil_{ij}(s) \right| + \Emb \left| X_i(s)-\Xtil_i(s) \right| + \Emb \left| X_j(s)-\Xtil_j(s) \right| \right) ds\,\measurea(dy) \\
		& \le \beta C_\gamma \int_0^t \left( \Emb \|\xi_{ij}-\xitil_{ij}\|_{*,s} + \Emb \|X_i-\Xtil_i\|_{*,s} + \Emb \|X_j-\Xtil_j\|_{*,s} \right) ds,
	\end{align*}	
	where the fourth line uses the Lipschitz property of $\Gammatil$ guaranteed by Condition \ref{cond:no_acceleration}(i) and Remark \ref{rmk:Lipschitz}(b).
	Also by Fatou's lemma,
	\begin{align*}
		\Emb \left[ \sup_{s \in [0,t]} \dBL(\nu_i(s), \nutil_i(s)) \right] & = \Emb \left[ \sup_{s \in [0,t]} \sup_{f \in \Bmb_1} |\lan f,\nu_i(s)- \nutil_i(s) \ran| \right] \\
		& \le \liminf_{n \to \infty} \frac{1}{n} \sum_{j=1}^n \left( \Emb \|X_j-\Xtil_j\|_{*,t} + \Emb \|\xi_{ij}-\xitil_{ij}\|_{*,t} \right) \\
		& \le \sup_{i,j \in \Nmb} \left( \Emb \|X_j-\Xtil_j\|_{*,t} + \Emb \|\xi_{ij}-\xitil_{ij}\|_{*,t} \right).
	\end{align*}
	Combining these three estimates gives
	\begin{align*}
		&  \sup_{i,j \in \Nmb} \left( \Emb \|X_j-\Xtil_j\|_{*,t} + \Emb \|\xi_{ij}-\xitil_{ij}\|_{*,t} \right) \\
		&  \quad \le 2(1+\beta)C_\gamma \int_0^t \sup_{i,j \in \Nmb} \left( \Emb \|X_j-\Xtil_j\|_{*,s} + \Emb \|\xi_{ij}-\xitil_{ij}\|_{*,s} \right) ds.
	\end{align*}
	The pathwise uniqueness then follows from Gronwall's lemma.

	{Next, using the pathwise uniqueness and the weak existence in part (a), it follows from the Yamada--Watanabe theorem (\cite[Theorem 1.5]{Kurtz2014weak})} that there exists a pathwise solution of \eqref{eq:no_acceleration_system}.
	
	{Finally}, recall the discretized system $(\Xtil_i^m, \xitil_{ij}^m, \nutil_i^m)$ given in \eqref{eq:Xtilm} and \eqref{eq:xitilm}.
	We claim that the following holds for each $k \in \Nmb_0$:
	\begin{enumerate}[(i)]
	\item $\{\Xtil_i^m[\frac{k}{m}] : i \in \Nmb\}$ are i.i.d.;
	\item $\{\Xtil_i^m[\frac{k}{m}] : i \in \Nmb\}$ are independent of $\{(\xi_{ij}(0),\Ntil_{ij}^m) :i,j \in \Nmb\}$;
	\item $\{(\Xtil_j^m[\frac{k}{m}],\xitil_{ij}^m[\frac{k}{m}],\Ntil_{ij}^m) : j \in \Nmb, j \ne i\}$ are i.i.d.\ conditioning on $\Xtil_i^m[\frac{k}{m}]$, and the conditional law $\Lmc(\{(\Xtil_j^m[\frac{k}{m}],\xitil_{ij}^m[\frac{k}{m}],\Ntil_{ij}^m) : j \in \Nmb, j \ne i\} \,|\, \Xtil_i^m[\frac{k}{m}]) = \Phi_{k,m}(\Xtil_i^m[\frac{k}{m}])$ for some measurable map $\Phi_{k,m} \colon \Dmb([0,\frac{k}{m}]:\Zmb) \to \Pmc((\Dmb([0,\frac{k}{m}]:\Zmb) \times \Dmb([0,\frac{k}{m}]:\Zmb) \times \Mmc_T)^\infty)$ independent of the choice of $i$.
	\end{enumerate}
	Again, we will prove this by induction.
	Since $(X_i(0))$ and $(\xi_{ij}(0))$ are all i.i.d., the claim holds for $k=0$.
	Now assume the claim holds for some $k \in \Nmb_0$.	
	From (iii) we have
	$$\nutil_i^m(\frac{k}{m}) = \Lmc((\Xtil_j^m(\frac{k}{m}),\xitil_{ij}^m(\frac{k}{m})) \,|\, \Xtil_i^m[\frac{k}{m}]) = \widetilde\Phi_{k,m}(\Xtil_i^m[\frac{k}{m}])$$
	for some measurable map $\widetilde\Phi_{k,m} \colon \Dmb([0,\frac{k}{m}]:\Zmb) \to \Pmc(\Zmb^2)$ independent of the choice of $i$.
	From (i), (ii) and the evolution of $\Xtil_i^m$ we see that $\{\Xtil_i^m[\frac{k+1}{m}] : i \in \Nmb\}$ are i.i.d., and independent of $\{(\xi_{ij}(0),\Ntil_{ij}^m) : i,j \in \Nmb\}$.
	It then follows from the evolution of $\xitil_{ij}^m$ that $\{(\Xtil_j^m[\frac{k+1}{m}],\xitil_{ij}^m[\frac{k+1}{m}],\Ntil_{ij}^m) : j \in \Nmb, j \ne i\}$ are i.i.d.\ conditioning on $\Xtil_i^m[\frac{k+1}{m}]$, and the conditional law $\Lmc(\{(\Xtil_j^m[\frac{k+1}{m}],\xitil_{ij}^m[\frac{k+1}{m}],\Ntil_{ij}^m) : j \in \Nmb, j \ne i\} \,|\, \Xtil_i^m[\frac{k+1}{m}]) = \Phi_{k+1,m}(\Xtil_i^m[\frac{k+1}{m}])$ for some measurable map $\Phi_{k+1,m} \colon \Dmb([0,\frac{k+1}{m}]:\Zmb) \to \Pmc((\Dmb([0,\frac{k+1}{m}]:\Zmb) \times \Dmb([0,\frac{k+1}{m}]:\Zmb) \times \Mmc_T)^\infty)$ independent of the choice of $i$.
	Therefore the claim holds for $k+1$ and hence holds for each $k \in \Nmb_0$ by induction.
	
	Now from the convergence \eqref{eq:no_acc_joint_cvg} we have that $\{X_i : i \in \Nmb\}$ are i.i.d., and independent of $\{(\xi_{ij}(0),N_{ij}) : i,j \in \Nmb\}$.
	From the evolution of $\xi_{ij}$ we see that $\xi_{ij} = \Psi(\xi_{ij}(0),X_i,X_j,N_{ij})$ for some measurable map $\Psi$.
	Therefore $\{(X_j[t],\xi_{ij}[t]) : j \in \Nmb, j \ne i\}$ are i.i.d.\ conditioning on $X_i[t]$, and $\nu_i(t) = \Lmc((X_j(t),\xi_{ij}(t)) \,|\, X_i[t])= \Phi_t(X_i[t])$ for some measurable map $\Phi_t \colon \Dmb([0,t]:\Zmb) \to \Pmc(\Zmb^2)$ independent of the choice of $i$.
	This gives part (b) and completes the proof.
\end{proof}

\begin{Remark}
	\label{rmk:alternative-proof}
	
	The above proof of pathwise existence uses the weak existence shown under the weaker assumption \eqref{eq:weaker-assumption}.
	An alternative proof under Condition \ref{cond:no_acceleration}(i), which implies the Lipschitz property of $\gamma(y,\cdot)$ and $\Gammatil(y,\cdot)$, is to define $\nu_i(t) := \Lmc((X_j(t),\xi_{ij}(t)) \,|\, X_i(0),N_i)$ and use a classical and easier Picard iteration argument, with a similar calculation as for the above proof of pathwise uniqueness.	
	Using this argument and formulation, one could further show that $\nu_i(t) = \widetilde{\Phi}_t(X_i(0),N_i)$, $\xi_{ij} = \widetilde{\Psi}(\xi_{ij}(0),N_{ij},X_i,X_j)$ and $X_i = \widetilde{\Theta}(X_i(0),N_i)$ for suitable measurable maps $\widetilde{\Phi}_t, \widetilde{\Psi}, \widetilde{\Theta}$.
	From this one could get $\nu_i(t) = \lim_{n \to \infty} \frac{1}{n} \sum_{j=1}^n \delta_{(X_j(t),\xi_{ij}(t))}$ and Theorem \ref{thm:no_acceleration_1}(b) without using the auxiliary processes $\{\Xtil_i^m,\xitil_{ij}^m,\nutil_i^m\}$. 
\end{Remark}

\begin{proof}[Proof of Theorem \ref{thm:no_acceleration_2}]
	(a) For each fixed $i \in [n]$ and $t \in [0,T]$, we have
	\begin{align*}
		& \Emb \|X_i^n-X_i\|_{*,t} \\
		& \le \Emb \int_{\Xmb_t} |y| \left| \one_{[0,\Gamma(y,X_i^n(s-),\nu_i^n(s-))]}(z) - \one_{[0,\Gamma(y,X_i(s-),\nu_i(s-))]}(z) \right| N_i(ds\,dy\,dz) \\
		& = \Emb \int_{[0,t]\times\Zmb} |y| \left| \Gamma(y,X_i^n(s),\nu_i^n(s)) - \Gamma(y,X_i(s),\nu_i(s)) \right| ds\,\measurea(dy).
	\end{align*}
	Since $\gamma$ is Lipschitz by Condition \ref{cond:no_acceleration}(i), and $\{(X_j(s), \xi_{ij}(s)) : j \in \Nmb, j \ne i\}$ are conditionally independent given $X_i[s]$ with the conditional law $\Lmc((X_j(s), \xi_{ij}(s)) |X_i[s]) = \nu_i(s)$, for $j \ne i$, by Theorem \ref{thm:no_acceleration_1}(b), we have
	\begin{align*}
		& \Emb |\Gamma(y,X_i^n(s),\nu_i^n(s))-\Gamma(y,X_i(s),\nu_i(s))| \\
		& \le \Emb \left| \frac{1}{n} \sum_{j=1}^n \gamma(y,X_i^n(s),X_j^n(s),\xi_{ij}^n(s)) - \frac{1}{n} \sum_{j=1}^n \gamma(y,X_i(s),X_j(s),\xi_{ij}(s)) \right| \\
		& \qquad + \Emb \left| \frac{1}{n} \sum_{j=1}^n \gamma(y,X_i(s),X_j(s),\xi_{ij}(s)) - \int_{\Zmb^2} \gamma(y,X_i(s),\xtil,\xitil) \,\nu_i(s)(d\xtil\,d\xitil) \right| \\
		& \le \gamma_y \left( \Emb |X_i^n(s)-X_i(s)| + \frac{1}{n} \sum_{j=1}^n \Emb |X_j^n(s)-X_j(s)| + \frac{1}{n} \sum_{j=1}^n \Emb |\xi_{ij}^n(s)-\xi_{ij}(s)| + \frac{\kappa_1}{\sqrt{n}} \right) \\
		& \le \gamma_y \left( 2 \max_{i \in [n]} \Emb \|X_i^n-X_i\|_{*,s} + \max_{i,j \in [n]} \Emb \|\xi_{ij}^n-\xi_{ij}\|_{*,s} + \frac{\kappa_1}{\sqrt{n}} \right),
	\end{align*}
	where the term $\frac{\kappa_1}{\sqrt{n}}$ in the second inequality follows from Lemma \ref{lem:iid-moment}.
	Also since $\Gammatil$ is Lipschitz by Condition \ref{cond:no_acceleration}(i), we have
	\begin{align*}
		& \Emb \|\xi_{ij}^n-\xi_{ij}\|_{*,t} \\
		& \le \Emb \int_{\Xmb_t} |y| \left| \one_{[0,\beta(n)\Gammatil(y,\xi_{ij}^n(s-),X_i^n(s-),X_j^n(s-))]}(z) - \one_{[0,\beta\Gammatil(y,\xi_{ij}(s-),X_i(s-),X_j(s-))]}(z) \right| N_{ij}(ds\,dy\,dz) \\
		& = \Emb \int_{[0,t]\times\Zmb} |y| \left| \beta(n)\Gammatil(y,\xi_{ij}^n(s),X_i^n(s),X_j^n(s)) - \beta\Gammatil(y,\xi_{ij}(s),X_i(s),X_j(s)) \right| ds\,\measurea(dy) \\
		& \le \Emb \int_{[0,t]\times\Zmb} |y| \left[ \gamma_y |\beta(n)-\beta| \right. \\
		& \qquad \left. + \gamma_y \beta\left( |\xi_{ij}^n(s)-\xi_{ij}(s)| + |X_i^n(s)-X_i(s)| + |X_j^n(s)-X_j(s)| \right) \right] ds\,\measurea(dy) \\
		& \le C_\gamma |\beta(n)-\beta| t + C_\gamma \beta \int_0^t \max_{i,j \in [n]} \Emb \|\xi_{ij}^n-\xi_{ij}\|_{*,s} \,ds + 2C_\gamma \beta \int_0^t \max_{i \in [n]} \Emb \|X_i^n-X_i\|_{*,s} \,ds.
	\end{align*}
	Combining above three displays gives
	\begin{align*}
		& \max_{i \in [n]} \Emb \|X_i^n-X_i\|_{*,t} + \max_{i,j \in [n]} \Emb \|\xi_{ij}^n-\xi_{ij}\|_{*,t} \\
		& \quad \le 2C_\gamma (\beta+1) \int_0^t \left( \max_{i \in [n]} \Emb \|X_i^n-X_i\|_{*,s} + \max_{i,j \in [n]} \Emb \|\xi_{ij}^n-\xi_{ij}\|_{*,s} \right) ds + \frac{\kappa_1C_\gamma t}{\sqrt{n}} + C_\gamma |\beta(n)-\beta| t.
	\end{align*}
	From Gronwall's inequality we have \eqref{eq:no_acc_LLN1}.

	(b) Using (a), Theorem \ref{thm:no_acceleration_1}(b) and a standard argument (see \cite[Chapter 1]{Sznitman1991}) one has \eqref{eq:no_acceleration_POC}.

	(c) 
	Fix $t \in [0,T]$ and $i \in \Nmb$.
	First note that $\nu_i$ is well defined, thanks to the conditional i.i.d.\ property in Theorem \ref{thm:no_acceleration_1}(b) of $\{(X_j,\xi_{ij}) : j \in \Nmb, j \ne i\}$ given $X_i$.
	Let $$\nubar_i^n(t) := \frac{1}{n} \sum_{j=1}^n \delta_{(X_i(t),\xi_{ij}(t))}, \quad \nubar_i^n := \frac{1}{n} \sum_{j=1}^n \delta_{(X_i(\cdot),\xi_{ij}(\cdot))}.$$
	From \eqref{eq:no_acc_LLN1} we have
	\begin{align*}
		\Emb \dBL(\nu_i^n(t),\nubar_i^n(t)) & \le \frac{1}{n} \sum_{j=1}^n \Emb \left( |X_j^n(t)-X_j(t)| + |\xi_{ij}^n(t)-\xi_{ij}(t)| \right) \to 0, \\
		\Emb \dBL(\nu_i^n,\nubar_i^n) & \le \frac{1}{n} \sum_{j=1}^n \Emb \left( \|X_j^n-X_j\|_{*,T} + \|\xi_{ij}^n-\xi_{ij}\|_{*,T} \right) \to 0.
	\end{align*}
	Using the definition of $\nu_i(t)$ in \eqref{eq:no_acceleration_system} and $\nu_i$ in \eqref{eq:no_acc_LLN_nu_i}, we have
	$$\nubar_i^n(t) \to \nu_i(t), \quad \nubar_i^n \to \nu_i \text{ a.s.}$$
	Combining these gives \eqref{eq:no_acc_LLN_nu_i} and \eqref{eq:no_acc_LLN_nu_i_t}.
	The last two statements \eqref{eq:no_acc_LLN_mu_i} and \eqref{eq:no_acc_LLN_mu_i_t} follow immediately from \eqref{eq:no_acc_LLN_nu_i} and \eqref{eq:no_acc_LLN_nu_i_t}, respectively.
\end{proof}

\subsection{Proofs of Theorems \ref{thm:acc_LLN} and \ref{thm:comparison}}
\label{sec:pf-acc}

\begin{proof}[Proof of Theorem \ref{thm:acc_LLN}]	
	(a) Since the limiting system is McKean--Vlasov, the proof of existence and uniqueness is standard (cf.\ \cite[Chapter 1]{Sznitman1991}, see also \cite[Theorem 2.1]{Graham1992Mckean}) and hence omitted.

	(b) It would be helpful to freeze the slow components $X$ and analyze the averaging effect of the fast component $\xi$ first.
	Consider the following auxiliary process with $\Delta=\Delta(n)\to 0$ (whose precise value will be stated later):
	\begin{align*}
		X_i^{\Delta}(t) & = X_i\left(\lfl \frac{t}{\Delta} \rfl\Delta\right), \\
		\xi_{ij}^{n,\Delta}(t) & = \xi_{ij}^n\left(\lfl \frac{t}{\Delta} \rfl\Delta\right) + \int_{[\lfl \frac{t}{\Delta} \rfl\Delta,t]\times\Zmb\times\Rmb_+} y\one_{[0,\beta(n)\Gammatil(y,\xi_{ij}^{n,\Delta}(s-),X_i^{\Delta}(s-),X_j^{\Delta}(s-))]}(z) \, N_{ij}(ds\,dy\,dz).
	\end{align*}	
	Note that for each $t \in [0,T]$,
	\begin{align}
		\Emb \sup_{s \in [\lfl \frac{t}{\Delta} \rfl\Delta, t]} |X_i^{\Delta}(s)-X_i(s)| 
		& \le \Emb \int_{[\lfl \frac{t}{\Delta} \rfl\Delta, t]\times\Zmb\times\Rmb_+} |y|\one_{[0,\Gamma(y,X_i(s-),\nu_i(s-))]}(z) \, N_i(ds\,dy\,dz) \notag \\
		& \le \Delta \int_{\Rmb_+} |y|\gamma_y \,\measurea(dy) = C_\gamma \Delta \label{eq:acc_aux_1}
	\end{align}
	by Condition \ref{cond:no_acceleration}(i).
	From this we have
	\begin{align}
		\Pmb(\xi_{ij}^{n,\Delta}(t) \ne \xi_{ij}^n(t)) 
		& \le \Pmb((X_i^{\Delta}(s),X_j^{\Delta}(s)) \ne (X_i^n(s),X_j^n(s)) \text{ for some } s \in [\lfl \frac{t}{\Delta} \rfl\Delta, t]) \notag \\
		& \le 2\Emb \sup_{s \in [\lfl \frac{t}{\Delta} \rfl\Delta, t]} |X_i^{\Delta}(s)-X_i(s)| + 2\Emb \sup_{s \in [\lfl \frac{t}{\Delta} \rfl\Delta, t]} |X_i^{n}(s)-X_i(s)| \notag \\
		& \le 2C_\gamma \Delta + 2\Emb \|X_i^n-X_i\|_{*,t}. \label{eq:acc_aux_2}
	\end{align}
	Also note that from Condition \ref{cond:no_acceleration}(i) and Condition \ref{cond:acceleration}(ii) we have
	\begin{equation}
		\label{eq:moment-bd}
		\Emb \|X_i^{\Delta}\|_{*,T}^2 \le \Emb \|X_i\|_{*,T}^2 \le \kappa_1 (\Emb |X_i(0)|^2 + C_\gamma^2 + C_{\gamma,2}) < \infty.
	\end{equation}
	Let $$Q_{i,j}^{n,\Delta,y}(s) := \gamma(y,X_i^{\Delta}(s),X_j^{\Delta}(s),\xi_{ij}^{n,\Delta}(s)) - \int_\Zmb \gamma(y,X_i^{\Delta}(s),X_j^{\Delta}(s),\xitil) \,Q(X_i^{\Delta}(s),X_j^{\Delta}(s),d\xitil).$$
	Given $X_i\left(\lfl \frac{s}{\Delta} \rfl\Delta\right), X_j\left(\lfl \frac{s}{\Delta} \rfl\Delta\right), X_k\left(\lfl \frac{s}{\Delta} \rfl\Delta\right), \xi_{ij}^n\left(\lfl \frac{s}{\Delta} \rfl\Delta\right), \xi_{ik}^n\left(\lfl \frac{s}{\Delta} \rfl\Delta\right)$, clearly $Q_{i,j}^{n,\Delta,y}(s)$ and $Q_{i,k}^{n,\Delta,y}(s)$ are conditionally independent.
	Using this, \eqref{eq:moment-bd}, Condition \ref{cond:acceleration}(i), and an application of time change $s \mapsto \frac{s}{\beta(n)}$ (so that the evolution of $\xi_{ij}^{n,\Delta}$ matches that of $\Ytil$ introduced in Condition \ref{cond:acceleration}) we have
	\begin{align*}
		& \Emb \left[ Q_{i,j}^{n,\Delta,y}(s) Q_{i,k}^{n,\Delta,y}(s) \right] \\
		& = \Emb \left\{ \Emb \left[ Q_{i,j}^{n,\Delta,y}(s) \,\Big|\, X_i\left(\lfl \frac{s}{\Delta} \rfl\Delta\right), X_j\left(\lfl \frac{s}{\Delta} \rfl\Delta\right), X_k\left(\lfl \frac{s}{\Delta} \rfl\Delta\right), \xi_{ij}^n\left(\lfl \frac{s}{\Delta} \rfl\Delta\right), \xi_{ik}^n\left(\lfl \frac{s}{\Delta} \rfl\Delta\right) \right] \right. \\
		& \left. \qquad \cdot \Emb \left[ Q_{i,k}^{n,\Delta,y}(s) \,\Big|\, X_i\left(\lfl \frac{s}{\Delta} \rfl\Delta\right), X_j\left(\lfl \frac{s}{\Delta} \rfl\Delta\right), X_k\left(\lfl \frac{s}{\Delta} \rfl\Delta\right), \xi_{ij}^n\left(\lfl \frac{s}{\Delta} \rfl\Delta\right), \xi_{ik}^n\left(\lfl \frac{s}{\Delta} \rfl\Delta\right) \right] \right\} \\
		& \le \widetilde\kappa^2\left(\beta(n)\left(s-\lfl \frac{s}{\Delta} \rfl\Delta\right)\right) \gamma_y^2 \Emb \left[ \left( 1 + \|X_i^{\Delta}\|_{*,T} + \|X_j^{\Delta}\|_{*,T} \right) \left( 1 + \|X_i^{\Delta}\|_{*,T} + \|X_k^{\Delta}\|_{*,T} \right) \right] \\
		& \le \kappa_2 \gamma_y^2 \widetilde\kappa^2\left(\beta(n)\left(s-\lfl \frac{s}{\Delta} \rfl\Delta\right)\right) 
	\end{align*}
	for each $i,j,k \in [n]$ with $j \ne k$.
	Therefore
	\begin{align*}
		\Emb \left[ \left( \frac{1}{n} \sum_{j=1}^n Q_{i,j}^{n,\Delta,y}(s) \right)^2 \right] & = \frac{1}{n^2} \sum_{j=1}^n \sum_{k \ne j}^n \Emb \left[ Q_{i,j}^{n,\Delta,y}(s) Q_{i,k}^{n,\Delta,y}(s) \right] + \frac{1}{n^2} \sum_{j=1}^n \Emb \left[ \left( Q_{i,j}^{n,\Delta,y}(s) \right)^2 \right] \\
		& \le \kappa_2 \gamma_y^2 \widetilde\kappa^2\left(\beta(n)\left(s-\lfl \frac{s}{\Delta} \rfl\Delta\right)\right) + \frac{4\gamma_y^2}{n},
	\end{align*}
	and hence
	\begin{align}
		\int_{(k-1)\Delta}^{k\Delta} \Emb \left|\frac{1}{n} \sum_{j=1}^n Q_{i,j}^{n,\Delta,y}(s) \right| ds & \le \int_{(k-1)\Delta}^{k\Delta} \left( \sqrt{\kappa_2} \gamma_y \widetilde\kappa\left(\beta(n)\left(s-\lfl \frac{s}{\Delta} \rfl\Delta\right)\right) + \frac{2\gamma_y}{\sqrt{n}} \right) ds \notag \\
		& = \frac{\sqrt{\kappa_2}\gamma_y}{\beta(n)} \int_0^{\beta(n)\Delta} \widetilde\kappa(s)\,ds + \frac{2\Delta\gamma_y}{\sqrt{n}} \notag \\
		& \le \frac{\kappa_3\gamma_y}{\beta(n)} + \frac{2\Delta\gamma_y}{\sqrt{n}} \label{eq:exp_ergodic}
	\end{align}
	for each $k \in \Nmb$.

	Now we show \eqref{eq:acc_LLN1}.
	Note that
	\begin{align*}
		& \Emb \|X_i^n-X_i\|_{*,t} \\
		& \le \Emb \int_{\Xmb_t} |y| \left| \one_{[0,\Gamma(y,X_i^n(s-),\nu_i^n(s-))]}(z) - \one_{[0,\Gamma(y,X_i(s-),\nu_i(s-))]}(z) \right| N_i(ds\,dy\,dz) \\
		& = \Emb \int_{[0,t]\times\Zmb} |y| \left| \Gamma(y,X_i^n(s),\nu_i^n(s)) - \Gamma(y,X_i(s),\nu_i(s)) \right| ds\,\measurea(dy).
	\end{align*}
	Fixing $y \in \Nmb$ and $s \in [0,T]$, we have 
	\begin{align}
		& \Emb |\Gamma(y,X_i^n(s),\nu_i^n(s))-\Gamma(y,X_i(s),\nu_i(s))| \notag \\
		& = \Emb \left| \frac{1}{n} \sum_{j=1}^n \gamma(y,X_i^n(s),X_j^n(s),\xi_{ij}^n(s))-\int_{\Zmb^2} \gamma(y,X_i(s),\xtil,\xitil) \,\mu_s(d\xtil)\,Q(X_i(s),\xtil,d\xitil) \right| \notag \\
		& \le \frac{1}{n} \sum_{j=1}^n \Emb \left| \gamma(y,X_i^n(s),X_j^n(s),\xi_{ij}^n(s)) - \gamma(y,X_i^{\Delta}(s),X_j^{\Delta}(s),\xi_{ij}^{n,\Delta}(s)) \right| \notag \\
		& \quad + \Emb \left| \frac{1}{n} \sum_{j=1}^n \left( \gamma(y,X_i^{\Delta}(s),X_j^{\Delta}(s),\xi_{ij}^{n,\Delta}(s)) - \int_\Zmb \gamma(y,X_i^{\Delta}(s),X_j^{\Delta}(s),\xitil) \,Q(X_i^{\Delta}(s),X_j^{\Delta}(s),d\xitil) \right) \right| \notag \\
		& \quad + \frac{1}{n} \sum_{j=1}^n \Emb \left| \int_\Zmb \gamma(y,X_i^{\Delta}(s),X_j^{\Delta}(s),\xitil) \,Q(X_i^{\Delta}(s),X_j^{\Delta}(s),d\xitil) \right. \notag \\
		& \qquad \left. - \int_\Zmb \gamma(y,X_i(s),X_j(s),\xitil) \,Q(X_i(s),X_j(s),d\xitil) \right| \notag \\
		& \quad + \Emb \left| \frac{1}{n} \sum_{j=1}^n \int_\Zmb \gamma(y,X_i(s),X_j(s),\xitil) \,Q(X_i(s),X_j(s),d\xitil) \right. \notag \\
		& \qquad \left. - \int_{\Zmb^2} \gamma(y,X_i(s),\xtil,\xitil) \,\mu_s(d\xtil)\,Q(X_i(s),\xtil,d\xitil) \right|. \label{eq:acc_pf_triangle}
	\end{align}

	Next we analyze each of the four terms on the RHS of \eqref{eq:acc_pf_triangle}.
	{For the first term, from Condition \ref{cond:no_acceleration}(i) and \eqref{eq:gamma-Lipschitz} we have}
	\begin{align*}
		& \Emb \left| \gamma(y,X_i^n(s),X_j^n(s),\xi_{ij}^n(s)) - \gamma(y,X_i^{\Delta}(s),X_j^{\Delta}(s),\xi_{ij}^{n,\Delta}(s)) \right| \\
		& \le \gamma_y \left( \Emb |X_i^n(s)-X_i(s)| + \Emb |X_i^{\Delta}(s)-X_i(s)| + \Emb |X_j^n(s)-X_j(s)| + \Emb |X_j^{\Delta}(s)-X_j(s)| \right. \\
		& \qquad \left. + \Pmb(\xi_{ij}^n(s) \ne \xi_{ij}^{n,\Delta}(s)) \right) \\
		& \le 4 \gamma_y \Emb \|X_i^n-X_i\|_{*,s} + 4\gamma_y C_\gamma \Delta,
	\end{align*}
	{where the last line uses the exchangeability of $\{(X_j^n,X_j^{\Delta},X_j) : j \in \Nmb\}$, \eqref{eq:acc_aux_1} and \eqref{eq:acc_aux_2}.}
	For the second term, {by the assumption $0 \le \gamma(y,\cdot) \le \gamma_y$ and using} \eqref{eq:exp_ergodic} we have
	\begin{align*}
		& \int_0^t \Emb \left| \frac{1}{n} \sum_{j=1}^n \left( \gamma(y,X_i^{\Delta}(s),X_j^{\Delta}(s),\xi_{ij}^{n,\Delta}(s)) - \int_\Zmb \gamma(y,X_i^{\Delta}(s),X_j^{\Delta}(s),\xitil) \,Q(X_i^{\Delta}(s),X_j^{\Delta}(s),d\xitil) \right) \right| ds \\
		& \le \sum_{k=1}^{\lfl \frac{t}{\Delta} \rfl} \int_{(k-1)\Delta}^{k\Delta} \Emb \left| \frac{1}{n} \sum_{j=1}^n Q_{i,j}^{n,\Delta,y}(s) \right| ds + \Delta \gamma_y \\
		& \le \kappa_3\frac{\gamma_yt}{\beta(n)\Delta} + \frac{2\gamma_yt}{\sqrt{n}} + \Delta \gamma_y.
	\end{align*}
	For the third term, {from Condition \ref{cond:no_acceleration}(i) we have
	$$\left| \int_\Zmb \gamma(y,x_1,x_2,\xitil) \,Q(x_1,x_2,d\xitil) - \int_\Zmb \gamma(y,x_1',x_2',\xitil) \,Q(x_1',x_2',d\xitil) \right| \le \gamma_y \one_{\{(x_1,x_2) \ne (x_1',x_2')\}}$$
	for each $y,x_1,x_2,x_1',x_2'\in\Zmb$, and hence}
	\begin{align*}
		& \Emb \left| \int_\Zmb \gamma(y,X_i^{\Delta}(s),X_j^{\Delta}(s),\xitil) \,Q(X_i^{\Delta}(s),X_j^{\Delta}(s),d\xitil) - \int_\Zmb \gamma(y,X_i(s),X_j(s),\xitil) \,Q(X_i(s),X_j(s),d\xitil) \right| \\
		& \le \gamma_y \left( \Pmb(X_i^{\Delta}(s) \ne X_i(s)) + \Pmb(X_j^{\Delta}(s) \ne X_j(s)) \right) \le 2\gamma_y C_\gamma \Delta,
	\end{align*}
	where the last inequality follows from \eqref{eq:acc_aux_1}.
	For the fourth term, since $X_j$ are i.i.d.\ with law $\mu$, from Lemma \ref{lem:iid-moment} we have
	\begin{align*}
		& \Emb \left| \frac{1}{n} \sum_{j=1}^n \int_\Zmb \gamma(y,X_i(s),X_j(s),\xitil) \,Q(X_i(s),X_j(s),d\xitil) - \int_{\Zmb^2} \gamma(y,X_i(s),\xtil,\xitil) \,\mu_s(d\xtil)\,Q(X_i(s),\xtil,d\xitil) \right| \\
		& \le \frac{\kappa_4\gamma_y}{\sqrt{n}}.
	\end{align*}
	
	Combining all of the above estimates and using Condition \ref{cond:no_acceleration}(i) gives
	\begin{equation*}
		\Emb \|X_i^n-X_i\|_{*,t} \le 4C_\gamma \int_0^t \Emb \|X_i^n-X_i\|_{*,s} \,ds + \kappa_3 \frac{C_\gamma t}{\beta(n) \Delta} + C_\gamma \Delta + 6 C_\gamma^2 t \Delta + \frac{(\kappa_4+2)C_\gamma t}{\sqrt{n}}.
	\end{equation*}
	Using Gronwall's inequality and taking $\Delta=\Delta(n) = \frac{1}{\sqrt{\beta(n)}}$, we have \eqref{eq:acc_LLN1}.
	
	(c) Using (b), the independence of $\{X_i\}$ and a standard argument (see \cite[Chapter 1]{Sznitman1991}) one has \eqref{eq:acceleration_POC}.
	
	(d) Fix $t \in [0,T]$.
	Let 
	$$\mubar^n := \frac{1}{n} \sum_{i=1}^n \delta_{X_i(\cdot)}, \quad \mubar^n(t) := \frac{1}{n} \sum_{i=1}^n \delta_{X_i(t)}.$$
	From \eqref{eq:acc_LLN1} we have
	\begin{align*}
		\Emb\dBL(\mu^n(t),\mubar^n(t)) & \le \frac{1}{n} \sum_{i=1}^n \Emb |X_i^n(t)-X_i(t)| \to 0, \\
		\Emb\dBL(\mu^n,\mubar^n)  &\le \frac{1}{n} \sum_{i=1}^n \Emb \|X_i^n-X_i\|_{*,T} \to 0.
	\end{align*}
	Also note that
	$$\mubar^n(t) \to \mu(t), \quad \mubar^n \to \mu,$$
	in probability by independence of $X_i$.
	Combining these gives \eqref{eq:acc_LLN2}.

	(e) Fix $t \in (0,T]$.
	Take $\Delta=\Delta(n)=O(\frac{1}{\sqrt{\beta(n)}})$ such that $(k+\frac{1}{2})\Delta \le t < (k+1)\Delta$ for some $k=k(n) \in \Nmb$.
	Define
	\begin{align*}
		\nu_i^{n,\Delta,1}(t) & =\frac{1}{n} \sum_{j=1}^n \delta_{(X_j^{\Delta}(t),\xi_{ij}^{n,\Delta}(t))}, \\
		\nu_i^{n,\Delta,2}(t) & =\frac{1}{n} \sum_{j=1}^n \delta_{X_j^{\Delta}(t)} \otimes \Lmc(\xi_{ij}^{n,\Delta}(t) \,|\, \xi_{ij}^n(k\Delta),X_i^{\Delta}(t),X_j^{\Delta}(t)), \\
		\nu_i^{n,\Delta,3}(t) & =\frac{1}{n} \sum_{j=1}^n \delta_{X_j^{\Delta}(t)} \otimes Q(X_i^{\Delta}(t),X_j^{\Delta}(t),\cdot), \\
		\nu_i^{n,\Delta,4}(t) & =\frac{1}{n} \sum_{j=1}^n \delta_{X_j(t)} \otimes Q(X_i(t),X_j(t),\cdot).
	\end{align*}
	From \eqref{eq:acc_LLN1}, \eqref{eq:acc_aux_1} and \eqref{eq:acc_aux_2} we have
	\begin{align*}
		\Emb \dBL(\nu_i^n(t),\nu_i^{n,\Delta,1}(t)) & \le \frac{1}{n} \sum_{j=1}^n \sup_{f \in \Bmb_1} \left| f(X_j^n(t),\xi_{ij}^n(t)) - f(X_j^{\Delta}(t),\xi_{ij}^{n,\Delta}(t)) \right| \\
		& \le \frac{1}{n} \sum_{j=1}^n \left( \Emb|X_j^n(t)-X_j^{\Delta}(t)| + 2\Pmb(\xi_{ij}^n(t) \ne \xi_{ij}^{n,\Delta}(t)) \right) \to 0, \\
		\Emb \dBL(\nu_i^{n,\Delta,3}(t),\nu_i^{n,\Delta,4}(t)) & \le \frac{1}{n} \sum_{j=1}^n \Pmb(X_j^\Delta(t) \ne X_j(t)) \to 0.
	\end{align*}
	Using the conditional independence of $\{\xi_{ij}^{n,\Delta}(t) : j \in [n]\}$ given $\{X_j^\Delta(t) : j \in [n]\}$, we have
	\begin{align*}
		& \Emb \left[ \left( \lan f,\nu_i^{n,\Delta,1}(t) \ran - \lan f,\nu_i^{n,\Delta,2}(t) \ran \right)^2 \right] \\
		& = \frac{1}{n^2} \sum_{j=1}^n \Emb \left[ \left( f(X_j^{\Delta}(t),\xi_{ij}^{n,\Delta}(t)) - \lan f,\delta_{X_j^{\Delta}(t)} \otimes \Lmc(\xi_{ij}^{n,\Delta}(t) \,|\, \xi_{ij}^n(k\Delta),X_i^{\Delta}(t),X_j^{\Delta}(t)) \ran \right)^2 \right] \\
		& \le \frac{\|f\|_\infty^2}{n} \to 0
	\end{align*}
	for each bounded and continuous function $f$.
	Similarly, using the independence of $\{X_i : i \in [n]\}$ and the weak law of large numbers we have 
	$$\nu_i^{n,\Delta,4}(t) \to \nu_i(t)$$
	in probability.
	Lastly, from \eqref{eq:acc_asump_addition} and the definition of $\xi_{ij}^{n,\Delta}$ we have
	\begin{align*}
		& \Emb \left| \lan f,\nu_i^{n,\Delta,2}(t) \ran - \lan f,\nu_i^{n,\Delta,3}(t) \ran \right| \\
		&  = \Emb \left| \frac{1}{n} \sum_{j=1}^n \left\lan f,\delta_{X_j^{\Delta}(t)} \otimes \Lmc(\xi_{ij}^{n,\Delta}(t) \,|\, \xi_{ij}^n(k\Delta),X_i^{\Delta}(t),X_j^{\Delta}(t)) - \delta_{X_j^{\Delta}(t)} \otimes Q(X_i^{\Delta}(t),X_j^{\Delta}(t),\cdot) \right\ran \right| \\
		& \le \frac{\|f\|_\infty}{n} \sum_{j=1}^n \Emb \left[ \sum_{\xitil \in \Zmb} \left|\Pmb(\xi_{ij}^{n,\Delta}(t)=\xitil \,|\, \xi_{ij}^n(k\Delta),X_i^{\Delta}(t),X_j^{\Delta}(t)) - Q(X_i^{\Delta}(t),X_j^{\Delta}(t),\{\xitil\})\right| \right] \\
		& \le \frac{\|f\|_\infty}{n} \sum_{j=1}^n \Emb \left[ C(\beta(n)(t-k\Delta)) (1+|X_i^{\Delta}(t)|+|X_j^{\Delta}(t)|) \right] \to 0 \\
		& \le \kappa_5 C(\beta(n)(t-k\Delta)) \to 0
	\end{align*}
	for each bounded and continuous function $f$.
	Combining above estimates gives \eqref{eq:acc_LLN3}.
	Finally, \eqref{eq:acc_LLN4} is a direct consequence of \eqref{eq:acc_LLN3}.
\end{proof}

\begin{Remark}
	\label{rmk:kappa}
	If we are only interested in the convergence of $X_i^n, \mu^n, \nu_i^n$, then it would be sufficient to assume that $\lim_{t \to \infty} \frac{1}{t}\int_0^t\widetilde\kappa(s)\,ds = 0$, instead of $\int_0^\infty \widetilde\kappa(t)\,dt < \infty$.
	We only have to replace $\kappa_3$ by $\kappa_3 \int_0^{\beta(n)\Delta} \widetilde\kappa(s)\,ds$ in \eqref{eq:exp_ergodic} and in what follows.
\end{Remark}

\begin{proof}[Proof of Theorem \ref{thm:comparison}]
	The proof is quite similar to that of Theorem \ref{thm:acc_LLN}, except that $\nu_i^n$ there is replaced by $\nu_i^\beta$ here.
	So we will omit certain common arguments.
	
	(a) First consider the following auxiliary process with $\Delta=\Delta(\beta) \to 0$ (whose precise value will be stated later):
	\begin{align*}
		X_i^{\Delta}(t) & = X_i\left(\lfl \frac{t}{\Delta} \rfl\Delta\right), \\
		\xi_{ij}^{\beta,\Delta}(t) & = \xi_{ij}^\beta\left(\lfl \frac{t}{\Delta} \rfl\Delta\right) + \int_{[\lfl \frac{t}{\Delta} \rfl\Delta,t]\times\Zmb\times\Rmb_+} y\one_{[0,\beta\Gammatil(y,\xi_{ij}^{\beta,\Delta}(s-),X_i^{\Delta}(s-),X_j^{\Delta}(s-))]}(z) \, N_{ij}(ds\,dy\,dz).
	\end{align*}
	Using Condition \ref{cond:no_acceleration}(i), one can check that 
	\begin{align}
		& \Emb \sup_{s \in [\lfl \frac{t}{\Delta} \rfl\Delta, t]} |X_i^{\Delta}(s)-X_i(s)| 
		\le C_\gamma \Delta, \label{eq:comparison_aux_1} \\
		& \Pmb(\xi_{ij}^{\beta,\Delta}(t) \ne \xi_{ij}^\beta(t)) 
		\le 2C_\gamma \Delta + 2\Emb \|X_i^\beta-X_i\|_{*,t}. \label{eq:comparison_aux_2}
	\end{align}
	From this, \eqref{eq:moment-bd}, Condition \ref{cond:acceleration}(i) and an application of time change $s \mapsto \frac{s}{\beta}$ (so that the evolution of $\xi_{ij}^{\beta,\Delta}$ matches that of $\Ytil$) we have
	\begin{align}
		& \int_{(k-1)\Delta}^{k\Delta} \Emb \left| \frac{1}{n} \sum_{j=1}^n \left( \gamma(y,X_i^{\Delta}(s),X_j^{\Delta}(s),\xi_{ij}^{\beta,\Delta}(s)) \right. \right. \notag \\
		& \left. \left. \quad - \int_\Zmb \gamma(y,X_i^{\Delta}(s),X_j^{\Delta}(s),\xitil) \,Q(X_i^{\Delta}(s),X_j^{\Delta}(s),d\xitil) \right) \right| ds \le  \frac{\kappa_1\gamma_y}{\beta} + \frac{2\Delta\gamma_y}{\sqrt{n}} \label{eq:comparison_exp_ergodic}
	\end{align}
	for each $k \in \Nmb$.		
	
	Now we show \eqref{eq:comparison_1}.
	Since $X_i$ are i.i.d.\ with law $\mu$, {it follows from \eqref{eq:acceleration_system} and the law of large numbers that}
	$$\nu_i(t)(d\xtil\,d\xitil)=\lim_{n \to \infty} \frac{1}{n} \sum_{j=1}^n \delta_{X_j(t)}(d\xtil) \otimes Q(X_i(t),X_j(t),d\xitil).$$
	Hence
	\begin{align*}
		\Emb \|X_i^\beta-X_i\|_{*,t}
		& \le \Emb \int_{[0,t]\times\Zmb} |y| \left| \Gamma(y,X_i^\beta(s),\nu_i^\beta(s)) - \Gamma(y,X_i(s),\nu_i(s)) \right| ds\,\measurea(dy) \\
		& \le \int_{[0,t]\times\Zmb} |y| \lim_{n \to \infty} \Emb \left| \frac{1}{n} \sum_{j=1}^n \left( \gamma(y,X_i^\beta(s),X_j^\beta(s),\xi_{ij}^\beta(s)) \right. \right. \\
		& \qquad \left. \left. - \int_\Zmb \gamma(y,X_i(s),X_j(s),\xitil) \,Q(X_i(s),X_j(s),d\xitil) \right) \right| ds\,\measurea(dy),
	\end{align*}
	where the last line uses the dominated convergence theorem.
	For each $y \in \Zmb$, $j \in \Nmb$ and $s \in [0,T]$,
	\begin{align*}
		& \Emb \left| \frac{1}{n} \sum_{j=1}^n \left( \gamma(y,X_i^\beta(s),X_j^\beta(s),\xi_{ij}^\beta(s)) - \int_\Zmb \gamma(y,X_i(s),X_j(s),\xitil) \,Q(X_i(s),X_j(s),d\xitil) \right) \right| \\
		& \le \frac{1}{n} \sum_{j=1}^n \Emb \left| \gamma(y,X_i^\beta(s),X_j^\beta(s),\xi_{ij}^\beta(s)) - \gamma(y,X_i^{\Delta}(s),X_j^{\Delta}(s),\xi_{ij}^{\beta,\Delta}(s)) \right| \\
		& \quad + \Emb \left| \frac{1}{n} \sum_{j=1}^n \left( \gamma(y,X_i^{\Delta}(s),X_j^{\Delta}(s),\xi_{ij}^{\beta,\Delta}(s)) - \int_\Zmb \gamma(y,X_i^{\Delta}(s),X_j^{\Delta}(s),\xitil) \, Q(X_i^{\Delta}(s),X_j^{\Delta}(s),d\xitil) \right) \right| \\
		& \quad + \frac{1}{n} \sum_{j=1}^n \Emb \left| \int_\Zmb \gamma(y,X_i^{\Delta}(s),X_j^{\Delta}(s),\xitil) \, Q(X_i^{\Delta}(s),X_j^{\Delta}(s),d\xitil) \right. \\
		& \qquad \left. - \int_\Zmb \gamma(y,X_i(s),X_j(s),\xitil) \, Q(X_i(s),X_j(s),d\xitil) \right|.
	\end{align*}	
	Similar to the analysis of first three terms on the RHS of \eqref{eq:acc_pf_triangle}, it follows from Condition \ref{cond:no_acceleration}(i), \eqref{eq:comparison_aux_1}, \eqref{eq:comparison_aux_2} and \eqref{eq:comparison_exp_ergodic} that	
	\begin{align*}
		& \Emb \left| \gamma(y,X_i^\beta(s),X_j^\beta(s),\xi_{ij}^\beta(s)) - \gamma(y,X_i^{\Delta}(s),X_j^{\Delta}(s),\xi_{ij}^{\beta,\Delta}(s)) \right| \\
		& \qquad \le 4\gamma_y \Emb \|X_i^\beta-X_i\|_{*,s} + 4\gamma_y C_\gamma \Delta, \\
		& \int_0^t \Emb \left| \frac{1}{n} \sum_{j=1}^n \left( \gamma(y,X_i^{\Delta}(s),X_j^{\Delta}(s),\xi_{ij}^{\beta,\Delta}(s)) - \int_\Zmb \gamma(y,X_i^{\Delta}(s),X_j^{\Delta}(s),\xitil) \, Q(X_i^{\Delta}(s),X_j^{\Delta}(s),d\xitil) \right) \right| ds \\
		& \qquad \le \kappa_1\frac{\gamma_yt}{\beta\Delta} + \frac{2\gamma_yt}{\sqrt{n}} + \Delta \gamma_y, \\
		& \Emb \left| \int_\Zmb \gamma(y,X_i^{\Delta}(s),X_j^{\Delta}(s),\xitil) \, Q(X_i^{\Delta}(s),X_j^{\Delta}(s),d\xitil) - \int_\Zmb \gamma(y,X_i(s),X_j(s),\xitil) \, Q(X_i(s),X_j(s),d\xitil) \right| \\
		& \qquad \le 2\gamma_y C_\gamma \Delta.
	\end{align*}

	Combining all of the above estimates and using Condition \ref{cond:no_acceleration}(i) gives
	\begin{equation*}
		\Emb \|X_i^\beta-X_i\|_{*,t} \le 4C_\gamma \int_0^t \Emb \|X_i^\beta-X_i\|_{*,s} \,ds + \kappa_1 \frac{C_\gamma t}{\beta\Delta} + C_\gamma \Delta + 6 C_\gamma^2 t \Delta.
	\end{equation*}
	Using Gronwall's inequality and taking $\Delta=\Delta(\beta) = \frac{1}{\sqrt{\beta}}$ we have \eqref{eq:comparison_1}.

	Using the definition of $\dBL$ and \eqref{eq:comparison_1}, we have
	\begin{align*}
		\dBL(\mu^\beta,\mu) & = \sup_{f \in \Bmb_1} \left| \Emb f(X_i^\beta) - \Emb f(X_i) \right| \le \Emb \|X_i^\beta-X_i\|_{*,T} \le \frac{\kappa}{\sqrt{\beta}}.
	\end{align*}
	This gives \eqref{eq:comparison_2}.

	(b) 
	Fix $t \in (0,T]$.
	Take $\Delta=\Delta(\beta)=O(\frac{1}{\sqrt{\beta}})$ such that $(k+\frac{1}{2})\Delta \le t < (k+1)\Delta$ for some $k=k(\beta) \in \Nmb$.
	Define
	\begin{align*}
		\nu_i^{\beta,\Delta,1}(t) & = \lim_{n \to \infty} \frac{1}{n} \sum_{j=1}^n \delta_{(X_j^{\Delta}(t),\xi_{ij}^{\beta,\Delta}(t))}, \\
		\nu_i^{\beta,\Delta,2}(t) & = \lim_{n \to \infty} \frac{1}{n} \sum_{j=1}^n \delta_{X_j^{\Delta}(t)} \otimes \Lmc(\xi_{ij}^{\beta,\Delta}(t) \,|\, \xi_{ij}^{\beta}(k\Delta),X_i^{\Delta}(t),X_j^{\Delta}(t)), \\
		\nu_i^{\beta,\Delta,3}(t) & = \lim_{n \to \infty} \frac{1}{n} \sum_{j=1}^n \delta_{X_j^{\Delta}(t)} \otimes Q(X_i^{\Delta}(t),X_j^{\Delta}(t),\cdot), \\
		\nu_i^{\beta,\Delta,4}(t) & = \lim_{n \to \infty} \frac{1}{n} \sum_{j=1}^n \delta_{X_j(t)} \otimes Q(X_i(t),X_j(t),\cdot).
	\end{align*}
	From \eqref{eq:comparison_1}, \eqref{eq:comparison_aux_1} and \eqref{eq:comparison_aux_2} we have
	\begin{align*}
		\Emb \dBL(\nu_i^\beta(t),\nu_i^{\beta,\Delta,1}(t)) & \le \liminf_{n \to \infty} \frac{1}{n} \sum_{j=1}^n \left( \Emb|X_j^\beta(t)-X_j^{\Delta}(t)| +2\Pmb(\xi_{ij}^\beta(t) \ne \xi_{ij}^{\beta,\Delta}(t)) \right) \to 0, \\
		\Emb \dBL(\nu_i^{\beta,\Delta,3}(t),\nu_i^{\beta,\Delta,4}(t)) & \le \liminf_{n \to \infty} \frac{1}{n} \sum_{j=1}^n \Pmb(X_j^\Delta(t) \ne X_j(t)) \to 0,
	\end{align*}
	as $\beta \to \infty$.
	Using the conditional independence of $\{\xi_{ij}^{\beta,\Delta}(t) : j \in \Nmb\}$ given $\{X_j^\Delta(t) : j \in \Nmb\}$, we have
	\begin{align*}
		& \Emb \left[ \left( \lan f,\nu_i^{\beta,\Delta,1}(t) \ran - \lan f,\nu_i^{\beta,\Delta,2}(t) \ran \right)^2 \right] \\
		& \le \liminf_{n \to \infty} \frac{1}{n^2} \sum_{j=1}^n \Emb \left[ \left( f(X_j^{\Delta}(t),\xi_{ij}^{\beta,\Delta}(t)) - \lan f,\delta_{X_j^{\Delta}(t)} \otimes \Lmc(\xi_{ij}^{\beta,\Delta}(t) \,|\, \xi_{ij}^\beta(k\Delta),X_i^{\Delta}(t),X_j^{\Delta}(t)) \ran \right)^2 \right] \\
		& = 0
	\end{align*}
	for each bounded and continuous function $f$.
	This means $$\nu_i^{\beta,\Delta,1}(t) = \nu_i^{\beta,\Delta,2}(t) \text{ a.s.}$$
	Similarly, using the independence of $\{X_i : i \in \Nmb\}$ and the weak law of large numbers we have 
	$$\nu_i^{\beta,\Delta,4}(t) = \nu_i(t) \text{ a.s.}$$
	Moreover, from \eqref{eq:acc_asump_addition} and the definition of $\xi_{ij}^{\beta,\Delta}$ we have
	\begin{align*}
		& \Emb \left| \lan f,\nu_i^{\beta,\Delta,2}(t) \ran - \lan f,\nu_i^{\beta,\Delta,3}(t) \ran \right| \\
		& \le \liminf_{n \to \infty} \frac{2\|f\|_\infty}{n} \sum_{j=1}^n \Emb \left[ \sum_{\xitil \in \Zmb} \left|\Pmb(\xi_{ij}^{\beta,\Delta}(t)=\xitil \,|\, \xi_{ij}^\beta(k\Delta),X_i^{\Delta}(t),X_j^{\Delta}(t)) - Q(X_i^{\Delta}(t),X_j^{\Delta}(t),\{\xitil\})\right| \right] \\
		& \le \liminf_{n \to \infty} \frac{2\|f\|_\infty}{n} \sum_{j=1}^n \Emb \left[ C((t-k\Delta)\beta) (1+|X_i^{\Delta}(t)|+|X_j^{\Delta}(t)|) \right] \to 0 \\
		& \le \kappa_2 C((t-k\Delta)\beta) \to 0
	\end{align*}
	as $\beta \to \infty$, for each bounded and continuous function $f$.	
	Combining above estimates gives \eqref{eq:comparison_3}.
	Finally, \eqref{eq:comparison_4} is a direct consequence of \eqref{eq:comparison_3}.
\end{proof}

\begin{Remark}
	\label{rmk:kappa-acc}
	If we are only interested in the convergence of $X_i^\beta, \mu^\beta, \nu_i^\beta$, then it would be sufficient to assume that $\lim_{t \to \infty} \frac{1}{t} \int_0^t \widetilde\kappa(s)\,ds = 0$, instead of $\int_0^\infty \widetilde\kappa(t)\,dt < \infty$.
	We only have to replace $\kappa_1$ by $\kappa_1 \int_0^{\beta\Delta} \widetilde\kappa(s)\,ds$ in \eqref{eq:comparison_exp_ergodic} and in what follows.
\end{Remark}

\section{Proofs of central limit theorems}
\label{sec:pf-CLT}

In this section we prove Theorems \ref{thm:CLT_nu} and \ref{thm:CLT_mu}.

\subsection{Asymptotics of symmetric statistics}
\label{sec:Dynkin}

The proof of CLT crucially relies on certain classical results from \cite{Dynkin1983} on limit laws of degenerate symmetric statistics.
In this section we briefly review these results.

Let $\Smb$ be a Polish space and let $\{Y_n\}_{n=1}^\infty$ be a sequence of i.i.d.\ $\Smb$-valued random variables having common probability law $\thetabar$.
For $k \in \Nmb$, let $L^2(\thetabar^{\otimes k})$ be the space of all real-valued square integrable functions on $(\Smb^k, \Bmc(\Smb)^{\otimes k}, \thetabar^{\otimes k})$.
Denote by $L^2_c(\thetabar^{\otimes k})$ the subspace of centered functions, namely $\phi \in L^2(\thetabar^{\otimes k})$ such that for all $1 \le j \le k$,
\begin{equation*}
	\int_\Smb \phi(x_1,\dotsc,x_{j-1},x,x_{j+1},\dotsc,x_k) \, \thetabar(dx) = 0, \quad \thetabar^{\otimes k-1} \text{ a.e. } (x_1,\dotsc,x_{j-1},x_{j+1},\dotsc,x_k).
\end{equation*}
Denote by $L^2_{sym}(\thetabar^{\otimes k})$ the subspace of symmetric functions, namely $\phi \in L^2(\thetabar^{\otimes k})$ such that for every permutation $\pi$ on $\{1,\dotsc,k\}$,
\begin{equation*}
	\phi(x_1,\dotsc,x_k) = \phi(x_{\pi(1)},\dotsc,x_{\pi(k)}), \quad \thetabar^{\otimes k} \text{ a.e. } (x_1,\dotsc,x_k).
\end{equation*}
Also, denote by $L^2_{c,sym}(\thetabar^{\otimes k})$ the subspace of centered symmetric functions in $L^2(\thetabar^{\otimes k})$, namely $L^2_{c,sym}(\thetabar^{\otimes k}) := L^2_c(\thetabar^{\otimes k}) \bigcap L^2_{sym}(\thetabar^{\otimes k})$.
Given $\phi_k \in L^2_{sym}(\thetabar^{\otimes k})$ define the symmetric statistic $\Umc^n_k (\phi_k)$ as
\begin{equation*}
	\Umc^n_k (\phi_k) :=
	\begin{cases}
		\displaystyle \sum_{1 \le i_1 < i_2 < \dotsb < i_k \le n} \phi_k(Y_{i_1},\dotsc,Y_{i_k}) & \text{for } n \ge k \\
		0 & \text{for } n<k.
	\end{cases}
\end{equation*}
In order to describe the asymptotic distributions of such statistics consider a Gaussian field $\{I_1(h) : h \in L^2(\thetabar)\}$ such that
\begin{equation*}
	\Ebf \left( I_1(h) \right) = 0, \: \Ebf \left( I_1(h)I_1(g) \right) = \langle h,g \rangle_{L^2(\thetabar)}, \quad h,g \in L^2(\thetabar).
\end{equation*}
For $h \in L^2(\thetabar)$, define $\phi_k^h \in L^2_{sym}(\thetabar^{\otimes k})$ as
\begin{equation*}
	\phi_k^h(x_1,\dotsc,x_k) := h(x_1) \dotsc h(x_k)
\end{equation*}
and set $\phi_0^h := 1$.

The multiple Wiener integral (MWI) of $\phi_k^h$, denoted as $I_k(\phi_k^h)$, is defined through the following formula.
For $k \ge 1$,
\begin{equation*} 
\label{eq:MWI_formula}
	I_k(\phi_k^h) := \sum_{j=0}^{\lfloor k/2 \rfloor} (-1)^j C_{k,j} ||h||^{2j}_{L^2(\thetabar)} (I_1(h))^{k-2j}, \text{ where } C_{k,j} := \frac{k!}{(k-2j)! 2^j j!}, j=0,\dotsc,\lfloor k/2 \rfloor.
\end{equation*}
The following representation gives an equivalent way to characterize the MWI of $\phi_k^h $:
\begin{equation*}
	\sum_{k=0}^{\infty} \frac{t^k}{k!} I_k(\phi_k^h) = \exp \left( tI_1(h) - \frac{t^2}{2} ||h||^2_{L^2(\thetabar)} \right), \quad t \in \Rmb,
\end{equation*}
where we set $I_0(\phi_0^h) := 1$. 
We extend the definition of $I_k$ to the linear span of $\{\phi_k^h, h \in L^2(\thetabar)\}$ by linearity. 
It can be checked that for all $f$ in this linear span,
\begin{equation} \label{eq:MWI_moments}
	\Ebf(I_k(f))^2 = k! \, ||f||^2_{L^2(\thetabar^{\otimes k})}.
\end{equation}
Using this identity and standard denseness arguments, the definition of $I_k(f)$ can be extended to all $f \in L^2_{sym}(\thetabar^{\otimes k})$ and the identity $\eqref{eq:MWI_moments}$ holds for all $f \in L^2_{sym}(\thetabar^{\otimes k})$. 
The following result is taken from \cite{Dynkin1983}.

\begin{Lemma}[Dynkin-Mandelbaum \cite{Dynkin1983}] \label{thm:Dynkin}
	Let $\{ \phi_k \}_{k=1}^\infty$ be such that $\phi_k \in L^2_{c,sym}(\thetabar^{\otimes k})$ for each $k \ge 1$. 
	Then the following convergence holds as $n \to \infty$:
	\begin{equation*}
		\left( n^{-\frac{k}{2}} \Umc^n_k (\phi_k) \right)_{k \ge 1} \Rightarrow \left( \frac{1}{k!} I_k(\phi_k) \right)_{k \ge 1}
	\end{equation*}
	as a sequence of $\Rmb^\infty$-valued random variables.
\end{Lemma}

\subsection{Girsanov change of measure}

Recall the probability space $(\Omega, \Fmc, \Pmb)$ under which systems \eqref{eq:system} and \eqref{eq:no_acceleration_system} are defined.
Recall the canonical spaces and processes introduced in Section \ref{sec:canonical}.
Let $$\nubar_i^n(t) = \frac{1}{n} \sum_{j=1}^n \delta_{(X_j(t),\xi_{ij}(t))}.$$
Define
\begin{align*}
	J^n(t) & := \sum_{i=1}^n \left( \int_{\Xmb_t} r_{s-}^{n,i}(y,z) N_{i}(ds\,dy\,dz) - \int_{[0,t]\times\Zmb} e_s^{n,i}(y)\,ds\,\measurea(dy) \right),
\end{align*}
where
\begin{align*}
	r_s^{n,i}(y,z) & := \one_{[0,\Gamma(y,X_i(s),\nu_i(s))]}(z) \log \frac{\Gamma(y,X_i(s),\nubar_i^n(s))}{\Gamma(y,X_i(s),\nu_i(s))}, \\
	e_s^{n,i}(y) & := \Gamma(y,X_i(s),\nubar_i^n(s)) - \Gamma(y,X_i(s),\nu_i(s)) = \Gamma(y,X_i(s),\nubar_i^n(s)-\nu_i(s)). 
\end{align*}
Let $\Fmcbar_t^n := \sigma\{N_i(A),X_i[t],\xi_{ij}[t] : i,j \in [n], A \in \Bmc(\Xmb_t)\}$.
{Since $\gamma$ is bounded from above and away from $0$ by Condition \ref{cond:CLT} and $\rho$ is a finite measure, one can check that $\{\exp (J^n(t))\}$ is an $\Fmcbar_t^n$-martingale under $P^n$ (see e.g.\ \cite[Theorem VI.T4]{Bremaud1981point})}.
Define a new probability measure $Q^n$ on $\Omega_n$ by
\begin{equation*}
	\frac{dQ^n}{dP^n} := \exp \left( J^n(T) \right).
\end{equation*}

By Girsanov's Theorem {(see e.g.\ \cite[Theorem VI.T3]{Bremaud1981point} and \cite[Theorem III.3.24]{JacodShiryaev2013limit})}, $\{X_i,\xi_{ij} : i,j \in [n]\}$ has the same probability distribution under $Q^n$ as $\{X_i^n,\xi_{ij}^n : i,j \in [n]\}$ under $\Pmb$.
Let 
\begin{align*}
	\etabar^n(\varphi) & := \sqrt{n} \left( \frac{1}{n} \sum_{j=1}^n \varphi(X_j,\xi_{1j}) - m_\varphi(X_1) \right), \quad \varphi \in \Amc, \\
	\etabar_x^n(\varphi) & := \sqrt{n} \left( \frac{1}{n} \sum_{j=1}^n \varphi(X_j) - \Emb \varphi(X_1) \right), \quad \varphi \in \Amc_x.
\end{align*}
Thus in order to prove Theorems \ref{thm:CLT_nu} and \ref{thm:CLT_mu} it suffices to show that
\begin{align*}
	\lim_{n \to \infty} \Emb_{Q^n} \exp \left( \sqrt{-1} \etabar^n(\varphi) \right) & = \int_{\Omega_0} \exp \left( -\half (\sigma^\varphi_{\omega_0})^2 \right) P_0(d\omega_0), \quad \varphi \in \Amc, \\
	\lim_{n \to \infty} \Emb_{Q^n} \exp \left( \sqrt{-1} \etabar_x^n(\varphi) \right) & = \exp \left( -\half \|(I-A)^{-1} \Phi\|_{\Hmc}^2 \right), \quad \varphi \in \Amc_x,
\end{align*}
which is equivalent to showing
\begin{align}
	\lim_{n \to \infty} \Emb_{P^n} \exp \left( \sqrt{-1} \etabar^n(\varphi) + J^n(T) \right) & = \int_{\Omega_0} \exp \left( -\half (\sigma^\varphi_{\omega_0})^2 \right) P_0(d\omega_0), \quad \varphi \in \Amc, \label{eq:goal_CLT_nu} \\
	\lim_{n \to \infty} \Emb_{P^n} \exp \left( \sqrt{-1} \etabar_x^n(\varphi) + J^n(T) \right) &  = \exp \left( -\half \|(I-A)^{-1} \Phi\|_{\Hmc}^2 \right), \quad \varphi \in \Amc_x \label{eq:goal_CLT_mu}.
\end{align}
For this we will need to study the asymptotics of $J^n$ as $n \to \infty$.

\subsection{Asymptotics of $J^n$}

Now we analyze the asymptotics of $J^n$.
Recall the constant $\varepsilon$ from Condition \ref{cond:CLT}.
From Taylor's expansion, there exists a $\kappa_0 \in (0,\infty)$ such that for all $\alpha,\beta \in [\varepsilon,1/\varepsilon]$,
$$\log \frac{\alpha}{\beta} = (\frac{\alpha}{\beta}-1) - \half (\frac{\alpha}{\beta}-1)^2 + \vartheta(\alpha,\beta) (\frac{\alpha}{\beta}-1)^3,$$
where $|\vartheta(\alpha,\beta)|\le\kappa_0$.
Letting $\vartheta^{n,i}_s(y) := \vartheta(\Gamma(y,X_i(s),\nubar_i^n(s)),\Gamma(y,X_i(s),\nu_i(s)))$, we have
\begin{align*}
	\log \frac{\Gamma(y,X_i(s),\nubar_i^n(s))}{\Gamma(y,X_i(s),\nu_i(s))} & = \left( \frac{\Gamma(y,X_i(s),\nubar_i^n(s))}{\Gamma(y,X_i(s),\nu_i(s))} - 1 \right) - \half \left( \frac{\Gamma(y,X_i(s),\nubar_i^n(s))}{\Gamma(y,X_i(s),\nu_i(s))} - 1 \right)^2 \\
	& \qquad + \vartheta^{n,i}_s(y) \left( \frac{\Gamma(y,X_i(s),\nubar_i^n(s))}{\Gamma(y,X_i(s),\nu_i(s))} - 1 \right)^3.
\end{align*}
Letting $\Ncompensated_i$ be the compensated PRM of $N_i$, we have
\begin{align}
	J^n(T) & = \sum_{i=1}^n \int_{\Xmb_T} r_{s-}^{n,i}(y,z) N_{i}(ds\,dy\,dz) - \sum_{i=1}^n \int_{[0,T]\times\Zmb} e_s^{n,i}(y)\,ds\,\measurea(dy) \notag \\
	& = \sum_{i=1}^n \int_{\Xmb_T} \one_{[0,\Gamma(y,X_i(s-),\nu_i(s-))]}(z) \left( \frac{\Gamma(y,X_i(s-),\nubar_i^n(s-))}{\Gamma(y,X_i(s-),\nu_i(s-))} - 1 \right) \Ncompensated_{i}(ds\,dy\,dz) \notag \\
	& \quad - \half \sum_{i=1}^n \int_{\Xmb_T} \one_{[0,\Gamma(y,X_i(s-),\nu_i(s-))]}(z) \left( \frac{\Gamma(y,X_i(s-),\nubar_i^n(s-))}{\Gamma(y,X_i(s-),\nu_i(s-))} - 1 \right)^2 N_{i}(ds\,dy\,dz) \notag \\
	& \quad + \sum_{i=1}^n \int_{\Xmb_T} \one_{[0,\Gamma(y,X_i(s-),\nu_i(s-))]}(z) \vartheta^{n,i}_{s-}(y) \left( \frac{\Gamma(y,X_i(s-),\nubar_i^n(s-))}{\Gamma(y,X_i(s-),\nu_i(s-))} - 1 \right)^3 N_{i}(ds\,dy\,dz) \notag \\
	& =: J^{n,1} - \half J^{n,2} + J^{n,3}. \label{eq:Jn}
\end{align}

First we analyze $J^{n,3}$.

\begin{Lemma}
	\label{lem:Jn3}
	$\Emb_{P^n} |J^{n,3}| \to 0$ as $n \to \infty$.
\end{Lemma}

\begin{proof}
	Since $|\vartheta^{n,i}_s(y)| \le \kappa_0$ and $\Gamma$ is bounded from above and away from $0$ by Condition \ref{cond:CLT}, we have
	\begin{align*}
		\Emb_{P^n} |J^{n,3}| & \le \frac{\kappa_0}{\varepsilon^2} \Emb_{P^n} \sum_{i=1}^n \int_{[0,T]\times\Zmb} \left| \Gamma(y,X_i(s),\nubar_i^n(s)-\nu_i(s)) \right|^3 ds\,\measurea(dy).
	\end{align*}
	Note that
	\begin{align*}
		\max_{i \in [n]} \sup_{y \in \Zmb} \sup_{s \in [0,T]} \Emb_{P^n} \left| \Gamma(y,X_i(s),\nubar_i^n(s)-\nu_i(s)) \right|^3 \le \frac{\kappa_1}{n^{3/2}}
	\end{align*}
	by conditional i.i.d.\ property of $(X_j,\xi_{ij})$ given $X_i$ and Lemma \ref{lem:iid-moment}.
	Therefore
	\begin{equation*}
		\Emb_{P^n} |J^{n,3}| \le \frac{\kappa_2}{n^{1/2}} \to 0
	\end{equation*}
	as $n \to \infty$.
\end{proof}

Before analyzing $J^{n,2}$, let 
\begin{align}
	\gamma_s^{ij,y} & := \gamma(y,X_i(s),X_j(s),\xi_{ij}(s)) - \lan \gamma(y,X_i(s),\cdot,\cdot), \nu_i(s) \ran. \label{eq:gamma-ij}
\end{align}
Note that
\begin{equation}
	\label{eq:gamma_s_property}
	\Emb_{P^n} \left[ \gamma_s^{ij,y} \,\Big|\, X_i \right]=0, \mbox{ for } j \ne i.
\end{equation}
We first show the following estimate.

\begin{Lemma}
	\label{lem:Jn2_prep}
	For each $s \in [0,T]$ and bounded measurable function $f$,
	\begin{align*}
		& \Emb_{P^n} \left| \frac{1}{n^2} \sum_{1 \le i < j < k \le n} \left( \gamma_{s}^{ij,y}\gamma_{s}^{ik,y}f(X_i(s),\nu_i(s)) - \Emb_{P^n} \left[ \gamma_{s}^{ij,y}\gamma_{s}^{ik,y}f(X_i(s),\nu_i(s)) \,\Big|\, X_j,X_k \right] \right) \right|^2 \\
		& \qquad \le \frac{4\|f\|_\infty^2}{\varepsilon^4n}.
	\end{align*}
\end{Lemma}

\begin{proof}
	We claim that for $i<j<k$ and $\itil<\jtil<\ktil$,
	\begin{align}
		& \Emb_{P^n} \left[ \left( \gamma_{s}^{ij,y}\gamma_{s}^{ik,y}f(X_i(s),\nu_i(s)) - \Emb_{P^n} \left[ \gamma_{s}^{ij,y}\gamma_{s}^{ik,y}f(X_i(s),\nu_i(s)) \,\Big|\, X_j,X_k \right] \right) \right. \notag \\
		& \qquad \left. \cdot\left( \gamma_{s}^{\itil\jtil,y}\gamma_{s}^{\itil\ktil,y}f(X_\itil(s),\nu_\itil(s)) - \Emb_{P^n} \left[ \gamma_{s}^{\itil\jtil,y}\gamma_{s}^{\itil\ktil,y}f(X_\itil(s),\nu_\itil(s)) \,\Big|\, X_\jtil,X_\ktil \right] \right) \right] = 0 \label{eq:Jn2_claim}
	\end{align}
	except when $i=\itil$, $j=\jtil$ and $k=\ktil$.

	To see this, first consider $k < \ktil$.
	Using the independence of the collection $\{X_i,\xi_{ij}(0),N_{ij} : i,j \in \Nmb\}$ and \eqref{eq:gamma_s_property}, we have
	\begin{align*}
		& \Emb_{P^n} \left[ \Emb_{P^n} \left[ \gamma_{s}^{\itil\jtil,y}\gamma_{s}^{\itil\ktil,y}f(X_\itil(s),\nu_\itil(s)) \,\Big|\, X_\jtil,X_\ktil \right] \Big|\, X_i,X_j,X_k,\xi_{ij}(0),N_{ij},\xi_{ik}(0),N_{ik},X_\itil,X_\jtil \right] \\
		& = \Emb_{P^n} \left[ \gamma_{s}^{\itil\jtil,y}\gamma_{s}^{\itil\ktil,y}f(X_\itil(s),\nu_\itil(s)) \,\Big|\, X_\jtil \right] \\
		& = \Emb_{P^n} \left[ \Emb_{P^n} \left[ \gamma_{s}^{\itil\jtil,y}\gamma_{s}^{\itil\ktil,y}f(X_\itil(s),\nu_\itil(s)) \,\Big|\, X_\itil, X_\jtil, \xi_{\itil\jtil} \right] \Big|\, X_\jtil \right] = 0
	\end{align*}
	and
	\begin{align*}
		& \Emb_{P^n} \left[ \gamma_{s}^{\itil\jtil,y}\gamma_{s}^{\itil\ktil,y}f(X_\itil(s),\nu_\itil(s)) \,\Big|\, X_i,X_j,X_k,\xi_{ij}(0),N_{ij},\xi_{ik}(0),N_{ik},X_\itil,X_\jtil \right] \\
		&  = \Emb_{P^n} \left[ \Emb_{P^n} \left[ \gamma_{s}^{\itil\jtil,y}\gamma_{s}^{\itil\ktil,y}f(X_\itil(s),\nu_\itil(s)) \,\Big|\, X_i,X_j,X_k,\xi_{ij}(0),N_{ij},\xi_{ik}(0),N_{ik},X_\itil,X_\jtil,\xi_{\itil\jtil} \right] \right. \\
		&  \qquad \left. \,|\, X_i,X_j,X_k,\xi_{ij}(0),N_{ij},\xi_{ik}(0),N_{ik},X_\itil,X_\jtil \right] \\
		& = 0.
	\end{align*}
	Combining these two and conditioning on $X_i,X_j,X_k,\xi_{ij}(0),N_{ij},\xi_{ik}(0),N_{ik},X_\itil,X_\jtil$ in the LHS of \eqref{eq:Jn2_claim}, we have verified \eqref{eq:Jn2_claim} for $k < \ktil$.
	Therefore \eqref{eq:Jn2_claim} holds whenever $k \ne \ktil$.
	
	Next consider $k=\ktil$ but $j < \jtil$.
	Note that we still have
	\begin{align*}
		& \Emb_{P^n} \left[ \Emb_{P^n} \left[ \gamma_{s}^{\itil\jtil,y}\gamma_{s}^{\itil\ktil,y}f(X_\itil(s),\nu_\itil(s)) \,\Big|\, X_\jtil,X_\ktil \right] \Big|\, X_i,X_j,\xi_{ij}(0),N_{ij},\xi_{i\ktil}(0),N_{i\ktil},X_\itil,X_\ktil \right] \\
		& = \Emb_{P^n} \left[ \gamma_{s}^{\itil\jtil,y}\gamma_{s}^{\itil\ktil,y}f(X_\itil(s),\nu_\itil(s)) \,\Big|\, X_\ktil \right] \\
		& = \Emb_{P^n} \left[ \Emb_{P^n} \left[ \gamma_{s}^{\itil\jtil,y}\gamma_{s}^{\itil\ktil,y}f(X_\itil(s),\nu_\itil(s)) \,\Big|\, X_\itil, X_\ktil, \xi_{\itil\ktil} \right] \Big|\, X_\ktil \right] = 0
	\end{align*}
	and
	\begin{align*}
		& \Emb_{P^n} \left[ \gamma_{s}^{\itil\jtil,y}\gamma_{s}^{\itil\ktil,y}f(X_\itil(s),\nu_\itil(s)) \,\Big|\, X_i,X_j,\xi_{ij}(0),N_{ij},\xi_{i\ktil}(0),N_{i\ktil},X_\itil,X_\ktil \right] = 0.
	\end{align*}
	So again we have \eqref{eq:Jn2_claim} for $k=\ktil$ and $j < \jtil$.
	Therefore \eqref{eq:Jn2_claim} holds whenever $k \ne \ktil$ or $j \ne \jtil$.
	
	Consider $k=\ktil$, $j=\jtil$ but $i<\itil$.
	Note that
	\begin{align*}
		& \Emb_{P^n} \left[ \Emb_{P^n} \left[ \gamma_{s}^{\itil\jtil,y}\gamma_{s}^{\itil\ktil,y}f(X_\itil(s),\nu_\itil(s)) \,\Big|\, X_\jtil,X_\ktil \right] \Big|\, X_i,\xi_{i\jtil},\xi_{i\ktil},X_\jtil,X_\ktil \right] \\
		& = \Emb_{P^n} \left[ \gamma_{s}^{\itil\jtil,y}\gamma_{s}^{\itil\ktil,y}f(X_\itil(s),\nu_\itil(s)) \,\Big|\, X_\jtil,X_\ktil \right] \\
		& = \Emb_{P^n} \left[ \gamma_{s}^{\itil\jtil,y}\gamma_{s}^{\itil\ktil,y}f(X_\itil(s),\nu_\itil(s)) \,\Big|\, X_i,\xi_{i\jtil},\xi_{i\ktil},X_\jtil,X_\ktil \right].
	\end{align*}
	Conditioning on $X_i,\xi_{i\jtil},\xi_{i\ktil},X_\jtil,X_\ktil$ in the LHS of \eqref{eq:Jn2_claim}, we have verified \eqref{eq:Jn2_claim} for $k=\ktil$, $j=\jtil$ but $i<\itil$.
	Therefore \eqref{eq:Jn2_claim} holds whenever $k \ne \ktil$, $j \ne \jtil$, or $i \ne \itil$.
	So we have verified the claim.

	Using the claim and Condition \ref{cond:CLT} we have
	\begin{align*}
		& \Emb_{P^n} \left| \frac{1}{n^2} \sum_{1 \le i < j < k \le n} \left( \gamma_{s}^{ij,y}\gamma_{s}^{ik,y}f(X_i(s),\nu_i(s)) - \Emb_{P^n} \left[ \gamma_{s}^{ij,y}\gamma_{s}^{ik,y}f(X_i(s),\nu_i(s)) \,\Big|\, X_j,X_k \right] \right) \right|^2 \\
		& = \frac{1}{n^4} \sum_{1 \le i < j < k \le n} \Emb_{P^n} \left( \gamma_{s}^{ij,y}\gamma_{s}^{ik,y}f(X_i(s),\nu_i(s)) - \Emb_{P^n} \left[ \gamma_{s}^{ij,y}\gamma_{s}^{ik,y}f(X_i(s),\nu_i(s)) \,\Big|\, X_j,X_k \right] \right)^2 \\
		& \le \frac{4\|f\|_\infty^2}{\varepsilon^4n}.
	\end{align*}
	This completes the proof.
\end{proof}

Now we analyze $J^{n,2}$.
Let
\begin{equation}
	\label{eq:Shat}
	\hat{\Smc}^{n,2} := \{ (j,k) \in [n]^2 : j \ne 1, k \ne 1, j \ne k\}.
\end{equation}
Recall $\gamma_s^{ij,y}$ introduced in \eqref{eq:gamma-ij}. {Abusing notation, define}
\begin{align*}
	\lan (\gamma_s^{ij,y})^2, \nu_i(s) \ran & := \int_{\Zmb^2} \left( \gamma(y,X_i(s),\xtil,\xitil) - \lan \gamma(y,X_i(s),\cdot,\cdot), \nu_i(s) \ran \right)^2 \nu_i(s)(d\xtil\,d\xitil) \\
	& = \Emb_{P^n} \left[ (\gamma_s^{ij,y})^2 \,\Big|\, X_i \right], \quad j \ne i.
\end{align*}

\begin{Lemma}
	\label{lem:Jn2}
	\begin{align*}
		J^{n,2} & = \int_{[0,T]\times\Zmb} \Emb_{P^n} \left[ \frac{\lan (\gamma_{s}^{12,y})^2, \nu_1(s) \ran}{\Gamma(y,X_1(s),\nu_1(s))} \right] ds\,\measurea(dy) \\
		& \qquad + \frac{1}{n} \sum_{(j,k) \in \hat{\Smc}^{n,2}} \int_{[0,T]\times\Zmb} \Emb_{P^n} \left[ \frac{\gamma_{s}^{1j,y}\gamma_{s}^{1k,y}}{\Gamma(y,X_1(s),\nu_1(s))} \,\Big|\, X_j,X_k \right] \,ds\,\measurea(dy) + R^{n,2},
	\end{align*}
	where $R^{n,2} \to 0$ in probability as $n \to \infty$.
\end{Lemma}

\begin{proof}
	We can write
	\begin{align*}
		J^{n,2} & = \sum_{i=1}^n \int_{\Xmb_T} \one_{[0,\Gamma(y,X_i(s-),\nu_i(s-))]}(z) \frac{\Gamma^2(y,X_i(s-),\nubar_i^n(s-)-\nu_i(s-))}{\Gamma^2(y,X_i(s-),\nu_i(s-))} N_{i}(ds\,dy\,dz) \\
		& = \frac{1}{n^2} \sum_{i,j,k=1}^n \int_{\Xmb_T} \one_{[0,\Gamma(y,X_i(s-),\nu_i(s-))]}(z) \frac{\gamma_{s-}^{ij,y}\gamma_{s-}^{ik,y}}{\Gamma^2(y,X_i(s-),\nu_i(s-))} N_{i}(ds\,dy\,dz). 
	\end{align*}
	Let
	\begin{align*}
		\Jtil^{n,2} & := \frac{1}{n^2} \sum_{i,j,k=1}^n \int_{[0,T]\times\Zmb} \frac{\gamma_{s-}^{ij,y}\gamma_{s-}^{ik,y}}{\Gamma(y,X_i(s-),\nu_i(s-))} \,ds\,\measurea(dy).
	\end{align*}
	Note that
	\begin{align*}
		\max_{i \in [n]} \int_{\Xmb_T} \Emb_{P^n} \left( \sum_{j,k=1}^n \one_{[0,\Gamma(y,X_i(s),\nu_i(s))]}(z) \frac{\gamma_{s}^{ij,y}\gamma_{s}^{ik,y}}{\Gamma^2(y,X_i(s),\nu_i(s))} \right)^2 \,ds\,\measurea(dy)\,dz \le \kappa_1 n^2
	\end{align*}
	by conditional i.i.d.\ property of $(X_j,\xi_{ij})$ given $X_i$ and Lemma \ref{lem:iid-moment}.
	Therefore
	\begin{align*}
		& \Emb_{P^n} |J^{n,2} - \Jtil^{n,2}|^2 \\
		& = \Emb_{P^n} \left| \frac{1}{n^2} \sum_{i,j,k=1}^n \int_{\Xmb_T} \one_{[0,\Gamma(y,X_i(s-),\nu_i(s-))]}(z) \frac{\gamma_{s-}^{ij,y}\gamma_{s-}^{ik,y}}{\Gamma^2(y,X_i(s-),\nu_i(s-))} \Ncompensated_{i}(ds\,dy\,dz) \right|^2 \\
		& = \frac{1}{n^4} \sum_{i=1}^n \int_{\Xmb_T} \Emb_{P^n} \left( \sum_{j,k=1}^n \one_{[0,\Gamma(y,X_i(s),\nu_i(s))]}(z) \frac{\gamma_{s}^{ij,y}\gamma_{s}^{ik,y}}{\Gamma^2(y,X_i(s),\nu_i(s))} \right)^2 \,ds\,\measurea(dy)\,dz \\
		& \le \frac{\kappa_1}{n} \to 0
	\end{align*}
	as $n \to \infty$.
	
	We rewrite $\Jtil^{n,2}$ as
	\begin{align*}
		\Jtil^{n,2} & = \frac{1}{n^2} \sum_{m=1}^5 \sum_{(i,j,k) \in \Smc^{n,m}} \int_{[0,T]\times\Zmb} \frac{\gamma_{s}^{ij,y}\gamma_{s}^{ik,y}}{\Gamma(y,X_i(s),\nu_i(s))} \,ds\,\measurea(dy) =: \sum_{m=1}^5 \Tmc^{n,m},
	\end{align*}
	where $\Smc^{n,1}, \Smc^{n,2}, \Smc^{n,3}, \Smc^{n,4}, \Smc^{n,5}$ are collections of $(i,j,k) \in [n]^3$ which equal $\{i=j=k\}$, $\{i=j \ne k\}$, $\{i=k \ne j\}$, $\{i \ne j=k\}$, $\{i,j,k \mbox{ distinct}\}$, respectively.
	Using Conditions \ref{cond:no_acceleration} and \ref{cond:CLT}, {we have
	\begin{equation*}
		\Emb_{P^n} |\Tmc^{n,1}| \le \frac{\kappa_2}{n} \to 0
	\end{equation*}
	as $n \to \infty$, and further using the conditional i.i.d.\ property of $(X_j,\xi_{ij})$ given $X_i$, we can show
	\begin{equation*}
		\Emb_{P^n} |\Tmc^{n,2}+\Tmc^{n,3}|^2 \le \frac{\kappa_2}{n} \to 0
	\end{equation*}
	as $n \to \infty$.}
	
	For the fourth term $\Tmc^{n,4}$, we have
	\begin{align*}
		\Tmc^{n,4} & = \frac{1}{n^2} \sum_{1 \le i \ne j \le n} \int_{[0,T]\times\Zmb} \frac{(\gamma_{s}^{ij,y})^2}{\Gamma(y,X_i(s),\nu_i(s))} \,ds\,\measurea(dy) \\
		& = \frac{1}{n^2} \sum_{1 \le i \ne j \le n} \int_{[0,T]\times\Zmb} \frac{(\gamma_{s}^{ij,y})^2 - \lan (\gamma_{s}^{ij,y})^2, \nu_i(s) \ran}{\Gamma(y,X_i(s),\nu_i(s))} \,ds\,\measurea(dy) \\
		& \qquad + \frac{1}{n^2} \sum_{1 \le i \ne j \le n} \int_{[0,T]\times\Zmb} \frac{\lan (\gamma_{s}^{ij,y})^2, \nu_i(s) \ran}{\Gamma(y,X_i(s),\nu_i(s))} \,ds\,\measurea(dy).
	\end{align*}
	Here the first term 
	\begin{equation*}
		\frac{1}{n^2} \sum_{1 \le i \ne j \le n} \int_{[0,T]\times\Zmb} \frac{(\gamma_{s}^{ij,y})^2 - \lan (\gamma_{s}^{ij,y})^2, \nu_i(s) \ran}{\Gamma(y,X_i(s),\nu_i(s))} \,ds\,\measurea(dy) \to 0
	\end{equation*}
	in probability, by conditional i.i.d.\ property of $(X_j,\xi_{ij})$ given $X_i$, and the second term
	\begin{equation*}
		\frac{1}{n^2} \sum_{1 \le i \ne j \le n} \int_{[0,T]\times\Zmb} \frac{\lan (\gamma_{s}^{ij,y})^2, \nu_i(s) \ran}{\Gamma(y,X_i(s),\nu_i(s))} \,ds\,\measurea(dy) \to \int_{[0,T]\times\Zmb} \Emb_{P^n} \left[ \frac{\lan (\gamma_{s}^{12,y})^2, \nu_1(s) \ran}{\Gamma(y,X_1(s),\nu_1(s))} \right] ds\,\measurea(dy)
	\end{equation*}
	in probability, by i.i.d.\ property of $(X_i,\nu_i)$, as $n \to \infty$.
	
	Finally for the last term $\Tmc^{n,5}$, we have
	\begin{align}
		\Tmc^{n,5} & = \frac{1}{n^2} \sum_{(i,j,k) \in \Smc^{n,5}} \int_{[0,T]\times\Zmb} \frac{\gamma_{s}^{ij,y}\gamma_{s}^{ik,y}}{\Gamma(y,X_i(s),\nu_i(s))} \,ds\,\measurea(dy) \notag \\
		& = \frac{1}{n^2} \sum_{(i,j,k) \in \Smc^{n,5}} \int_{[0,T]\times\Zmb} \left( \frac{\gamma_{s}^{ij,y}\gamma_{s}^{ik,y}}{\Gamma(y,X_i(s),\nu_i(s))} - \Emb_{P^n} \left[ \frac{\gamma_{s}^{ij,y}\gamma_{s}^{ik,y}}{\Gamma(y,X_i(s),\nu_i(s))} \,\Big|\, X_j,X_k \right] \right) \,ds\,\measurea(dy) \notag \\
		& \qquad + \frac{1}{n^2} \sum_{(i,j,k) \in \Smc^{n,5}} \int_{[0,T]\times\Zmb} \Emb_{P^n} \left[ \frac{\gamma_{s}^{ij,y}\gamma_{s}^{ik,y}}{\Gamma(y,X_i(s),\nu_i(s))} \,\Big|\, X_j,X_k \right] \,ds\,\measurea(dy). \label{eq:Tn5}
	\end{align}
	From Lemma \ref{lem:Jn2_prep} we have
	\begin{align*}
		& \Emb_{P^n} \left( \frac{1}{n^2} \sum_{(i,j,k) \in \Smc^{n,5}} \int_{[0,T]\times\Zmb} \left( \frac{\gamma_{s}^{ij,y}\gamma_{s}^{ik,y}}{\Gamma(y,X_i(s),\nu_i(s))} - \Emb_{P^n} \left[ \frac{\gamma_{s}^{ij,y}\gamma_{s}^{ik,y}}{\Gamma(y,X_i(s),\nu_i(s))} \,\Big|\, X_j,X_k \right] \right) ds\,\measurea(dy) \right)^2 \\
		& \le \frac{\kappa_3}{n}.
	\end{align*}	
	So the first term in \eqref{eq:Tn5} goes to $0$ in probability.
	The second term in \eqref{eq:Tn5} could be written as
	\begin{align*}
		& \frac{n-2}{n^2} \sum_{(j,k) \in [n]^2 : j \ne k} \int_{[0,T]\times\Zmb} \Emb_{P^n} \left[ \frac{\gamma_{s}^{1j,y}\gamma_{s}^{1k,y}}{\Gamma(y,X_1(s),\nu_1(s))} \,\Big|\, X_j,X_k \right] \,ds\,\measurea(dy) \\
		& = \frac{1}{n} \sum_{(j,k) \in \hat{\Smc}^{n,2}} \int_{[0,T]\times\Zmb} \Emb_{P^n} \left[ \frac{\gamma_{s}^{1j,y}\gamma_{s}^{1k,y}}{\Gamma(y,X_1(s),\nu_1(s))} \,\Big|\, X_j,X_k \right] \,ds\,\measurea(dy) + R^{n,5},
	\end{align*}	
	and one can easily check that $\Emb_{P^n} |R^{n,5}| \le \frac{\kappa_4}{\sqrt{n}}$.
	
	Combining above estimates completes the proof.
\end{proof}

Next we analyze $J^{n,1}$.
Let $\theta_{ij}(s) := \Lmc(\xi_{ij}(s) \,|\, X_i[s],X_j[s])$ and note that
\begin{align*}
	\lan \gamma_s^{ij,y}, \theta_{ij}(s) \ran & = \lan \gamma(y,X_i(s),X_j(s),\cdot), \theta_{ij}(s) \ran - \lan \gamma(y,X_i(s),\cdot,\cdot), \nu_i(s) \ran \\
	& = \Emb_{P^n} \left[ \gamma_s^{ij,y} \,\Big|\, X_i(s), X_j(s) \right].
\end{align*}

\begin{Lemma}
	\label{lem:Jn1}
	\begin{align}
		J^{n,1} & = \frac{1}{n} \sum_{(i,j) \in \hat{\Smc}^{n,2}} \int_{\Xmb_T} \one_{[0,\Gamma(y,X_i(s-),\nu_i(s-))]}(z) \frac{\gamma_{s-}^{ij,y} - \lan \gamma_{s-}^{ij,y}, \theta_{ij} \ran }{\Gamma(y,X_i(s-),\nu_i(s-))} \,\Ncompensated_{i}(ds\,dy\,dz) \label{eq:Jn1} \\
		& \qquad + \frac{1}{n} \sum_{(i,j) \in \hat{\Smc}^{n,2}} \int_{\Xmb_T} \one_{[0,\Gamma(y,X_i(s-),\nu_i(s-))]}(z) \frac{\lan \gamma_{s-}^{ij,y}, \theta_{ij} \ran }{\Gamma(y,X_i(s-),\nu_i(s-))} \,\Ncompensated_{i}(ds\,dy\,dz) + R^{n,1}, \notag
	\end{align}
	where $\Emb_{P^n} |R^{n,1}| \to 0$ as $n \to \infty$.
\end{Lemma}

\begin{proof}
	We can write
	\begin{align*}
		J^{n,1} & = \sum_{i=1}^n \int_{\Xmb_T} \one_{[0,\Gamma(y,X_i(s-),\nu_i(s-))]}(z) \frac{\Gamma(y,X_i(s-),\nubar_i^n(s-)-\nu_i(s-))}{\Gamma(y,X_i(s-),\nu_i(s-))} \,\Ncompensated_{i}(ds\,dy\,dz) \\
		& = \frac{1}{n} \sum_{i,j=1}^n \int_{\Xmb_T} \one_{[0,\Gamma(y,X_i(s-),\nu_i(s-))]}(z) \frac{\gamma_{s-}^{ij,y}}{\Gamma(y,X_i(s-),\nu_i(s-))} \,\Ncompensated_{i}(ds\,dy\,dz).
	\end{align*}
	Therefore \eqref{eq:Jn1} holds with 
	\begin{align*}
		R^{n,1} & = \frac{1}{n} \sum_{j=1}^n \int_{\Xmb_T} \one_{[0,\Gamma(y,X_1(s-),\nu_1(s-))]}(z) \frac{\gamma_{s-}^{1j,y}}{\Gamma(y,X_1(s-),\nu_1(s-))} \,\Ncompensated_{1}(ds\,dy\,dz) \\
		& \quad + \frac{1}{n} \sum_{i=2}^n \int_{\Xmb_T} \one_{[0,\Gamma(y,X_i(s-),\nu_i(s-))]}(z) \frac{\gamma_{s-}^{i1,y}+\gamma_{s-}^{ii,y}}{\Gamma(y,X_i(s-),\nu_i(s-))} \,\Ncompensated_{i}(ds\,dy\,dz).
	\end{align*}
	Here for the first term,
	\begin{align*}
		& \Emb_{P^n} \left[ \left( \frac{1}{n} \sum_{j=1}^n \int_{\Xmb_T} \one_{[0,\Gamma(y,X_1(s-),\nu_1(s-))]}(z) \frac{\gamma_{s-}^{1j,y}}{\Gamma(y,X_1(s-),\nu_1(s-))} \,\Ncompensated_{1}(ds\,dy\,dz) \right)^2 \right] \\
		& = \Emb_{P^n} \int_{[0,T]\times\Zmb} \frac{\left(\frac{1}{n} \sum_{j=1}^n\gamma_{s}^{1j,y}\right)^2}{\Gamma(y,X_1(s),\nu_1(s))} \,ds\,\measurea(dy) \le \frac{\kappa_1}{n},
	\end{align*}
	since 
	$$\Emb_{P^n} \left[\frac{\gamma_{s}^{1j,y} \gamma_{s}^{1k,y}}{\Gamma(y,X_1(s),\nu_1(s))} \right] = 0, \quad \mbox{ whenever } j \ne k.$$
	For the second term, using the independence of $\Ncompensated_i$ clearly we have
	\begin{align*}
		& \Emb_{P^n} \left| \frac{1}{n} \sum_{i=2}^n \int_{\Xmb_T} \one_{[0,\Gamma(y,X_i(s-),\nu_i(s-))]}(z) \frac{\gamma_{s-}^{i1,y}+\gamma_{s-}^{ii,y}}{\Gamma(y,X_i(s-),\nu_i(s-))} \,\Ncompensated_{i}(ds\,dy\,dz) \right| \le \frac{\kappa_2}{\sqrt{n}}.
	\end{align*}
	Therefore $\Emb_{P^n} |R^{n,1}| \le \frac{\kappa_3}{\sqrt{n}} \to 0$ as $n \to \infty$, and this completes the proof.
\end{proof}

The most difficult part is to analyze the following term in $J^{n,1}$:
\begin{equation}
	\label{eq:Un}
	U^n := \frac{1}{n} \sum_{2 \le i < j \le n} u^n(i,j),
\end{equation}
where
\begin{align*}
	u^n(i,j) & := \int_{\Xmb_T} \one_{[0,\Gamma(y,X_i(s-),\nu_i(s-))]}(z) \frac{\gamma_{s-}^{ij,y} - \lan \gamma_{s-}^{ij,y}, \theta_{ij} \ran }{\Gamma(y,X_i(s-),\nu_i(s-))} \,\Ncompensated_{i}(ds\,dy\,dz) \\
	& \qquad + \int_{\Xmb_T} \one_{[0,\Gamma(y,X_j(s-),\nu_j(s-))]}(z) \frac{\gamma_{s-}^{ji,y} - \lan \gamma_{s-}^{ji,y}, \theta_{ji} \ran }{\Gamma(y,X_j(s-),\nu_j(s-))} \,\Ncompensated_{j}(ds\,dy\,dz).
\end{align*}

Recall $\Omega_v, \Omega_0, \alpha(\omega_0,\cdot), P_0$ introduced in Sections \ref{sec:canonical} and \ref{sec:integral_operators}.
We now use the results from Section \ref{sec:Dynkin} with $\Smb = \Omega_v$ and $\thetabar = \alpha(\omega_0,\cdot)$, $\omega_0 \in \Omega_0$ to get the asymptotics of $(U^n, J^{n,1}, J^{n,2}, J^{n,3})$.
For each $\omega_0 \in \Omega_0$, $k \ge 1$ and $g \in L_{sym}^2(\alpha(\omega_0,\cdot)^{\otimes k})$, the MWI $I_k^{\omega_0}(g)$ is defined as in Section \ref{sec:Dynkin}.
More precisely, let $\Amc^k$ be the collection of all measurable $g \colon \Omega_0 \times \Omega_v^k \to \Rmb$ such that
\begin{equation*}
	\int_{\Omega_v^k} |g(\omega_0,\omega_1,\dotsc,\omega_k)|^2 \, \alpha(\omega_0,d\omega_1)\dotsm\alpha(\omega_0,d\omega_k) < \infty, \quad P_0 \mbox{ a.e.\ } \omega_0
\end{equation*}
and for every permutation $\pi$ on $[k]$, $g(\omega_0,\omega_1,\dotsc,\omega_k)=g(\omega_0,\omega_{\pi(1)},\dotsc,\omega_{\pi(k)})$, $P_0 \otimes \alpha^{\otimes k}$ a.s.\, where $P_0 \otimes \alpha^{\otimes k} (d\omega_0,d\omega_1,\dotsc,d\omega_k) := P_0(d\omega_0) \prod_{i=1}^k \alpha(\omega_0,d\omega_i)$.
Then there is a measurable space $(\Omega^*,\Fmc^*)$ and a regular conditional probability distribution $\lambda^* \colon \Omega_0 \times \Fmc^* \to [0,1]$ such that on the probability space $(\Omega_0 \times \Omega^*, \Bmc(\Omega_0) \otimes \Fmc^*, P_0 \otimes \lambda^*)$, where
\begin{equation*}
	P_0 \otimes \lambda^* (A \times B) := \int_A \lambda^*(\omega_0,B) \, P_0(d\omega_0), \quad A \times B \in \Bmc(\Omega_0) \otimes \Fmc^*,
\end{equation*}
there is a collection of real-valued random variables $\{I_k(g) : g \in \Amc^k, k \ge 1\}$ satisfying
\begin{enumerate}[(a)]
\item 
	For all $g \in \Amc^1$ the conditional distribution of $I_1(g)$ given $\Gmc^* := \Bmc(\Omega_0) \otimes \{\emptyset, \Omega^*\}$ is Normal with mean $0$ and variance $\int_{\Omega_v} g^2(\omega_0,\omega_1) \, \alpha(\omega_0,d\omega_1)$;
\item 
	$I_k$ is (a.s.) linear map on $\Amc^k$;
\item 
	For $g \in \Amc^k$ of the form
	\begin{equation*}
		g(\omega_0,\omega_1,\dotsc,\omega_k) = \prod_{i=1}^k \gtil(\omega_0,\omega_i), \quad \mbox{s.t. } \int_{\Omega_v} \gtil^2(\omega_0,\omega_1) \, \alpha(\omega_0,d\omega_1) < \infty, \: P_0 \mbox{ a.e.\ } \omega_0,
	\end{equation*}
	we have
	\begin{align*}
		& I_k(g)(\omega_0,\omega^*) = \sum_{j=1}^{\lfl k/2 \rfl} (-1)^j C_{k,j} \left( \int_{\Omega_v} \gtil^2(\omega_0,\omega_1) \, \alpha(\omega_0,d\omega_1) \right)^j (I_1(\gtil)(\omega_0,\omega^*))^{k-2j}, \\
		& \quad P_0 \otimes \lambda^* \mbox{ a.e.\ } (\omega_0,\omega^*)
	\end{align*}
	and
	\begin{equation*}
		\int_{\Omega^*} (I_k(g)(\omega_0,\omega^*))^2 \, \lambda^*(\omega_0,d\omega^*) = k! \left( \int_{\Omega_v} \gtil^2(\omega_0,\omega_1) \, \alpha(\omega_0,d\omega_1) \right)^k, \: P_0 \mbox{ a.e.\ } \omega_0,
	\end{equation*}
	where $C_{k,j}$ are as in \eqref{eq:MWI_formula}.
\end{enumerate}
We write $I_k(g)(\omega_0,\cdot)$ as $I_k^{\omega_0}(g)$.

Recall $U^n$ from \eqref{eq:Un} and the canonical processes $V_i$ from Section \ref{sec:canonical}.
The following lemma is a key ingredient to get the asymptotics of $(U^n, J^{n,1}, J^{n,2}, J^{n,3})$.

\begin{Lemma}
	\label{lem:Un_joint}
	Let $\{ \phi_k \}_{k=1}^\infty$ be such that $\phi_k \in \Amc^k$ for each $k \ge 1$.
	Let 
	\begin{equation*}
		\hat{\Umc}^n_k (\phi_k) :=
		\begin{cases}
			\displaystyle \sum_{2 \le i_1 < i_2 < \dotsb < i_k \le n} \phi_k(V_{i_1},\dotsc,V_{i_k}) & \text{for } n \ge k \\
			0 & \text{for } n<k.
		\end{cases}
	\end{equation*}
	Then the following convergence holds as $n \to \infty$:
	\begin{equation*}
		\left( U^n, \left( n^{-\frac{k}{2}} \hat{\Umc}^n_k (\phi_k) \right)_{k \ge 1} \right) \Rightarrow \left( Z, \left( \frac{1}{k!} I_k^{X_1}(\phi_k) \right)_{k \ge 1} \right)
	\end{equation*}
	as a sequence of $\Rmb^\infty$-valued random variables, where $Z$ is Gaussian with mean $0$ variance 
	\begin{equation}
		\label{eq:Zvar}
		\sigma^2 := \Emb_{P^n} \int_{[0,T]\times\Zmb} \frac{(\gamma_{s-}^{ij,y} - \lan \gamma_{s-}^{ij,y}, \theta_{ij} \ran )^2}{\Gamma(y,X_i(s-),\nu_i(s-))} \,ds\,\measurea(dy)
	\end{equation}
	and is independent of $\{I_k^{X_1}(\cdot)\}_{k \ge 1}$.
\end{Lemma}

\begin{proof}
	Fix $m \in \Nmb$, $\phi_k \in \Amc^k$ and $t, s_k \in \Rmb$ for $k=1,\dotsc,m$.
	Denote by $\Emb_{{P^n},V}$ the conditional expectation under ${P^n}$ given $(X_i,N_i)_{i=1}^n$.
	Since $\xi_{ij}$ is conditionally independent given $(X_i,N_i)_{i=1}^n$, we have
	\begin{equation*}
		\sigma_n^2 := \Emb_{{P^n},V} [U^n]^2 = \frac{1}{n^2} \sum_{2 \le i < j \le n} \Emb_{{P^n},V} [u^n(i,j)]^2.
	\end{equation*}
	Note that
	\begin{equation*}
		\Emb_{P^n} [\sigma_n^2] = \frac{1}{n^2} \sum_{2 \le i < j \le n} \Emb_{P^n} [u^n(i,j)]^2 \to \Emb_{P^n} \int_{[0,T]\times\Zmb} \frac{(\gamma_{s-}^{ij,y} - \lan \gamma_{s-}^{ij,y}, \theta_{ij} \ran )^2}{\Gamma(y,X_i(s-),\nu_i(s-))} \,ds\,\measurea(dy) = \sigma^2
	\end{equation*}
	and (since the cross product term below is zero when $(i,j,\itil,\jtil)$ are distinct)
	\begin{align*}
		\Emb_{P^n} [\sigma_n^2 - \Emb_{P^n} [\sigma_n^2]]^2 = \frac{1}{n^4} \Emb_{P^n} \left[ \sum_{2 \le i < j \le n} \left( \Emb_{{P^n},V} [u^n(i,j)]^2 - \Emb_{P^n} [u^n(i,j)]^2 \right) \right]^2 \le \frac{\kappa_1}{n} \to 0.
	\end{align*}
	So $\sigma_n^2 \to \sigma^2$ in probability as $n \to \infty$.
	Suppose without loss of generality that $\sigma^2 > 0$, since otherwise we have that $Z=0$, $U^n \to 0$ in probability as $n \to \infty$ and the desired convergence holds trivially by Lemma \ref{thm:Dynkin}.
	Also note that
	\begin{align*}
		\Emb_{P^n} \sum_{2 \le i \ne j \le n} \Emb_{{P^n},V} \left| \frac{u^n(i,j)}{n} \right|^4 \le \frac{\kappa_2}{n^2} \to 0
	\end{align*}
	as $n \to \infty$.
	Hence the Lyapunov's condition for CLT (see \cite[Theorem 27.3]{Billingsley1995probability}) holds with $\delta=2$:
	\begin{equation*}
		\lim_{n \to \infty} \frac{1}{\sigma_n^{2+\delta}} \sum_{2 \le i \ne j \le n} \Emb_{{P^n},V} \left| \frac{u^n(i,j)}{n} \right|^{2+\delta} = 0,
	\end{equation*}
	where the convergence is in probability.
	It then follows from standard proofs of CLT and a subsequence argument that for each $t \in \Rmb$,
	\begin{equation*}
		\Emb_{{P^n},V} \left[ e^{\sqrt{-1}tU^n} \right] - e^{-t^2\sigma_n^2/2} \to 0
	\end{equation*}
	in probability as $n \to \infty$.
	This together with the convergence of $\sigma_n^2 \to \sigma^2$ implies that
	\begin{equation}
		\label{eq:Un_pf1}
		\Emb_{{P^n},V} \left[ e^{\sqrt{-1}tU^n} \right] \to e^{-t^2\sigma^2/2}
	\end{equation}
	in probability as $n \to \infty$.
	Now let $(t, s_1, \ldots s_m) \mapsto \varphi_n(t,s_1,\dotsc,s_m)$ be the characteristic function of $$(U^n,n^{-\frac{1}{2}} \hat{\Umc}^n_1 (\phi_1),\dotsc,n^{-\frac{m}{2}} \hat{\Umc}^n_m (\phi_m)),$$ and $$\varphi(t,s_1,\dotsc,s_m) := e^{-\half t^2\sigma^2} \psi(s_1,\dotsc,s_m),\;\; (t,s_1,\dotsc,s_m)\in \Rmb^{m+1}$$
	be that of $(Z,I_1^{X_1}(\phi_1),\dotsc,\frac{1}{m!} I_m^{X_1}(\phi_m))$.
	From Lemma \ref{thm:Dynkin} it follows that for all $(s_1,\dotsc,s_m)\in \Rmb^{m}$
	\begin{equation}
		\Emb_{P^n} [e^{\sqrt{-1} \sum_{k=1}^m s_k n^{-\frac{k}{2}} \hat{\Umc}^n_k (\phi_k)}] \to \psi(s_1,\dotsc,s_m) \mbox{ as } n \to \infty.
		\label{eq:Un_pf2}
	\end{equation}
	Thus as $n \to \infty$,
	\begin{align*}
		& \varphi_n(t,s_1,\dotsc,s_m) - \varphi(t,s_1,\dotsc,s_m) \\
		& = \Emb_{P^n} \left[ e^{\sqrt{-1}tU^n + \sqrt{-1} \sum_{k=1}^m s_k n^{-\frac{k}{2}} \hat{\Umc}^n_k (\phi_k)} - e^{-\half t^2\sigma^2} \psi(s_1,\dotsc,s_m) \right] \\
		& = \Emb_{P^n} \left[ \left( \Emb_{{P^n},V} \left[ e^{\sqrt{-1}tU^n}\right] - e^{-\half t^2\sigma^2} \right) e^{\sqrt{-1} \sum_{k=1}^m s_k n^{-\frac{k}{2}} \hat{\Umc}^n_k (\phi_k)} \right] \\
		& \quad + \left(\Emb_{P^n} [e^{\sqrt{-1} \sum_{k=1}^m s_k n^{-\frac{k}{2}} \hat{\Umc}^n_k (\phi_k)}] - \psi(s_1,\dotsc,s_m)\right) e^{-\half t^2\sigma^2} \\
		& \to 0,
	\end{align*}
	where the convergence follows from \eqref{eq:Un_pf1} and \eqref{eq:Un_pf2}.
	This completes the proof.
\end{proof}

Now we analyze the asymptotics of $(U^n, J^{n,1}, J^{n,2}, J^{n,3})$.
Recall $A_{\omega_0}, h, \Xi$ introduced in Section \ref{sec:integral_operators}.
For $\omega_0 \in \Omega_0$, denote by $A_{\omega_0}^*$ the adjoint operator of $A_{\omega_0}$, that is
\begin{equation*}
	A_{\omega_0}^* g(\omega_1) = \int_{\Omega_v} g(\omega_2) h(\omega_1,\omega_2) \, \alpha(\omega_0,d\omega_2), \quad g \in \Hmc_{\omega_0}, \omega_1 \in \Omega_v.
\end{equation*}

\begin{Lemma}
	\label{lem:A}
	For $P_0$ a.e.\ $\omega_0$,
	(a) Trace$(A_{\omega_0}A_{\omega_0}^*) = \int_{\Omega_v^2} h^2(\omega_1,\omega_2) \,\alpha(\omega_0,d\omega_1)\,\alpha(\omega_0,d\omega_2) = \int_{\Omega_v^2} h^2(\omega_1,\omega_2) \,\Xi(d\omega_1)\,\Xi(d\omega_2) = \int_{[0,T]\times\Zmb} \lambda(t,y) \,dt\,\measurea(dy)$,
	where
	\begin{equation*}
		\lambda(t,y) := \Emb_{P^n} \left[ \frac{\lan \gammabar_{t,y}(X_1[t],X_2(t),\cdot), \theta(t,X_1[t],X_2[t]) \ran^2 }{\Gamma(y,X_1(t),\nu_1(t))} \right], \quad (t,y) \in [0,T]\times\Zmb.
	\end{equation*}
	(b) Trace$(A_{\omega_0}^n)=0$ for all $n \ge 2$.
	(c) $I-A_{\omega_0}$ is invertible.
\end{Lemma}

\begin{proof}
	The first equality in part (a) follows from the definition of $A_{\omega_0}$, the second uses the observation that $h(\omega_1,\omega_2)$ does not depend on $\xi_*(\omega_1)$ or $\xi_*(\omega_2)$, and the third follows from the definition of $\lambda$ and $h$.	
	Part (b) follows on noting that
	\begin{equation*}
		\text{Trace}(A^n_{\omega_0}) = \int_{\Omega_v^n} h(\omega_1,\omega_2) h(\omega_2,\omega_3) \dotsm h(\omega_n,\omega_1) \,\alpha(\omega_0,d\omega_1)\,\alpha(\omega_0,d\omega_2)\dotsm\alpha(\omega_0,d\omega_n) = 0;
	\end{equation*}
	see also \cite[Lemma 2.7]{ShigaTanaka1985}.
	Part (c) is immediate from \cite[Lemma 1.3]{ShigaTanaka1985}.
\end{proof}

Let
\begin{equation*}
	m(y,x[s],\xtil[s]) := \Emb_{P^n} \left[ \frac{\gamma_{s}^{12,y}\gamma_{s}^{13,y}}{\Gamma(y,X_1(s),\nu_1(s))} \,\Big|\, X_2[s]=x[s],X_3[s]=\xtil[s] \right]
\end{equation*}
and define functions $\ell, F \colon \Omega_v \times \Omega_v \to \Rmb$ ($\Xi \times \Xi$ a.s.) is by
\begin{align}
	\ell(\omega_1,\omega_2) & := \int_{[0,T]\times\Zmb} m(y,X_*(\omega_1)[s],X_*(\omega_2)[s]) \,ds\,\measurea(dy), \label{eq:ell} \\
	F(\omega_1,\omega_2) & := h(\omega_1,\omega_2) + h(\omega_2,\omega_1) - \ell(\omega_1,\omega_2). \label{eq:F}
\end{align}
Note that
\begin{equation}
	\label{eq:h_and_ell}
	\ell(\omega_1,\omega_2) = \int_{\Omega_v} h(\omega_3,\omega_1) h(\omega_3,\omega_2) \, \Xi(d\omega_3), \quad P_0 \text{ a.s.\ } \omega_0.
\end{equation}
Recall the independent normal random variable $Z$ from Lemma \ref{lem:Un_joint}.
Let
\begin{equation}
	\label{eq:J}
	J := \frac{1}{2} I_2^{X_1}(F) - \frac{1}{2} \text{Trace}(A_{X_1}A_{X_1}^*) + Z - \frac{1}{2} \sigma^2.
\end{equation}
The following lemma is the key step.

\begin{Lemma}
	\label{lem:joint_cvg}
	As $n \to \infty$,
	\begin{align*}
		\sqrt{-1} \etabar^n(\varphi) + J^n(T) \Rightarrow \sqrt{-1} I_1^{X_1}(\varphi) + J.
	\end{align*}
\end{Lemma}

\begin{proof}
	Recall $\hat{\Umc}_k^n$ in Lemma \ref{lem:Un_joint}, $\hat{\Smc}^{n,2}$ in \eqref{eq:Shat}, $U^n$ in \eqref{eq:Un}, and $\ell$ in \eqref{eq:ell}.
	From Lemmas \ref{lem:Jn2} and \ref{lem:Jn1} we have
	\begin{align*}
		J^{n,1} & = U^n + \hat{\Umc}_2^n(h^{sym}) + R^{n,1}, \\
		J^{n,2} & = \int_{[0,T]\times\Zmb} \Emb_{P^n} \left[ \frac{\lan (\gamma_{s}^{12,y})^2, \nu_1(s) \ran}{\Gamma(y,X_1(s),\nu_1(s))} \right] ds\,\measurea(dy) + \hat{\Umc}_2^n(\ell) + R^{n,2},
	\end{align*}
	where $h^{sym}(\omega_1,\omega_2) := \half \left[ h(\omega_1,\omega_2) + h(\omega_2,\omega_1) \right]$ for $\omega_1,\omega_2 \in \Omega_v$.
	It then follows from Lemmas \ref{lem:Jn3}, \ref{lem:Jn2}, \ref{lem:Jn1} and \ref{lem:Un_joint} that
	\begin{align*}
		& \left( \etabar^n(\varphi), J^{n,1}, J^{n,2}, J^{n,3} \right) \\
		& \Rightarrow \left( I_1^{X_1}(\varphi), Z + I_2^{X_1}(h^{sym}), \int_{[0,T]\times\Zmb} \Emb_{P^n} \left[ \frac{\lan (\gamma_{s}^{12,y})^2, \nu_1(s) \ran}{\Gamma(y,X_1(s),\nu_1(s))} \right] ds\,\measurea(dy) + I_2^{X_1}(\ell), 0 \right).
	\end{align*}
	Also note that
	\begin{align*}
		& \Emb_{P^n} \left[ \frac{\lan (\gamma_{s}^{12,y})^2, \nu_1(s) \ran}{\Gamma(y,X_1(s),\nu_1(s))} \right] \\
		& = \Emb_{P^n} \left[ \frac{\Emb_{P^n} [ \lan (\gamma_{s}^{12,y})^2, \nu_1(s) \ran \,|\, X_1]}{\Gamma(y,X_1(s),\nu_1(s))} \right] \\
		& = \Emb_{P^n} \left[ \frac{\Emb_{P^n} [(\gamma_{s}^{ij,y} - \lan \gamma_{s}^{ij,y}, \theta_{ij} \ran )^2 \,|\, X_1] + \Emb_{P^n} [\lan \gammabar_{t,y}(X_1[t],X_2(t),\cdot), \theta(t,X_1[t],X_2[t]) \ran^2 \,|\, X_1] }{\Gamma(y,X_1(t),\nu_1(t))} \right] \\
		& = \sigma^2 + \text{Trace}(A_{X_1}A_{X_1}^*)
	\end{align*}
	a.s., by \eqref{eq:Zvar} and Lemma \ref{lem:A}.
	The result follows on combining above two displays with \eqref{eq:Jn}, \eqref{eq:F}, and \eqref{eq:J}.	
\end{proof}

\subsection{Completing the proof of Theorem \ref{thm:CLT_nu}}

Recall $\Gmc^* = \Bmc(\Omega_0) \otimes \{\emptyset, \Omega^*\}$.
It follows from \eqref{eq:h_and_ell}, Lemma \ref{lem:A} and \cite[Lemma 1.2]{ShigaTanaka1985} that $P_0$ a.s.
\begin{equation*}
	\Emb_{P_0 \otimes \lambda^*} \left[ \exp \left( \frac{1}{2} I_2^{X_1}(F) \right) \Big| \Gmc^* \right] = \exp \left( \frac{1}{2} \text{Trace}(A_{X_1}A_{X_1}^*) \right).
\end{equation*}
Therefore
\begin{equation*}
	\Emb_{P_0 \otimes \lambda^*} \left[ \exp \left( \frac{1}{2} I_2^{X_1}(F) - \frac{1}{2} \text{Trace}(A_{X_1}A_{X_1}^*) \right) \right] = 1.
\end{equation*}
It then follows from Lemma \ref{lem:Un_joint} that
\begin{equation*}
	\Emb_{P_0 \otimes \lambda^*} \left[ \exp \left( J \right) \right] = 1.
\end{equation*}
Also, recall that
\begin{equation*}
	\Emb_{P^n} \left[ \exp \left( J^n(T) \right) \right] = 1.
\end{equation*}
It then follows from Lemma \ref{lem:joint_cvg} (with $\varphi \equiv 0$) and Scheffe's lemma {(see e.g.\ \cite[Theorem 16.14]{Billingsley1995probability} after an application of the Skorokhod representation theorem)} that $\{\exp \left( J^n(T) \right)\}$ is uniformly integrable.
Since $|\exp \left( \sqrt{-1} \etabar^n(\varphi) \right)| = 1$, 
$$\{\exp \left( \sqrt{-1} \etabar^n(\varphi) + J^n(T) \right)\}$$
is also uniformly integrable.
{Using this and Lemma \ref{lem:joint_cvg} we have}
\begin{align}
	& \lim_{n \to \infty} \Emb_{P^n} \exp \left( \sqrt{-1} \etabar^n(\varphi) + J^n(T) \right) \notag \\
	& = \Emb_{P_0 \otimes \lambda^*} \exp \left( \sqrt{-1} I_1^{X_1}(\varphi) + J \right) \notag \\
	& = \Emb_{P_0 \otimes \lambda^*} \exp \left( \sqrt{-1} I_1^{X_1}(\varphi) + \frac{1}{2} I_2^{X_1}(F) - \frac{1}{2} \text{Trace}(A_{X_1}A_{X_1}^*) \right) \Emb_{P_0 \otimes \lambda^*} \exp \left( Z - \frac{1}{2} \sigma^2 \right) \notag \\
	& = \Emb_{P_0 \otimes \lambda^*} \left( \Emb_{P_0 \otimes \lambda^*} \left[ \exp \left( \sqrt{-1} I_1^{X_1}(\varphi) + \frac{1}{2} I_2^{X_1}(F) - \frac{1}{2} \text{Trace}(A_{X_1}A_{X_1}^*) \right) \Big| \Gmc^* \right] \right) \notag \\
	& = \int_{\Omega_0} \exp \left( -\half (\sigma^\varphi_{\omega_0})^2 \right) P_0(d\omega_0), \label{eq:pf-CLT-nu}
\end{align}
where the last line is a consequence of Lemma \ref{lem:A} and \cite[Lemma 1.3]{ShigaTanaka1985}.
Thus we have proved \eqref{eq:goal_CLT_nu} which completes the proof of Theorem \ref{thm:CLT_nu}.

\subsection{Completing the proof of Theorem \ref{thm:CLT_mu}}

Clearly the function $\varphi \in \Amc_x$ can be viewed as an element (abusing notation) $\varphi$ of $\Amc$ defined by $\varphi(x,\xi) := \varphi(x)$, $x,\xi \in \Dmb([0,T]:\Zmb)^2$, and $\Phi_{\omega_0} = \Phi$ for $P_0$ a.s.\ $\omega_0$. 
Note that $h(\omega_1,\omega_2)$ depends on $\omega_2$ only through $X_*(\omega_2)$.
It then follows from the definition of $A_{\omega_0}$ and $A$ that 
$(I-A)^{-1} \Phi(\omega)=(I-A_{\omega_0})^{-1}\Phi_{\omega_0}(\omega)$ for $P_0$ a.s.\ $\omega_0$ and the dependence on $\omega$ is actually only through $X_*(\omega)$.
It then follows from the definition of $\sigma_{\omega_0}^{\varphi}$ that
\begin{equation*}
	\sigma_{\omega_0}^{\varphi} = \|(I-A)^{-1} \Phi\|_{\Hmc}, \quad P_0 \mbox{ a.s.\ } \omega_0.
\end{equation*}
Therefore from \eqref{eq:pf-CLT-nu} we have
\begin{align*}
	\lim_{n \to \infty} \Emb_{P^n} \exp \left( \sqrt{-1} \etabar_x^n(\varphi) + J^n(T) \right) & = \int_{\Omega_0} \exp \left( -\half \|(I-A)^{-1} \Phi\|_{\Hmc}^2 \right) P_0(d\omega_0) \\
	& = \exp \left( -\half \|(I-A)^{-1} \Phi\|_{\Hmc}^2 \right).
\end{align*}
This gives \eqref{eq:goal_CLT_mu} and completes the proof of Theorem \ref{thm:CLT_mu}.

\appendix

\section{Proof of Theorem \ref{thm:iid_LLN}}
\label{sec:pf-iid}

\begin{proof}[Proof of Theorem \ref{thm:iid_LLN}]
	(a) Since the limiting system is McKean--Vlasov, the proof of existence and uniqueness is standard (cf.\ \cite[Chapter 1]{Sznitman1991}, see also \cite[Theorem 2.1]{Graham1992Mckean}) and hence omitted.

	(b) Now we show \eqref{eq:iid_LLN1}.
	For each fixed $i \in [n]$ and $t \in [0,T]$, we have
	\begin{align*}
		& \Emb \|X_i^n-X_i\|_{*,t} \\
		& \le \Emb \int_{\Xmb_t} |y|\left| \one_{[0,\Gamma(y,X_i^n(s-),\nu_i^n(s-))]}(z) - \one_{[0,\Gamma(y,X_i(s-),\nu(s-))]}(z) \right| N_i(ds\,dy\,dz) \\
		& = \Emb \int_{[0,t]\times\Zmb} |y| \left| \Gamma(y,X_i^n(s),\nu_i^n(s))-\Gamma(y,X_i(s),\nu(s)) \right| ds\,\measurea(dy). 
	\end{align*}
	Fixing $y \in \Zmb$ and $s \in [0,T]$, we have
	\begin{align*}
		& \Emb \left| \Gamma(y,X_i^n(s),\nu_i^n(s))-\Gamma(y,X_i(s),\nu(s)) \right| \\
		& = \Emb \left| \frac{1}{n} \sum_{j=1}^n \gamma(y,X_i^n(s),X_j^n(s),\xi_{ij}(s)) - \int_{\Zmb^2} \gamma(y,X_i(s),\xtil,\xitil) \,\nu(s)(d\xtil\,d\xitil) \right| \\
		& \le \Emb \left| \frac{1}{n} \sum_{j=1}^n \gamma(y,X_i^n(s),X_j^n(s),\xi_{ij}(s)) - \frac{1}{n} \sum_{j=1}^n \gamma(y,X_i(s),X_j(s),\xi_{ij}(s)) \right| \\
		& \qquad + \Emb \left| \frac{1}{n} \sum_{j=1}^n \gamma(y,X_i(s),X_j(s),\xi_{ij}(s)) - \int_{\Zmb^2} \gamma(y,X_i(s),\xtil,\xitil) \,\nu(s)(d\xtil\,d\xitil) \right|.
	\end{align*}
	From Condition \ref{cond:iid}, Remark \ref{rmk:Lipschitz}(b) and the exchangeability of $\{(X_j^n,X_j) : j \in [n]\}$ we have 
	\begin{align*}
		& \Emb \left| \frac{1}{n} \sum_{j=1}^n \gamma(y,X_i^n(s),X_j^n(s),\xi_{ij}(s)) - \frac{1}{n} \sum_{j=1}^n \gamma(y,X_i(s),X_j(s),\xi_{ij}(s)) \right| \\
		& \le \frac{1}{n} \sum_{j=1}^n \gamma_y \left( \Emb |X_i^n(s)-X_i(s)| + \Emb |X_j^n(s)-X_j(s)| \right) = 2\gamma_y \Emb |X_i^n(s)-X_i(s)|.
	\end{align*}
	Since $\{(X_j,\xi_{ij}) : j\in[n]\}$ are independent with common joint law $\nu$, using Lemma \ref{lem:iid-moment} we have
	\begin{align*}
		& \Emb \left| \frac{1}{n} \sum_{j=1}^n \gamma(y,X_i(s),X_j(s),\xi_{ij}(s)) - \int_{\Zmb^2} \gamma(y,X_i(s),\xtil,\xitil) \,\nu(s)(d\xtil\,d\xitil) \right| 
		\le \frac{\kappa_1\gamma_y}{\sqrt{n}}.
	\end{align*}
	Combining above four displays gives
	\begin{align*}
		\Emb \|X_i^n-X_i\|_{*,t} 
		& \le \int_{[0,t]\times\Zmb} |y| \left( 2\gamma_y \Emb |X_i^n(s)-X_i(s)| + \frac{\kappa_1\gamma_y}{\sqrt{n}} \right) ds\,\measurea(dy) \\
		& \le 2C_\gamma \int_0^t \Emb \|X_i^n-X_i\|_{*,s} \,ds + \frac{\kappa_1C_\gamma}{\sqrt{n}}.
	\end{align*}
	It then follows from Gronwall's inequality that
	\begin{equation*}
		\Emb \|X_i^n-X_i\|_{*,T} \le \frac{\kappa_2}{\sqrt{n}}
	\end{equation*}
	for some $\kappa_2 < \infty$.
	This gives \eqref{eq:iid_LLN1}.
	
	(c) Using (b), the independence of $\{X_i\}$ and a standard argument (see \cite[Chapter 1]{Sznitman1991}) one has \eqref{eq:iid_POC}.
	
	(d) Using (b) and a standard argument (see \cite[Chapter 1]{Sznitman1991} and \cite[Appendix A]{BhamidiBudhirajaWu2019weakly}) one has \eqref{eq:iid_LLN_nu_i} and \eqref{eq:iid_LLN_nu_i_t}. The last two statements \eqref{eq:iid_LLN_mu_i} and \eqref{eq:iid_LLN_mu_i_t} follow immediately from \eqref{eq:iid_LLN_nu_i} and \eqref{eq:iid_LLN_nu_i_t}, respectively.
\end{proof}

{\textbf{Acknowledgment:} We thank an anonymous referee for suggesting the alternative and more direct argument described in Remark \ref{rmk:alternative-proof}.}

\bibliographystyle{plain}


\end{document}